\documentclass[10pt]{article}
\usepackage[dvipdfmx]{graphicx}
\usepackage{amsmath}
\usepackage{amssymb}
\usepackage{amsthm}
\usepackage{bm}
\usepackage{url}
\usepackage{mathrsfs}
\newtheorem{dfn}{Definition}[section]
\newtheorem{thm}[dfn]{Theorem}

\newtheorem{lem}[dfn]{Lemma}

\newtheorem{rem}[dfn]{Remark}

\newtheorem{prob}[dfn]{Problem}

\newcommand{\C}{\mathbb{C}}

\newcommand{\Z}{\mathbb{Z}}
\newcommand{\N}{\mathbb{N}}

\newcommand{\rd}{\mathrm{d}}
\newcommand{\im}{\mathrm{i}}
\newcommand{\bphi}{\bm{\varphi}}
\newcommand{\bydef}{\,\stackrel{\mbox{\tiny\textnormal{\raisebox{0ex}[0ex][0ex]{def}}}}{=}\,}

\numberwithin{equation}{section}
\usepackage{mathtools}
\mathtoolsset{showonlyrefs=true}

\usepackage[top=1in, bottom=1in, left=1in, right=1in]{geometry}
\usepackage[usenames]{color}

\newcommand{\revise}[1]{{#1}}

\begin{document}
\title{
	Computer-assisted proofs for finding the monodromy of Picard--Fuchs differential equations for a family of K3 toric hypersurfaces
}
\author{%
	Toshimasa Ishige\thanks{%
		Department of Mathematics and Informatics, Graduate School of Science, Chiba University, Chiba, 263-8522, Japan
	}\and
	Akitoshi~Takayasu\thanks{%
		Institute of Systems and Information Engineering, University of Tsukuba, 1-1-1 Tennodai, Tsukuba, Ibaraki 305-8573, Japan (\texttt{takitoshi@risk.tsukuba.ac.jp})
	}
}
\date{\today}
\maketitle
\begin{center}
	\emph{In honor of Piotr Zgliczy\'nski’s contributions on his 60th birthday.}
\end{center}
\begin{abstract}

In this paper, we present a numerical method for rigorously finding the monodromy of linear differential equations. Beginning at a base point where certain particular solutions are explicitly given by series expansions, we first compute the value of fundamental system of solutions using interval arithmetic to rigorously control truncation and rounding errors. The solutions are then analytically continued along a prescribed contour encircling the singular points of the differential equation via a rigorous integrator. From these computations, the monodromy matrices are derived, generating the monodromy group of the differential equation. This method establishes a mathematically rigorous framework for addressing the monodromy problem in differential equations. For a notable example, we apply our computer-assisted proof method to resolve the monodromy problem for a Picard--Fuchs differential equation associated with a family of K3 toric hypersurfaces.

\end{abstract}
\begin{center}
	{\bf \small Keywords.} 
	{ \small Monodromy matrices, Picard--Fuchs differential equations, Computer-assisted proofs,\\ Interval arithmetic, K3 toric hypersurfaces}
\end{center}
\section{Introduction}\label{sec:intro}


Linear differential equations in the complex domain may admit multivalued solutions. Such multivaluedness is characterized by the monodromy, which reflects the topological properties, namely the fundamental group of the domain of definition. The monodromy group describes the transformation of solutions of a differential equation under analytic continuation around the singular locus, providing a detailed characterization of the global structure and behavior of the solutions. It captures how solutions evolve going around the singular locus, offering a comprehensive view of their behavior. Furthermore, studying the monodromy group provides insights into the global solution space and its connection to underlying geometric structures.
The monodromy of linear differential equations has been a subject of mathematical exploration since the late nineteenth century (cf., e.g., \cite{Gray2008,Iwasaki1991}) and has found recent applications in mathematical physics \cite{Castro_2013}. For instance, the monodromy has been used to describe particle production in charged black holes \cite{Chen2023}.

This paper presents a general framework for computer-assisted proofs addressing monodromy problems for differential equations. Our approach is specifically applied to determine the monodromy of Picard--Fuchs differential equations associated with a family of K3 toric hypersurfaces. 

We begin by introducing the concept of monodromy for linear differential equations. More precisely, we confine our discussion to linear ordinary differential equations (ODEs) or systems of linear partial differential equations (PDEs) that can be transformed to a \emph{Pfaffian equation} of the form:
\begin{equation}\label{eq:Pfaffian-equation}
\rd {u}=\left(\sum_{k=1}^m A^k(x) \rd x_k\right) {u},
\end{equation}
where ${x}=(x_1,\dots,x_m)^\mathrm{T} \in \mathbb{C}^m$ denotes the vector of variables,  ${u}=(u_1,\dots,u_n)^\mathrm{T}\in \mathbb{C}^n$ the vector of unknown variables, and $A^k({x})=\left(a^k_{i,j}({x})\right) (k=1,\dots,m)$ coefficient matrices, with $a^k_{i,j}({x})$ assumed to be meromorphic functions on $\mathbb{C}^m$.

From the standard text book of differential equations \revise{(cf.\ e.g., \cite[Chapter 11]{Haraoka2020})}, it is well-known that the equation \eqref{eq:Pfaffian-equation} is \emph{integrable} if $\partial A^l/\partial x_k-\partial A^k/\partial x_l=A^k A^l-A^lA^k$, ($k,l=1,\dots,m$).
\revise{In the case $m=1$ in \eqref{eq:Pfaffian-equation} corresponding to ODEs, equation \eqref{eq:Pfaffian-equation} is integrable.} In what follows, we assume that \eqref{eq:Pfaffian-equation} is integrable.
A point ${x}\in \mathbb{C}^m$ is called a \emph{regular point} if all $a^k_{i,j}({x})$ are holomorphic in a neighborhood of ${x}$. Let $D$ denote the set of all regular points. 
If ${x}^0=(x^0_1,\dots,x^0_m)$ lies in~$D$, then for any ${u}^0\in \mathbb{C}^n$,  there exists a unique solution to \eqref{eq:Pfaffian-equation} that is holomorphic in a neighborhood of~${x}^0$ and satisfying ${u}({x}^0)={u}^0$.
Thus, the  set of  all solutions to \eqref{eq:Pfaffian-equation} forms an $n$-dimensional linear space, the basis of which is referred to as a \emph{fundamental system of solutions}.

A solution that is holomorphic in a neighborhood of a regular point  can be analytically continued along any curve in $D$. Consequently, the set $\mathcal{S}=\mathbb{C}^m\setminus D$, where at least one of \revise{$A^k({x})$} ($k=1,\dots,m$) has a pole, is referred to as the \emph{singular locus},  and any point in the singular locus is called a \emph{singular point}.
 
Fix $b\in D$ as a \emph{base point}. We denote by $\pi_{1}\left(D,b\right)$ the \emph{fundamental group} of $D$ with the base point~$b$. Let $[\Sigma]$ be a homotopy class of loops in $D$ under homotopy equivalence, starting and ending at $b\in D$.
The product of loops $[\Sigma_i][\Sigma_j]$ is defined as the loop going along $[\Sigma_i]$ first, followed by $[\Sigma_j]$.
Let $\Phi(x)$ be an $n\times n$ matrix whose columns form a fundamental system of solutions in a neighborhood of $b$. 
Since the initial value matrix $\Phi(b)$ uniquely determines $\Phi(x)$, we also refer to $\Phi(b)$ as  the fundamental system of solutions for the sake of simplicity. 

Let $\Sigma_\ast$ denote the analytic continuation along the loop $[\Sigma]\in\pi_{1}\left(D,b\right)$. Then $(\Sigma_\ast \Phi)(x)$ is also a fundamental system of solutions defined on a neighborhood of $b$. We abuse the notation using
$\Sigma_\ast \Phi(b)$ instead of $(\Sigma_\ast \Phi)(b)$. 
Then, there exists a nonsingular matrix $M_\Sigma\in \mathrm{GL}_n(\C)$ such that
\begin{equation}\label{eq:MonodromyDef}
\Sigma_{\ast}\Phi(b) = \Phi(b) M_\Sigma.
\end{equation}
This matrix $M_\Sigma$ is called the \emph{monodromy matrix}.
The monodromy matrix implies the existence of a map
\[
\rho:\pi_1\left(D,b\right)\to \mathrm{GL}_n(\C),\quad [\Sigma]\mapsto M_\Sigma.
\]
In particular, since $e_\ast\Phi(b)=\Phi(b)$, where $[e]$ denotes the identity element in $\pi_{1}\left(D,b\right)$, and $(\alpha\beta)_\ast \Phi(b)=\beta_\ast(\alpha_\ast\Phi(b))$ for any $[\alpha]$, $[\beta]\in \pi_{1}\left(D,b\right)$, it follows that
\begin{align}
	\rho([e])=M_e=\mathrm{Id},\quad \rho([\alpha][\beta])=M_{\alpha\beta}=M_\alpha M_\beta=\rho([\alpha])\rho([\beta]).
\end{align}
Thus, the map $\rho$ is a group homomorphism and  is referred to as the \emph{monodromy representation} of the fundamental group.
The image of the map $\rho$, which forms a subgroup of $\mathrm{GL}_n(\C)$, is called \emph{monodromy group} of the differential equation. 
Under the conjugate relation\footnote{%
	\mbox{Two subgroups $G_1$, $G_2$ of $GL_n(\C)$ are conjugate $\Leftrightarrow \exists C\in GL_n(\C)$ s.t.~ $G_2 = C^{-1}G_1C$.}%
}, any monodromy representation belongs to the same conjugacy class, which is \revise{uniquely determined by the differential equation}.
\revise{We call this conjugacy class the \emph{monodromy} of the differential equation}.

{Obtaining the monodromy group is essential for understanding the global behavior of solutions, as it characterizes how solutions transform under analytic continuation along different loops. This is known as the monodromy problem, which is defined as follows:}

\begin{prob}[{\cite[Problem~4.1.1 in Chapter~2]{Iwasaki1991}}]
For a given linear differential equation, find an explicit expression of its monodromy; or find generators of the 
monodromy group with respect to a fundamental system of solutions.
\end{prob}
%
%
\revise{Solving the monodromy problem} is difficult and \revise{no general analytical method is known}. We quote a sentence just after Remark 4.1.3 in Chapter~2 of \cite{Iwasaki1991}, \revise{which motivates the present study}.
\begin{quote} \it
Unfortunately, we know only a very restricted number of equations whose monodromy problem we can solve. To each of such equations, one applies an appropriate method, according to the property of the equation, which stems from a method used for the hypergeometric differential equation. 
\end{quote}
As an example of linear differential equations, let us consider the monodromy problem of the Picard--Fuchs differential equation for a family of toric K3 hypersurfaces\revise{\footnote{%
\revise{There are 4319 classes of 3-dimensional reflexive polytopes \cite{KreuzerSkarke1998} each giving rise to a family of toric K3 hypersurfaces along with its Picard--Fuchs differential equation and the solutions obtained by Gel'fand--Kapronov--Zelevinsky system \cite{GelfandEtAl1988,GELFAND1990255}.}}}, which has been analytically solved by one of the authors using specific properties of K3 surfaces \cite{bib:ishige}.
However, the techniques developed in \cite{bib:ishige} are not applicable to the Picard--Fuchs differential equations for families of toric Calabi--Yau hypersurfaces of dimension three\revise{\footnote{%
\revise{There are 473800776 classes of 4-dimensional reflexive polytopes \cite{KreuzerSkarke2000} each giving rise to a family of toric toric Calabi-Yau threefold hypersurfaces along with its Picard--Fuchs differential equation and the solutions obtained by Gel'fand--Kapronov-Zelevinsky system \cite{GelfandEtAl1988,GELFAND1990255}.}}~(cf.\ \cite{CoxKatz1999}, e.g.)}%
, which have attracted significant attention in relation to superstring theory. Consequently, the methodology for finding the monodromy is highly specialized and its development poses significant challenges.

On the contrary, in general, the monodromy matrices can be computed by conventional numerical methods; 
Let the contour $\Sigma$ be  expressed by a smooth map \revise{$[0,1]\ni t\mapsto x(t) \in D$} 
with \revise{ $x(0)=x(1)=b$}. For \eqref{eq:Pfaffian-equation}, define $P(t)\bydef\sum_{k=1}^m A^k(x)\,\rd x_k/\rd t$. Then, one can numerically compute the left side of \eqref{eq:MonodromyDef}  using the equations $\Sigma_\ast \Phi(b)=(\Sigma_* \mathrm{Id}) \Phi(b)$ and  
\begin{equation}\label{eq:Pexp}
	\Sigma_\ast \mathrm{Id}=\lim_{N \to \infty} \left(\mathrm{Id}+P(t_{N-1}) \Delta t \right)\dots \left(\mathrm{Id}+P(t_0) \Delta t \right)
\end{equation}
with $\Delta t=1/N$, $t_k=k/N$, where the right hand side of \eqref{eq:Pexp} is called the \emph{path ordered exponential} denoted by $\mathrm{Pexp}\left(\int_\Sigma\sum_{k=1}^m A^k(x) d x_k\right)=\mathrm{Pexp}(\int_0^1 P(t) dt)$.
This is an analog of  the Euler method commonly used for solving ODEs. 
However, the major drawback lies in the accumulated margin of error, though ``estimated'' to be of the order of $\Delta t$, becomes meaninglessly large when an explicit bound is sought.
Using a more sophisticated conventional technique, such as the fourth order Runge--Kutta method, does not alleviate this issue. 
This is another reason motivating the development of computer-assisted proof techniques for computing monodromy matrices. We believe that establishing such methods will contribute to resolving open monodromy problems for differential equations in the complex domain.
\revise{
It is also worth noting that there are several works that numerically approximating monodromy groups, such as \cite{Hofmann2013,ChenYangYuiErdenberger2008,vanenckevort2005monodromycalculatonsfourthorder}. In particular, in \cite{vanenckevort2005monodromycalculatonsfourthorder}, a Picard--Fuchs differential equation associated with  a Calabi--Yau manifold is numerically integrated along a closed loop around a singular point to identify the corresponding monodromy matrix though numerical errors are not rigorously controlled in that approach.
}

\subsection{General approach for finding the monodromy: computer-assisted proofs}\label{sec:general_approach}

Here we introduce a basic concept of our general approach for rigorously computing the monodromy matrices~$M_{\Sigma}$. 
This approach relies on the linearity of the solution space of differential equations, reducing the problem to explicitly constructing a fundamental system of solutions at the base point $p$ and determining its transformation under analytic continuation along a loop.

More precisely, our method involves four main steps.
First, using the validated numerics, we compute the value of $\Phi(x)$ at the base point $b$. 
Since the solution space of a linear differential equation can be represented entirely by its values, it is sufficient to compute the value \( \Phi(b) \) instead of handling the fundamental system of solutions $\Phi(x)$. This matrix compactly encodes all the necessary information about the solution space at~\( b \). Using techniques such as Taylor expansion or a certain series expansion with rational functions, the values of the fundamental system of solutions can be rigorously computed. 
\revise{This step utilizes interval arithmetic} (cf., e.g., \cite{MR2652784}), which ensures rigorous error bounds in numerical computations by representing each number as an interval containing all possible values due to rounding or approximation errors. 
Second, define a loop $\Sigma$ in the regular domain $D$, which encircles a singular point, say $p$, and starts and ends at the base point $b$.
This is crucial because \revise{the monodromy matrix associated with \( \Sigma \) reflects a transformation} of solutions on analytic continuation along the loop.
Third, taking into account the vector field associated with the differential equation, we rigorously compute the analytic continuation of the identity matrix $\mathrm{Id}$ along $\Sigma$ using a rigorous integrator of ODEs.
Fourth, we calculate the monodromy matrix $M_{\Sigma}$ using the conjugacy formula
\begin{align}\label{eq:MonodromyMat}
	M_{\Sigma_p}=\Phi(b)^{-1}\left((\Sigma_p)_\ast \mathrm{Id}\right)\Phi(b),
\end{align}
which explicitly encodes the transformation of the fundamental system of solutions after the analytic continuation along the loop~$\Sigma$.

The general method introduced in this paper belongs to the category of computer-assisted proofs (CAPs) via validated numerics. 
In particular, the rigorous computation of the fundamental system of solutions and the implementation of rigorous analytic continuation are central to this approach.
\begin{rem}[Software supporting interval arithmetic and rigorous integrator of ODEs]
	Over the past four decades, various software libraries and tools have been developed to support interval arithmetic, facilitating the implementation of validated numerics. Examples include INTLAB \cite{Ru99a} (a MATLAB toolbox), Arb \cite{Johansson2017} (a C library for arbitrary precision interval arithmetic), IntervalArithmetic.jl \cite{IntervalArithmetic.jl} (a Julia package for interval arithmetic),  kv \cite{kv} and CAPD \cite{KAPELA2021105578} (C++ libraries for validated numerics), all of which are widely used worldwide. These tools rigorously implement basic arithmetic operations and various mathematical functions with guaranteed error bounds, making them essential for applications in computer-assisted proofs and reliable numerical computations. Our first step in computing the fundamental system of solutions relies on these tools.	
	
	Furthermore, in the third step of our approach, rigorous analytic continuation of differential equations is performed using a rigorous integrator of ODEs, which computes solution trajectories of ODEs with guaranteed error bounds via interval arithmetic. One of the most notable achievements of rigorous integration of ODEs is Tucker's resolution of Smale's 14th problem \cite{MR1870856}. 
	For a variety of methods in rigorous integration of ODEs, we refer to Berz and Makino \cite{MR1962787}, B\"unger \cite{Bunger:2020aa}, Immler \cite{Immler:2018aa}, Kashiwagi and Oishi \cite{Kashi1}, Lessard and Reinhardt \cite{MR3148084}, Lohner \cite{Lohner}, and Zgliczy\'{n}ski \cite{MR1930946}, among others. These methods are based on fixed-point arguments, which are equivalent to demonstrating the existence of solution trajectories, and employ various interval arithmetic techniques to ensure rigorous error bounds.

	In our work, we use the kv library \cite{kv} for rigorous integration of ODEs. This integrator is based on an interval representation of the solutions' Taylor series and incorporates Affine arithmetic \cite{Rump2015}, a technique designed to mitigate the so-called wrapping effect commonly encountered in interval analysis.	
\end{rem}

\revise{
\begin{rem} [Monodromy without explicit fundamental system of solutions]
While it may be a concern that the monodromy problem cannot be solved without explicitly specifying a fundamental system of solutions, we demonstrate that this concern can be resolved.
%
In a simply connected neighborhood of the base point $b$, a given initial value $C\in \mathrm{GL}_n(\mathbb{C})$ at $b$ well defines a solution, denoted by $\Phi(C;x)$.
The monodromy matrix of $\Phi(C;x)$ along a loop $\Sigma$ can then be obtained as $C^{-1} (\Sigma_\ast \mathrm{Id}) C$, by computing the path-ordered exponential $\mathrm{Pexp}\left(\int_\Sigma\sum_{k=1}^m A^k(x) d x_k\right)$, as defined in \eqref{eq:Pexp}.
For example, the Heun equations are a candidate for application, since explicit fundamental solutions are generally unavailable.
In fact, \cite{GiscardTamar2022} employs the integral series representation of path-ordered exponentials to approximate solutions.
It is expected that the use of interval arithmetic would allow for rigorous computation of both solutions and monodromy matrices for such equations.

Furthermore, it would be worth clarifying the geometric reason why  $\Phi(C;x)$ is well defined in a simply connected neighborhood of $b$, and for that purpose, we place equation \eqref{eq:Pfaffian-equation} in the framework of fiber bundles (cf.\ \cite[Sections 9.4.1, 10.1, 10.2, and 10.3]{Nakahara2021}). 
Consider the trivial principal bundle $P(D,G)=D \times G \stackrel{\pi}{\to} D$, where $G=\mathrm{GL}_n(\mathbb{C})$, and 1-form $\mathcal{A}=-\sum_{k=1}^m A^k dx_k$ defines the connection. In fact, $P(D,G)$ turns out to be a flat $G$-bundle because the curvature 2-form $\mathcal{F}=d \mathcal{A} + \mathcal{A} \wedge \mathcal{A}$ vanishes everywhere on $D$, due to the integrability condition.
In this case, it is shown \cite[Section 6.6 (c)]{Morita2001} that the holonomy along the loop $\Sigma$ does not change by a deformation of $\Sigma$, and therefore, the holonomy homomorphism $\rho_0: \pi_1(D,b) \to G$ is well defined by $\rho_0([\Sigma])= P\exp{ \int_ \Sigma (-\mathcal{A})}$ which is the holonomy along $\Sigma$. In particular, the restricted holonomy group becomes trivial, and as a result, as well as  $\Phi(\mathrm{Id};x)= P\exp{ \int_ \gamma (-\mathcal{A})}$,  $\Phi(C;x)=\Phi(\mathrm{Id};x)C$ is consequently well defined irrespective of the choice of the curve $\gamma$ connecting $b$ to $x$  in a simply connected neighborhood of $b$.

\end{rem}
}

In the followings, we mainly focus on how to solve the monodromy problem using the validated numerics for the case of the Picard--Fuchs differential equation for a family of K3 surfaces, which is a good example of complicated monodromy problems.
In particular, such differential equation has monodromy matrices with integer entries if we adopt an appropriate fundamental system of solutions, which will be discussed in the next section.
This fact yields that, to complete the proof of finding the monodromy matrix, it is sufficient to prove the target margin of e rror is less than 0.5.
However, it is worth mentioning \revise{that the framework provided in this paper} has a potential of computer-assisted proofs in addressing general monodromy problems.

\subsection{Picard--Fuchs differential equation for a family of K3 toric hypersurfaces}\label{Sct:K3}

In this paper, we use validated numerics to compute the monodromy matrices of the Picard--Fuchs differential equation, a system of linear differential equations associated with a two-parameter family of K3 surfaces \revise{arising from a reflexive polytope.
As for preceding works, Nagano \cite{NAGANO2012} derived the Picard--Fuchs differential equations for  two-parameter families of K3 surfaces arising from several reflexive polytopes, although fundamental solutions and monodromy were not obtained in his work.
Narumiya and Shiga \cite{NarumiyaShiga2001} by deformation of 2-cycles derived the integer monodromy matrices of the Picard--Fuchs differential equation for a one-parameter family of  K3 surfaces arising from a reflexive polytope, but they did not provide the corresponding fundamental solutions. In \cite{bib:ishige}, one of the authors solved the monodromy problem completely for the Picard--Fuchs differential equation associated with K3 toric hypersurfaces.
The present paper can be regarded as providing a computer-assisted proof of the results obtained in \cite{bib:ishige}, where the monodromy problem was addressed without relying on validated numerics.}%

A K3 surface is by definition a simply connected compact complex surface with a unique non-vanishing 2-form. It is named after the three mathematicians, K\"{a}hler, Kummer and Kodaira, and with its rich structure, has been studied vigorously not only in mathematics but also in physics in connection with the string theory. \revise{We refer the reader to \cite{Aspinwall1997} for an introduction to K3 surfaces and their applications to string theory.} Since the K3 surfaces are of complex dimension 2, they are of real dimension 4. 
It is well-known in standard texts on algebraic geometry and topology that all the loops which are 1-cycles on a K3 surface $Y$ can continuously shrink to a point, since $Y$ is simply connected. Thus, the first homology group of the K3 surface is trivial.

On the other hand, the second homology group is generated by 22 homology equivalence classes of  2-cycles which are by definition formal linear $\mathbb{Z}$-sums of closed oriented manifolds of real dimension~2 without boundary.
Two 2-cycles, say $\Gamma$ and $\Gamma^\prime$, are homologous if  $\Gamma-\Gamma^\prime$ is equal to the boundary $\partial M$ of an oriented submanifold $M$ of real dimension 3 in $Y$, where $-\Gamma^\prime$ is obtained by reversing the orientation of $\Gamma^\prime$. A homology equivalence class denoted by $[\Gamma]$ is the set of  2-cycles which are homologous to $\Gamma$. Let $\omega$ be the unique non-vanishing holomorphic 2-form  of  $Y$.  We have $\rd \omega =0$. Due to the Stokes theorem, 
\[  
\int_\Gamma \omega - \int_{\Gamma^\prime}  \omega =\int_{\Gamma-\Gamma^\prime} \omega =\int_{\partial M} \omega
=\int _M \rd \omega=0. 
\]
Thus, $\int_{[\Gamma]} \omega$ is well defined by $\int_\Gamma \omega$.

Let the intersection number of any 2-cycles, $\Gamma$ and $\Gamma^\prime$,  be denoted by $\Gamma \cdot \Gamma^\prime$. This intersection number  is an integer, independent of the choice of 2-cycles within their respective homology equivalence classes,
and $[\Gamma]\cdot[\Gamma^\prime]$ is well defined by  $\Gamma \cdot \Gamma^\prime$.
Therefore, we use the same notation $\Gamma$  to represent both the 2-cycle and its corresponding homology class for simplicity.

For 2-cycles we have  $\Gamma \cdot \Gamma^\prime=\Gamma^\prime \cdot \Gamma$. The self-intersection number $\Gamma \cdot \Gamma$ is determined by deforming $\Gamma$ to another 2-cycle homologous to $\Gamma$ and counting the intersection points with $\Gamma$.

Consider a two-parameter family of a lattice polarized K3 surfaces $Y_{x,y}$ defined by
\begin{align}
\left\{(X, Y, Z) \in \mathbb{C}^{3} : f(X, Y, Z ; x, y)=0\right\},
\end{align}
where
\[
\begin{aligned}
	f(X, Y, Z ; x, y) \bydef X Y Z(X+Y+Z+1)+\lambda(x, y) X Y+\mu(x, y),\quad
	\lambda(x, y) \bydef x / y+1 / 4, \quad \mu(x, y)=x^{3} / y^{2}.
\end{aligned}
\]
The unique non-vanishing holomorphic 2-forms is given by
\[
\omega_{x, y}=\frac{\rd X \wedge \rd Y}{f_{Z}(X, Y, Z ; x, y)}=\frac{\rd Y \wedge \rd Z}{f_{X}(X, Y, Z ; x, y)}=\frac{\rd Z \wedge \rd X}{f_{Y}(X, Y, Z ; x, y)}.
\]


Let $\mathrm{H}_2(Y_{x,y},\mathbb{Z})$ denote the second homology group of the K3 surface.
For $Y_{x,y}$, we can choose $Z$-linearly independent $\Gamma_1,\dots,\Gamma_{22}$ in 
$\mathrm{H}_2(Y_{x,y},\mathbb{Z})$ such that
\begin{equation*}
	\int_{\Gamma_i} \omega_{x,y}
	\begin{cases}
		&\neq 0 \quad (i=1,\dots,4)\\
		&=0 \quad (i=5,\dots,22)
	\end{cases}
\end{equation*}
and 
\begin{equation*}
	N\bydef(\Gamma_i \cdot \Gamma_j )_{i,j=1,\dots,4}=
	\begin{pmatrix} 0 & 1 & 0 & 0 \\ 1& 0 & 0 & 0 \\ 0 & 0 & -2 & 0 \\ 0 & 0 & 0 & 4 \\ \end{pmatrix}, \quad
	\Gamma_i \cdot \Gamma_j=0 \quad \text{if} \quad 1 \le i \le 4 <  j  \le 22.
\end{equation*}
Equipped with the intersection, the second homology group $\mathrm{H}_2(Y_{x,y},\mathbb{Z})$ becomes a lattice. The sublattice generated by $\Gamma_1,\Gamma_2,\Gamma_3,\Gamma_4$ is called
the transcendental lattice also denoted by $N$.

Additionally, the K3 surface $Y_{x,y}$ is characterized by the periods, which is given by
\begin{equation}\label{eq:IntegrationRepresentation}
	\begin{split}
		(\varphi_1,\varphi_2, \varphi_3, \varphi_4)&=(\int_{\Gamma_1} \omega_{x,y}, \int_{\Gamma_2} \omega_{x,y},\int_{\Gamma_3} \omega_{x,y},  \int_{\Gamma_4} \omega_{x,y}) N^{-1}                   \\
		&=(\int_{\Gamma_2} \omega_{x,y}, \int_{\Gamma_1} \omega_{x,y},-\frac{1}{2}\int_{\Gamma_2} \omega_{x,y},  \frac{1}{4}\int_{\Gamma_4} \omega_{x,y}).
	\end{split}
\end{equation}
The differential equation is called the Picard--Fuchs differential equation for the family of lattice polarized K3 surfaces if their periods constitute a basis of  the solution space. Then, the Picard--Fuchs differential equation for  the family of $Y_{x,y}$ is given by  
\begin{equation}\label{eq:PicardFuchs}
	\begin{cases}
		\varphi_{xx} = \ell\varphi_{xy} + a\varphi_{x} + b \varphi_{y} + p\varphi\\
		\varphi_{yy} = m\varphi_{xy} + c\varphi_{x} + d\varphi_{y} + q\varphi\,,
	\end{cases}
\end{equation}
where $\varphi:\C^2\to \C$ is an unknown function and variable coefficients are defined by
\begin{equation}\label{eq:var_coeffs}
	\begin{aligned}
		h(x,y)&\bydef 1+20x+9y\\
		\ell (x,y)&\bydef -(8x+32x^2+4y+84xy+27y^2)/(2xh)\\
		a(x,y)&\bydef (4x+16x^2-3y-60xy-27y^2)/(2xyh)\\
		b(x,y)&\bydef -(16x+96x^2+4y+168xy+27y^2)/(4x^2h)\\
		p(x,y)&\bydef (2+12x+9y)/(xyh)\\
		m(x,y)&\bydef -(8x+32x^2+y+24xy)/(4yh)\\
		c(x,y)&\bydef x(1+4x)/(y^2h)\\
		d(x,y)&\bydef -(12x+16x^2+y+72xy)/(8xyh)\\
		q(x,y)&\bydef (1-8x)/(2y^2h).
	\end{aligned}
\end{equation}

In other words, $(\varphi_1,\varphi_2,\varphi_3,\varphi_4)$ in \eqref{eq:IntegrationRepresentation} constitute a basis of the solution space of the equation \eqref{eq:PicardFuchs}, which we call a fundamental system of solutions.
The monodromy  arises from the deformation of these 2-cycles, which serve as the domains of integration in \eqref{eq:IntegrationRepresentation}. In particular, with analytic continuation along a loop  $\Sigma$ based at $p_0=(x_0,y_0)$,  the 2-cycles $(\Gamma_1, \Gamma_2, \Gamma_3, \Gamma_4)$ are deformed to $(\Gamma_{1}^{\prime}, \Gamma_{2}^{\prime}, \Gamma_{3}^{\prime}, \Gamma_{4}^{\prime})$ in $Y_{x_0,y_0}$, resulting in a corresponding change in the periods from $(\varphi_1, \varphi_2,\varphi_3, \varphi_4)$ to $(\varphi_1^\prime,\varphi_2^\prime,\varphi_3^\prime , \varphi_4^\prime)$.  
Since $(\Gamma_1, \Gamma_2, \Gamma_3, \Gamma_4)$ is  a basis of the transcendental lattice $N$, there exists $P \in \mathrm{GL}_4(\Z)$ such that 
\[
(\Gamma_{1}^{\prime}, \Gamma_{2}^{\prime}, \Gamma_{3}^{\prime}, \Gamma_{4}^{\prime})=(\Gamma_1, \Gamma_2, \Gamma_3, \Gamma_4) P, \quad (\Gamma_1, \Gamma_2, \Gamma_3, \Gamma_4)=(\Gamma_{1}^{\prime}, \Gamma_{2}^{\prime}, \Gamma_{3}^{\prime}, \Gamma_{4}^{\prime}) P^{-1}.
\]
Note that $(\Gamma_{1}^{\prime}, \Gamma_{2}^{\prime}, \Gamma_{3}^{\prime}, \Gamma_{4}^{\prime})$ also forms a basis of $N$, and  $P^{-1}$ belongs to $\mathrm{GL}_4(\mathbb{Z})$. Since the intersection matrix $N$ remains invariant under the analytic continuation, we obtain the following relation:
\begin{equation}\label{eq:Isometry}
	(\Gamma_{1}^{\prime}, \Gamma_{2}^{\prime}, \Gamma_{3}^{\prime}, \Gamma_{4}^{\prime})^T\cdot (\Gamma_{1}^{\prime}, \Gamma_{2}^{\prime}, \Gamma_{3}^{\prime}, \Gamma_{4}^{\prime})=P^T (\Gamma_1, \Gamma_2, \Gamma_3, \Gamma_4) ^T\cdot (\Gamma_1, \Gamma_2, \Gamma_3, \Gamma_4) P=P^T N P =N.
\end{equation}  

From the integration representation \eqref{eq:IntegrationRepresentation}, the relation $(\varphi_1^\prime, \varphi_2^\prime, \varphi_3^\prime,\varphi_4^\prime)=(\varphi_1,\varphi_2,\varphi_3,\varphi_4) M_\Sigma$ implies
\begin{equation*}
	\begin{split}
		& (\Gamma_{1}^{\prime}, \Gamma_{2}^{\prime}, \Gamma_{3}^{\prime}, \Gamma_{4}^{\prime}) N^{-1}
		=(\Gamma_1, \Gamma_2, \Gamma_3, \Gamma_4) N^{-1} M_\Sigma \\
		\Longrightarrow &
		(\Gamma_{1}, \Gamma_{2}, \Gamma_{3}, \Gamma_{4}) P N^{-1} = 
		(\Gamma_1, \Gamma_2, \Gamma_3, \Gamma_4) N^{-1} M_\Sigma  \\
		\Longrightarrow & PN^{-1}=N^{-1}M_\Sigma    \\
		\Longrightarrow & M_\Sigma=N P N^{-1} =(P^T)^{-1}
	\end{split}
\end{equation*}
where the last equality is from \eqref{eq:Isometry}.

Therefore, $M_{\Sigma}$ becomes a unimodular matrix from the fact $(P^T)^{-1}$ is  in $\mathrm{GL}_4(\Z)$. Consequently, our target monodromy matrix satisfies the property 
\begin{align}\label{eq:theMonodromyMatrix}
{M_{\Sigma}\in \rm{GL}_4(\Z)}.
\end{align}

The calculation of $M_\Sigma$ by explicitly constructing concrete $\Gamma_1, \Gamma_2, \Gamma_3, \Gamma_4$ and deforming them along a loop $\Sigma$ requires an enormous effort (cf.\ \cite[Sections 4--5]{bib:ishige}). Alternatively, we employ validated numerics to achieve the computation of $M_\Sigma$ rigorously and efficiently.

\subsection{Main result and organization of the paper}
The unimodularity of our target monodromy matrix, discussed in the previous section, and rigorous inclusion of the monodromy matrix of the form \eqref{eq:MonodromyMat} provides our main theorem of this paper as follows:
\begin{thm}\label{thm-main-result}
	Consider the Picard--Fuchs differential equation of the form \eqref{eq:PicardFuchs}. 
	The monodromy matrices of \eqref{eq:PicardFuchs} along with $\Sigma_{p_i}$ from the base point $p_0$ are given by
	\begin{equation}\label{eq:ResultingMonodromy}
	\begin{aligned}
		M_{\Sigma_1}&= \begin{pmatrix}
			-1 & -2 & -2 & -1\\ 0 & -1 & 0 & 0\\ 0 & 4 & 3 & 2\\ 0 & -4 & -4 & -3
		\end{pmatrix},&
		M_{\Sigma_2} &= \begin{pmatrix}
			-1 & 0 & 0 & 0\\ 0 & -1 & 0 & 0\\ 0 & 0 & 3 & 2\\ 0 & 0 & -4 & -3
		\end{pmatrix},&
		M_{\Sigma_3} &= \begin{pmatrix}
			1 & 0 & 0 & 0\\ 0 & 1 & 0 & 0\\ 0 & 0 & -1 & 0\\ 0 & 0 & 0 & 1
		\end{pmatrix}\\
		M_{\Sigma_4} &= \begin{pmatrix}
			1 & 1 & -1 & 0\\ 0 & 1 & 0 & 0\\ 0 & 2 & -1 & 0\\ 0 & 0 & 0 & 1
		\end{pmatrix},&
		M_{\Sigma_5} &= \begin{pmatrix}
			0 & 1 & 0 & 0\\ 1 & 0 & 0 & 0\\ 0 & 0 & 1 & 0\\ 0 & 0 & 0 & 1
		\end{pmatrix},&
		M_{\Sigma_6} &= \begin{pmatrix}
			0 & 1 & 0 & 0\\ 1 & 0 & 0 & 0\\ 0 & 0 & 1 & 0\\ 0 & 0 & 0 & 1
		\end{pmatrix}.
	\end{aligned}
	\end{equation}

\end{thm}
Since $M_{\Sigma_i} (i=1,\dots,6)$ generate the monodromy group for the Picard--Fuchs differential equation \eqref {eq:PicardFuchs},  the monodromy problem for \eqref {eq:PicardFuchs} is rigorously resolved through computer-assisted proofs based on the results of Theorem \ref{thm-main-result}.

\begin{rem}
	While this form of monodromy matrices is already provided in \cite{bib:ishige} using specific properties of K3 surfaces, we emphasize that our general computer-assisted approach demonstrates the potential of addressing more complicated monodromy problems through validated numerics.
\end{rem}

\vspace*{2mm}

The rest of the present paper is organized as follows: In Section~\ref{sec:setup}, we introduce the Pfaffian equation of the Picard--Fuchs differential equation, detailing its formulation and the associated singular points. We also define the analytic continuation paths and present the fundamental system of solutions, which is central to constructing the monodromy matrices.
Section~\ref{sec:rigor} describes the rigorous numerical methods used to compute the fundamental system of solutions of the Picard--Fuchs differential equation. It provides details on the truncation error bounds and interval arithmetic techniques used to ensure the rigorous inclusion of the monodromy matrices.
Section~\ref{sec:results} presents the results of the rigorous computation of the monodromy matrices for the Picard--Fuchs differential equation, including the validation of their accuracy and the resolution of the monodromy problem.

\section{Setting up the problem}\label{sec:setup}
To compute the monodromy of \eqref{eq:PicardFuchs} using numerical methods within the framework of computer-assisted proofs, we consider the Pfaffian equation for $\bphi = (\varphi, \varphi_x, \varphi_y, \varphi_{xy})^T$ derived from the Picard--Fuchs differential equation \eqref{eq:PicardFuchs}. This form is given by
%
\begin{equation}\label{eq:diff_eq}
	\rd\bm{\bphi} = (A(x,y) \rd x + B(x,y) \rd y)\bphi,
\end{equation}
where the variable coefficients $A(x,y)$ and $B(x,y)$ are defined as
\begin{align*}
A(x,y)\bydef\begin{pmatrix}
0& 1& 0& 0\\
p& a& b& \ell\\
0& 0& 0& 1\\
B_0& B_1& B_2& B_3
\end{pmatrix},\quad
B(x,y)\bydef\begin{pmatrix}
0& 0& 1& 0\\
0& 0& 0& 1\\ 
q& c& d& m\\
C_0& C_1& C_2& C_3
\end{pmatrix}.
\end{align*}
Each element of the above matrices is defined by the following complex-valued functions defined in \eqref{eq:var_coeffs}:
\begin{align*}
B_0(x,y)&\bydef (p_y+bq+\ell (q_x+cp))/\kappa, &&C_0(x,y)\bydef (q_x+cp+m(p_y+bq))/\kappa\\
B_1(x,y)&\bydef (a_y+bc+\ell (c_x+ca)+\ell q)/\kappa,&&C_1(x,y)\bydef (c_x+ac+m(a_y+bc)+q)/\kappa\\
B_2(x,y)&\bydef (b_y+bd+\ell (d_x+bc)+p)/\kappa,&&C_2(x,y)\bydef (d_x+bc+m(b_y+bd)+mp)/\kappa\\
B_3(x,y)&\bydef (\ell _y+a+bm+\ell (m_x+d+c\ell ))/\kappa,&&C_3(x,y)\bydef (m_x+d+c\ell +m(\ell _y+a+bm))/\kappa,
\end{align*}
where  $\kappa\equiv\kappa(x,y) \bydef 1-\ell(x,y) m(x,y)$ and the subscripts denote the partial derivatives in each variable.
We note that the monodromy of \eqref{eq:diff_eq} is the equivalent to that of \eqref{eq:PicardFuchs}.

\begin{rem}
The Pfaffian equation \eqref{eq:diff_eq} holds $-A_y + B_x - (AB-BA) = 0$ for any $x,y\in \C$.
\end{rem}

Now let us consider a variable transformation between $(x,y)$ and $(\lambda,\mu)$ such that
\begin{equation}\label{eq:def_transformc}
	\lambda=\frac{x}{y}+\frac{1}{4},\quad\mu=\frac{x^3}{y^2}\iff x=\frac{\mu}{(\lambda-\frac{1}{4})^2},\quad y=\frac{\mu}{(\lambda-\frac{1}{4})^3}.
\end{equation}
Let $(\lambda_0,\mu_0)\bydef(2^{-10},2^{-10})\in\C^2$. We set a base point $p_0$ as 
\begin{align}\label{eq:base_pt}
	(x_0,y_0)\bydef\left(\mu_0\left(\lambda_0-\frac{1}{4}\right)^{-2},\ \mu_0\left(\lambda_0-\frac{1}{4}\right)^{-3}\right)\approx(0.0157478,-0.0632382),
\end{align}
\revise{which is inside of the convergence region of particular solutions to \eqref{eq:PicardFuchs} defined in Section \ref{sec:fundamental_sol}. This specific choice also allows us to control the truncation errors of the particular solutions, which is derived in Section~\ref{sec:rigor}.}

Next, we consider the singular locus\footnote{A set of singular points of \eqref{eq:PicardFuchs} are called the \emph{singular locus} as introduced in Section~\ref{sec:intro}.} of \eqref{eq:PicardFuchs}.
The singular locus is denoted by
\begin{align}\label{eq:singula-locus}
\mathcal{S}\bydef \left\{ (x,y)\in\C^2 : xy(4x+y)\left[\left(36x+\frac{27}{2}y+1\right)^2-(1-12x)^3\right]=0\right\}.
\end{align}

\begin{figure}[htbp]
	\centering
	\includegraphics[width=9.5cm]{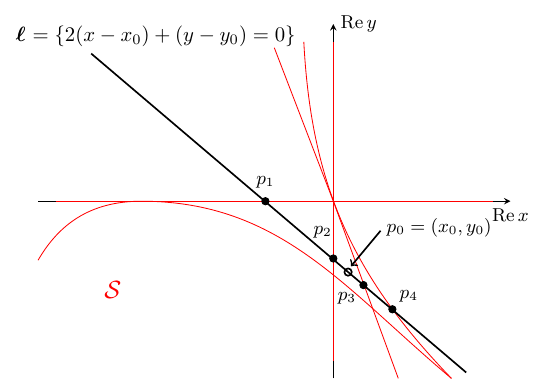}
	\caption{Singular locus $\mathcal{S}$ of \eqref{eq:PicardFuchs} (red lines) and the generic line $\bm{\ell}$ (black thick line): Four singular points $p_i$ ($i=1,\dots,4$) appear on $\bm{\ell}$, but two more points $p_5$ and $p_6$ cannot be seen in this picture because these have imaginary part. We also take the base point $p_0$ in \eqref{eq:base_pt} on the line $\bm{\ell}$.}\label{fig:Singular}
\end{figure}

In the followings, thanks to the Zariski--van Kampen theorem \cite{Shimada2003,ShimadaLectures}, without loss of generality we restrict our discussions on the generic line of $\C^2$, which is isomorphic to a complex plane $\C$, such that
\begin{equation}\label{eq:C2plane}
	\bm{\ell}\bydef\left\{(x,y)\in\C^2 : 2(x-x_0) + (y-y_0) = 0\right\}.
\end{equation}
Taking intersection of  $\mathcal{S}$ and $\bm{\ell}$, we have six singular points $p_i$ ($i=1,\dots,6$) such that
\begin{equation}\label{eq:sing_pts_approx}
\begin{aligned}
	p_1 &\approx (-0.0158713,0)\\
	p_2 &\approx (0,-0.0317426)\\
	p_3 &\approx (0.0158713,-0.0634852)\\
	p_4 &\approx (0.0164304,-0.0646034)\\
	p_5 &\approx (0.0933473+0.122495\im,-0.218437-0.24499\im)\\
	p_6 &\approx (0.0933473-0.122495\im,-0.218437+0.24499\im),
\end{aligned}
\end{equation}
where $\im=\sqrt{-1}$ is the imaginary unit. 
\revise{We note that rigorous enclosures for the values $p_1,\dots,p_6$ will later be obtained in Section \ref{sec:singular_points}.}
Figure \ref{fig:Singular} briefly displays a geometrical picture of our setting.

\subsection{Select a path of the contour}\label{sec:path_contour}
We set the contour of analytic continuation to obtain the monodromy. More precisely, we determine six paths $\Sigma_{i}$ ($i=1,\dots,6$)  of the contour from the base point $p_0=(x_0,y_0)$, which correspond to each singular point, and then derive the ODEs to solve. Let $x_i$ and $y_i$ denote the $x$- and $y$-component of the singular points $p_i$ ($i=1,\dots,6$), respectively. We also remark that each element is on the generic line $\bm{\ell}$ defined in \eqref{eq:C2plane}.

\begin{figure}[htbp]
	\centering
	\includegraphics[width=10cm]{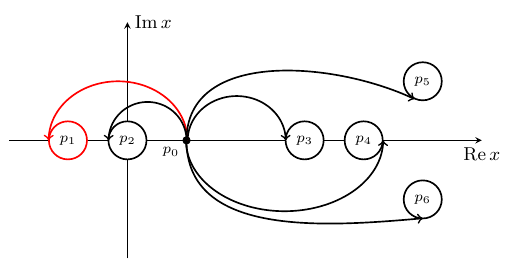}
	\caption{A brief sketch of each path $\Sigma_{i}$ ($i=1,\dots,6$) of the contour: From the base point $p_0$, the solution of \eqref{eq:diff_eq} is analytically continued to the neighborhood of each singular point via rigorous integration of ODEs. On each path, looping around the singular point, we get back to the base point.}
\end{figure}

The Zariski--van Kampen theorem \cite{Shimada2003,ShimadaLectures} implies that $\Sigma_i$ ($i=1,\dots,6$) generate $\pi_1(\mathbb{C}^2\setminus\mathcal{S},b)$ and  $M_i$ ($i=1,\dots,6$) for $\Sigma_i$ ($i=1,\dots,6$) consequently generate the monodromy group for \eqref{eq:PicardFuchs}.

\paragraph{The path \boldmath $\Sigma_1$\boldmath.}
We set a contour which loops enclosing the singular point $p_1$ from the base point $p_0$.
Let $r_1\bydef |x_1 - x_2|/2>0$.
We divide the path $\Sigma_1$ into three segments:
\begin{enumerate}
	\item[\boldmath$\Sigma_{1,1}$:]a counterclockwise semicircle centered at $c_1=(x_{c_1},y_{c_1})$ with the radius $r_{c_1}$ defined by
	\[
		x_{c_1} \bydef \frac{x_1-r_1+x_0}{2},\quad y_{c_1} \bydef \frac{y_1+2r_1+y_0}{2},\quad r_{c_1} \bydef \frac{|x_1-r_1-x_0|}{2}.
	\]
	This path is parameterized by the variable $t\in [0,1]$ defined as
	\[
	\begin{bmatrix}
	x_{1,1}(t)\\y_{1,1}(t)
	\end{bmatrix}\bydef\begin{bmatrix}
	x_{c_1} + r_{c_1} e^{\im\pi t}\\
	y_{c_1} - 2r_{c_1} e^{\im\pi t}
	\end{bmatrix}.
	\]
	\item[\boldmath$\Sigma_{1,2}$:]a counterclockwise circle centered at $p_1$ with the radius $r_1$. This path is defined by
	\[
	\begin{bmatrix}
	x_{1,2}(t)\\y_{1,2}(t)
	\end{bmatrix}\bydef\begin{bmatrix}
	x_1 + r_1 e^{\im\pi t}\\
	y_1 - 2r_1 e^{\im\pi t}
	\end{bmatrix},\quad t\in [-1,1].
	\]
	\item[\boldmath$\Sigma_{1,3}$:] a clockwise semicircle centered at $c_1$ with the radius $r_{c_1}$, which is the backward path of $\Sigma_{1,1}$. The path $\Sigma_{1,3}$ is defined by
	\[
	\begin{bmatrix}
		x_{1,3}(t)\\y_{1,3}(t)
	\end{bmatrix}\bydef\begin{bmatrix}
		x_{c_1} + r_{c_1} e^{-\im\pi t}\\
		y_{c_1} - 2r_{c_1} e^{-\im\pi t}
	\end{bmatrix},\quad t\in [-1,0].
	\]
\end{enumerate}
From the Pfaffian equation \eqref{eq:diff_eq}, to analytically continue the solution, we solve the initial value problems on each path $\Sigma_{1,j}$ ($j=1,2,3$).
\begin{align}\label{eq:ODEs_on_sigma1}
	\frac{\rd\bphi_{1,j}^{k}}{\rd t} = \left(A(x_{1,j}(t),y_{1,j}(t)) \frac{\rd x_{1,j}}{\rd t} + B(x_{1,j}(t),y_{1,j}(t)) \frac{\rd y_{1,j}}{\rd t}\right)\bphi_{1,j}^{k}
\end{align}
with the initial condition for $k=1,\dots,4$ given by
\[
	\bphi_{1,1}^{k}(0) = \bm{e}_k,\quad
	\bphi_{1,2}^{k}(-1) = \bphi_{1,1}^{k}(1),\quad
	\bphi_{1,3}^{k}(-1) = \bphi_{1,2}^{k}(1),
\]
where $\bm{e}_k$ is a canonical basis of vectors in $\C^4$. Therefore, the analytic continuation along the path $\Sigma_1=\bigcup_{j=1,2,3} \Sigma_{1,j}$ is given by
\begin{align}\label{eq:sigI_1}
(\Sigma_1)_\ast\, \mathrm{Id}=\left(\bphi_{1,3}^{1}(0),\bphi_{1,3}^{2}(0),\bphi_{1,3}^{3}(0),\bphi_{1,3}^{4}(0)\right)\in\mathrm{GL}_4(\C).
\end{align}

\begin{rem}
	It is worth noting that there is a natural idea to solve \eqref{eq:ODEs_on_sigma1} simultaneously as a boundary value problem, that is, we solve for $k=1,\dots,4$ 
	\begin{align}
		\begin{cases}
			\frac{\rd\bphi_{1,1}^{k}}{\rd t} = \left(A(x_{1,1}(t),y_{1,1}(t)) \frac{\rd x_{1,1}}{\rd t} + B(x_{1,1}(t),y_{1,1}(t)) \frac{\rd y_{1,1}}{\rd t}\right)\bphi_{1,1}^{k}, &\bphi_{1,1}^{k}(0) = \bm{e}_k\\
			\frac{\rd\bphi_{1,2}^{k}}{\rd t} = \left(A(x_{1,2}(t),y_{1,2}(t)) \frac{\rd x_{1,2}}{\rd t} + B(x_{1,2}(t),y_{1,2}(t)) \frac{\rd y_{1,2}}{\rd t}\right)\bphi_{1,2}^{k}, &\bphi_{1,2}^{k}(-1) = \bphi_{1,1}^{k}(1)\\
			\frac{\rd\bphi_{1,3}^{k}}{\rd t} = \left(A(x_{1,3}(t),y_{1,3}(t)) \frac{\rd x_{1,3}}{\rd t} + B(x_{1,3}(t),y_{1,3}(t)) \frac{\rd y_{1,3}}{\rd t}\right)\bphi_{1,3}^{k}, &\bphi_{1,3}^{k}(-1) = \bphi_{1,2}^{k}(1).
		\end{cases}
	\end{align}
	\revise{Such boundary value problem approach could potentially reduce numerical errors and control the wrapping effect by validating the entire path globally rather than validating locally in each path of contour. Investigating the feasibility and advantages of such an approach is an interesting direction for future work.}
	On the other hand, since our implementation depends on rigorous integrator of initial value problems, we solve \eqref{eq:ODEs_on_sigma1} by the step-by-step procedure.
\end{rem}

\paragraph{The path \boldmath $\Sigma_2$\boldmath.}
This path is a contour which loops around the singular point $p_2$ starting and ending at the base point $p_0$.
Let $r_2 = |x_0|$. This path is parameterized by
\[
\begin{bmatrix}
x_{2}(t)\\y_{2}(t)
\end{bmatrix}\bydef\begin{bmatrix}
x_2 + r_{2} e^{\im\pi t}\\
y_2 - 2r_{2} e^{\im\pi t}
\end{bmatrix},\quad t\in [0,2].
\]
From the Pfaffian equation \eqref{eq:diff_eq}, we solve the initial value problem to analytically continue the solution on $\Sigma_{2}$.
\[
\frac{\rd\bphi_{2}^{k}}{\rd t} = \left(A(x_{2}(t),y_{2}(t)) \frac{\rd x_{2}}{\rd t} + B(x_{2}(t),y_{2}(t)) \frac{\rd y_{2}}{\rd t}\right)\bphi_{2}^{k}
\]
with the initial condition $\bphi_{2}^{k}(0) = \bm{e}_k$ for $k=1,\dots,4$. Therefore, the analytic continuation is given by
\begin{align}\label{eq:sigI_2}
(\Sigma_2)_\ast\, \mathrm{Id}=\left(\bphi_{2}^{1}(2),\bphi_{2}^{2}(2),\bphi_{2}^{3}(2),\bphi_{2}^{4}(2)\right)\in\mathrm{GL}_4(\C).
\end{align}

\paragraph{The path \boldmath $\Sigma_3$\boldmath.}

The next path is a contour which loops around the singular point $p_3$ starting and ending at the base point $p_0$.
Let $r_3 = |x_3-x_0|$. This path is parameterized by
\[
\begin{bmatrix}
x_{3}(t)\\y_{3}(t)
\end{bmatrix}\bydef\begin{bmatrix}
x_3 + r_{3} e^{\im\pi t}\\
y_3 - 2r_{3} e^{\im\pi t}
\end{bmatrix},\quad t\in [-1,1].
\]
From the Pfaffian equation \eqref{eq:diff_eq}, we solve the initial value problem to analytically continue the solution on $\Sigma_{3}$.
\[
\frac{\rd\bphi_{3}^{k}}{\rd t} = \left(A(x_{3}(t),y_{3}(t)) \frac{\rd x_{3}}{\rd t} + B(x_{3}(t),y_{3}(t)) \frac{\rd y_{3}}{\rd t}\right)\bphi_{3}^{k}
\]
with the initial condition $\bphi_{3}^{k}(-1) = \bm{e}_k$ for $k=1,\dots,4$. Therefore, the analytic continuation is given by
\begin{align}\label{eq:sigI_3}
(\Sigma_3)_\ast\, \mathrm{Id}=\left(\bphi_{3}^{1}(1),\bphi_{3}^{2}(1),\bphi_{3}^{3}(1),\bphi_{3}^{4}(1)\right)\in\mathrm{GL}_4(\C).
\end{align}

\paragraph{The path \boldmath $\Sigma_4$\boldmath.}
The fourth path is a contour which loops around the singular point $p_4$ starting and ending at the base point $p_0$.
Let $r_4\bydef |x_4 - x_3|/2>0$.
We divide the path $\Sigma_4$ into three segments:
\begin{enumerate}
	\item[\boldmath$\Sigma_{4,1}$:]a counterclockwise semicircle centered at $c_4=(x_{c_4},y_{c_4})$ with radius $r_{c_4}>0$ defined by
	\[
	x_{c_4} \bydef \frac{x_4+r_4+x_0}{2},\quad y_{c_4} \bydef \frac{y_4-2r_4+y_0}{2},\quad r_{c_4} \bydef \frac{|x_4+r_4-x_0|}{2}.
	\]
	This path is defined by
	\[
	\begin{bmatrix}
	x_{4,1}(t)\\y_{4,1}(t)
	\end{bmatrix}\bydef\begin{bmatrix}
	x_{c_4} + r_{c_4} e^{\im\pi t}\\
	y_{c_4} - 2r_{c_4} e^{\im\pi t}
	\end{bmatrix},\quad t\in [-1,0].
	\]
	\item[\boldmath$\Sigma_{4,2}$:]a counterclockwise circle centered at $p_4$ with the radius $r_4$. This path is defined by
	\[
	\begin{bmatrix}
	x_{4,2}(t)\\y_{4,2}(t)
	\end{bmatrix}\bydef\begin{bmatrix}
	x_4 + r_4 e^{\im\pi t}\\
	y_4 - 2r_4 e^{\im\pi t}
	\end{bmatrix},\quad t\in [0,2].
	\]
	\item[\boldmath$\Sigma_{4,3}$:]a clockwise semicircle centered at $c_4$ with the radius $r_{c_4}$, which is the backward path of $\Sigma_{4,1}$. The path $\Sigma_{4,3}$ is parameterized by
	\[
	\begin{bmatrix}
	x_{4,3}(t)\\y_{4,3}(t)
	\end{bmatrix}\bydef\begin{bmatrix}
	x_{c_4} + r_{c_4} e^{-\im\pi t}\\
	y_{c_4} - 2r_{c_4} e^{-\im\pi t}
	\end{bmatrix},\quad t\in [0,1].
	\]
\end{enumerate}
From the Pfaffian equation \eqref{eq:diff_eq}, to analytically continue the solution, we solve the initial value problem on each path $\Sigma_{4,j}$ ($j=1,2,3$).
\[
\frac{\rd\bphi_{4,j}^{k}}{\rd t} = \left(A(x_{4,j}(t),y_{4,j}(t)) \frac{\rd x_{4,j}}{\rd t} + B(x_{4,j}(t),y_{4,j}(t)) \frac{\rd y_{4,j}}{\rd t}\right)\bphi_{4,j}^{k}
\]
with the initial condition for $k=1,\dots,4$ given by
\[
\bphi_{4,1}^{k}(-1) = \bm{e}_k,\quad
\bphi_{4,2}^{k}(0) = \bphi_{4,1}^{k}(0),\quad
\bphi_{4,3}^{k}(0) = \bphi_{4,2}^{k}(2).
\]
Therefore, the analytic continuation along the path $\Sigma_4=\bigcup_{j=1,2,3} \Sigma_{4,j}$ is given by
\begin{align}\label{eq:sigI_4}
(\Sigma_4)_\ast\, \mathrm{Id}=\left(\bphi_{4,3}^{1}(1),\bphi_{4,3}^{2}(1),\bphi_{4,3}^{3}(1),\bphi_{4,3}^{4}(1)\right)\in\mathrm{GL}_4(\C).
\end{align}

\paragraph{The path \boldmath $\Sigma_5$\boldmath.}
The fifth path is a contour which loops around the singular point $p_5$ starting and ending at the base point $p_0$.
Let $r_5\bydef |x_5 - x_4|/2>0$.
We divide the path $\Sigma_5$ into three segments:
\begin{enumerate}
	\item[\boldmath$\Sigma_{5,1}$:]a segment from $p_0=(x_0,y_0)$ to the point $(x_5-r_5,y_5+2r_5)$, which is defined by
	\[
	\begin{bmatrix}
	x_{5,1}(t)\\y_{5,1}(t)
	\end{bmatrix}\bydef\begin{bmatrix}
	(1-t)x_0 + t(x_5-r_5)\\
	(1-t)y_0 + t(y_5+2r_5)
	\end{bmatrix},\quad t\in [0,1].
	\]
	\item[\boldmath$\Sigma_{5,2}$:]a counterclockwise circle centered at $p_5$ with the radius $r_5$. This path is defined by
	\[
	\begin{bmatrix}
	x_{5,2}(t)\\y_{5,2}(t)
	\end{bmatrix}\bydef\begin{bmatrix}
	x_5 + r_5 e^{\im\pi t}\\
	y_5 - 2r_5 e^{\im\pi t}
	\end{bmatrix},\quad t\in [-1,1].
	\]
	\item[\boldmath$\Sigma_{5,3}$:] the backward segment from the point $(x_5-r_5,y_5+2r_5)$ to $p_0$. Then, the path $\Sigma_{5,3}$ is defined by
	\[
	\begin{bmatrix}
	x_{5,3}(t)\\y_{5,3}(t)
	\end{bmatrix}\bydef\begin{bmatrix}
	(1-t)(x_5-r_5) + tx_0\\
	(1-t)(y_5+2r_5) + ty_0
	\end{bmatrix},\quad t\in [0,1].
	\]
\end{enumerate}
From the Pfaffian equation \eqref{eq:diff_eq}, to analytically continue the solution, we solve the initial value problem on each path $\Sigma_{5,j}$ ($j=1,2,3$).
\[
\frac{\rd\bphi_{5,j}^{k}}{\rd t} = \left(A(x_{5,j}(t),y_{5,j}(t)) \frac{\rd x_{5,j}}{\rd t} + B(x_{5,j}(t),y_{5,j}(t)) \frac{\rd y_{5,j}}{\rd t}\right)\bphi_{5,j}^{k}
\]
with the initial condition for $k=1,\dots,4$ given by
\[
\bphi_{5,1}^{k}(0) = \bm{e}_k,\quad
\bphi_{5,2}^{k}(-1) = \bphi_{5,1}^{k}(1),\quad
\bphi_{5,3}^{k}(0) = \bphi_{5,2}^{k}(1).
\]
Therefore, the analytic continuation along the path $\Sigma_5=\bigcup_{j=1,2,3} \Sigma_{5,j}$  is given by
\begin{align}\label{eq:sigI_5}
(\Sigma_5)_\ast\, \mathrm{Id}=\left(\bphi_{5,3}^{1}(1),\bphi_{5,3}^{2}(1),\bphi_{5,3}^{3}(1),\bphi_{5,3}^{4}(1)\right)\in\mathrm{GL}_4(\C).
\end{align}

\paragraph{The path \boldmath $\Sigma_6$\boldmath.}
The final path is a contour which loops around the singular point $p_6$ starting and ending at the base point $p_0$.
Let $r_6\bydef |x_6 - x_4|/2>0$.
We divide the path $\Sigma_6$ into three segments:
\begin{enumerate}
	\item[\boldmath$\Sigma_{6,1}$:]a segment from $p_0=(x_0,y_0)$ to the point $(x_6-r_6,y_6+2r_6)$, which is defined by
	\[
	\begin{bmatrix}
	x_{6,1}(t)\\y_{6,1}(t)
	\end{bmatrix}\bydef\begin{bmatrix}
	(1-t)x_0 + t(x_6-r_6)\\
	(1-t)y_0 + t(y_6+2r_6)
	\end{bmatrix},\quad t\in [0,1].
	\]
	\item[\boldmath$\Sigma_{6,2}$:]a counterclockwise circle centered at $p_6$ with the radius $r_6$. This path is defined by
	\[
	\begin{bmatrix}
	x_{6,2}(t)\\y_{6,2}(t)
	\end{bmatrix}\bydef\begin{bmatrix}
	x_6 + r_6 e^{\im\pi t}\\
	y_6 - 2r_6 e^{\im\pi t}
	\end{bmatrix},\quad t\in [-1,1].
	\]
	\item[\boldmath$\Sigma_{6,3}$:]backward the segment from $(x_6-r_6,y_6+2r_6)$ to $p_0$. The path $\Sigma_{6,3}$ is defined by
	\[
	\begin{bmatrix}
	x_{6,3}(t)\\y_{6,3}(t)
	\end{bmatrix}\bydef\begin{bmatrix}
	(1-t)(x_6-r_6) + tx_0\\
	(1-t)(y_6+2r_6) + ty_0
	\end{bmatrix},\quad t\in [0,1].
	\]
\end{enumerate}
From the Pfaffian equation \eqref{eq:diff_eq}, to analytically continue the solution, we solve the initial value problem on each path $\Sigma_{6,j}$ ($j=1,2,3$).
\[
\frac{\rd\bphi_{6,j}^{k}}{\rd t} = \left(A(x_{6,j}(t),y_{6,j}(t)) \frac{\rd x_{6,j}}{\rd t} + B(x_{6,j}(t),y_{6,j}(t)) \frac{\rd y_{6,j}}{\rd t}\right)\bphi_{6,j}^{k}
\]
with the initial condition for $k=1,\dots,4$ given by
\[
\bphi_{6,1}^{k}(0) =  \bm{e}_k,\quad
\bphi_{6,2}^{k}(-1) = \bphi_{6,1}^{k}(1),\quad
\bphi_{6,3}^{k}(0) = \bphi_{6,2}^{k}(1).
\]
Therefore, the analytic continuation along the path $\Sigma_6=\bigcup_{j=1,2,3} \Sigma_{6,j}$ is given by
\begin{align}\label{eq:sigI_6}
(\Sigma_6)_\ast\, \mathrm{Id}=\left(\bphi_{6,3}^{1}(1),\bphi_{6,3}^{2}(1),\bphi_{6,3}^{3}(1),\bphi_{6,3}^{4}(1)\right)\in\mathrm{GL}_4(\C).
\end{align}

\subsection{Particular solutions to the Picard--Fuchs differential equation}\label{sec:fundamental_sol}
\revise{The rest of the setting} to compute the monodromy is to obtain the value of the fundamental system of solutions for  \eqref{eq:diff_eq} at the base point $p_0$. \revise{The Pfaffian equation \eqref{eq:diff_eq} is formulated for $\bphi = (\varphi, \varphi_x, \varphi_y, \varphi_{xy})^T$, and hence the fundamental system of solutions consists of particular solutions to \eqref{eq:PicardFuchs} and their partial derivatives.
In this part, we introduce certain series representations of the particular solutions defined on a neighborhood of $p_0$, which is presented in \cite{bib:ishige}.}
We also derive their derivatives to compute the value of fundamental system of solutions for obtaining the monodromy matrices of \eqref{eq:diff_eq}.
\begin{thm}[{\cite[Proposition~3.7]{bib:ishige}}]\label{thm:Fundamantal_sol}
	Consider the Picard--Fuchs differential equation \eqref{eq:PicardFuchs} under the coordinate $(\lambda,\mu)$ given in \eqref{eq:def_transformc}.
	If the variables $(\lambda,\mu)$ satisfy
	\begin{align}\label{eq:convergencecondition}
		\begin{cases}
			\lvert\lambda\rvert+\lvert\mu\rvert<\frac1{256}\\
			(\lvert\lambda\rvert+{\lvert\lambda\rvert}^2)/\lvert\mu\rvert<\frac{64}{25},
		\end{cases}
	\end{align}
	then the particular solutions of \eqref{eq:PicardFuchs} are given by the following series representations:
	\begin{align}\label{eq:phi1}
		\varphi_1(\lambda,\mu) &\bydef \sum_{l,m\ge0} a_{l,m}\lambda^l\mu^m\\\label{eq:phi2}
		\varphi_2(\lambda,\mu) &\bydef \frac{1}{2\pi^2}\sum_{l,m\ge0}a_{l,m}\left[(\log\mu+b_{l,m})^2-c_{l,m}\right]\lambda^l\mu^m+\frac{1}{2} d_{l+1,m}\frac{\lambda^{l+2m+1}}{\mu^{l+m+1}}\\\label{eq:phi3}
		\varphi_3(\lambda,\mu) &\bydef \frac{1}{4\pi^2}\sum_{l,m\ge0}d_{l+\frac{1}{2},m}\frac{\lambda^{l+2m+\frac{1}{2}}}{\mu^{l+m+\frac{1}{2}}}\\\label{eq:phi4}
		\varphi_4(\lambda,\mu) &\bydef \frac{1}{2\pi\im}\sum_{l,m\ge0}a_{l,m}(\log\mu+b_{l,m})\lambda^l\mu^m,
	\end{align}
	where $\im=\sqrt{-1}$ and the coefficients $a_{l,m}$, $b_{l,m}$, $c_{l,m}$, $d_{l,m}$ are defined by
	\begin{align}\label{eq:a_lm}
		a_{l,m}&\bydef \frac{(2l+4m)!}{(l+m)!\,l!\,(m!)^3}\\\label{eq:b_lm}
		b_{l,m}&\bydef \sum_{j=1}^{2l+4m}\frac{4}{j}-\sum_{j=1}^{l+m}\frac{1}{j}-\sum_{j=1}^{m}\frac{3}{j}\\\label{eq:c_lm}
		c_{l,m}&\bydef \sum_{j=1}^{2l+4m}\frac{16}{j^2}-\sum_{j=1}^{l+m}\frac{1}{j^2}-\sum_{j=1}^{m}\frac{3}{j^2}\\\label{eq:d_lm}
		d_{l,m}&\bydef (-1)^{l+m}\frac{\Gamma(l+m)^3}{\Gamma(2l)\Gamma(l+2m+1)\Gamma(m+1)}.
	\end{align}
\end{thm}

\begin{rem}
	Proposition~3.11 in \cite{bib:ishige} implies that this specific choice of particular solutions ensures the resulting monodromy matrices have integer entries. Therefore, we construct the fundamental system of solutions using these four particular solutions to obtain the monodromy matrices in $\mathrm{GL}_4(\Z)$.
\end{rem}

Note that the condition \eqref{eq:convergencecondition} holds at this base point, i.e., $(\lambda_0,\mu_0)=(2^{-10},2^{-10})$.

Next, we derive derivatives of the particular solutions.
From \eqref{eq:def_transformc} partial derivatives, with respect to $\lambda$ and $\mu$, are given by
\begin{align}\label{eq:lam_diff}
	\lambda_{x} = y^{-1},\quad \lambda_{y} = -xy^{-2},\quad\lambda_{xy}=-y^{-2},
\end{align}
\begin{align}\label{eq:mu_diff}
	\mu_{x}=3x^2y^{-2},\quad \mu_{y}=-2x^3y^{-3}\quad \mu_{xy}=-6x^2y^{-3}.
\end{align}
Using the chain rule of derivative, it follows that for $k=1,2,3,4$
\begin{align}\label{eq:phi_x}
	(\varphi_k)_x = \frac{\partial}{\partial x}\varphi_k(\lambda(x,y),\mu(x,y))
	=\frac{\partial\varphi_k}{\partial\lambda}\lambda_x + \frac{\partial\varphi_k}{\partial\mu}\mu_x.
\end{align}
Similarly, the $y$-derivative is given by
\begin{align}\label{eq:phi_y}
(\varphi_k)_y = \frac{\partial}{\partial y}\varphi_k(\lambda(x,y),\mu(x,y))
=\frac{\partial\varphi_k}{\partial\lambda}\lambda_y + \frac{\partial\varphi_k}{\partial\mu}\mu_y.
\end{align}
Furthermore, we have $xy$-derivative as follows:
\begin{align}
(\varphi_k)_{xy} &= \frac{\partial^2}{\partial x\partial y}\varphi_k(\lambda(x,y),\mu(x,y))\\
&=\frac{\partial}{\partial x}\left(\frac{\partial\varphi_k}{\partial\lambda}\lambda_y + \frac{\partial\varphi_k}{\partial\mu}\mu_y\right)\\
&=\frac{\partial}{\partial x}\left(\frac{\partial\varphi_k}{\partial\lambda}\right)\lambda_y + \frac{\partial\varphi_k}{\partial\lambda}\lambda_{xy} +  \frac{\partial}{\partial x}\left(\frac{\partial\varphi_k}{\partial\mu}\right)\mu_y + \frac{\partial\varphi_k}{\partial\mu}\mu_{xy}\\
&=\left(\frac{\partial^2\varphi_k}{\partial \lambda^2}\lambda_{x}+\frac{\partial^2\varphi_k}{\partial\lambda\partial\mu}\mu_{x}\right)\lambda_y + \frac{\partial\varphi_k}{\partial\lambda}\lambda_{xy} +  \left(\frac{\partial^2\varphi_k}{\partial\mu\partial\lambda}\lambda_{x} + \frac{\partial^2\varphi_k}{\partial\mu^2}\mu_{x}\right)\mu_y + \frac{\partial\varphi_k}{\partial\mu}\mu_{xy}.\label{eq:phi_xy}
\end{align}
Therefore, if we have the derivatives up to 2nd order with respect to both $\lambda$ and $\mu$, then we can construct the derivatives of the particular solutions $(\varphi_k)_x$, $(\varphi_k)_y$, and $(\varphi_k)_{xy}$ ($k=1,2,3,4$) using \eqref{eq:lam_diff}, \eqref{eq:mu_diff}, \eqref{eq:phi_x}, \eqref{eq:phi_y}, and \eqref{eq:phi_xy}.

The rest of this part is devoted to deriving each derivative of $\varphi_k$ with respect to $\lambda$ and $\mu$.
The derivatives of the first fundamental solution $\varphi_1$ defined in \eqref{eq:phi1} are given by
\begin{align}
\frac{\partial\varphi_1}{\partial\lambda}(\lambda,\mu) &= \sum_{l,m\ge0} la_{l,m}\lambda^{l-1}\mu^m\label{eq:phi_1_lam}\\
\frac{\partial\varphi_1}{\partial\mu}(\lambda,\mu) &= \sum_{l,m\ge0} ma_{l,m}\lambda^l\mu^{m-1}\label{eq:phi_1_mu}\\
\frac{\partial^2\varphi_1}{\partial\lambda^2}(\lambda,\mu) &= \sum_{l,m\ge0} l(l-1)a_{l,m}\lambda^{l-2}\mu^m\label{eq:phi_1_lamlam}\\
\frac{\partial^2\varphi_1}{\partial\mu\partial\lambda}(\lambda,\mu) &= \sum_{l,m\ge0} mla_{l,m}\lambda^{l-1}\mu^{m-1}\label{eq:phi_1_mulam}\\
\frac{\partial^2\varphi_1}{\partial\mu^2}(\lambda,\mu) &= \sum_{l,m\ge0} m(m-1)a_{l,m}\lambda^l\mu^{m-2}.\label{eq:phi_1_mumu}
\end{align}

Second, using the elementary calculations, the derivatives of the fundamental solution $\varphi_2$ defined in \eqref{eq:phi2} are given by
\begin{align}
\frac{\partial\varphi_2}{\partial\lambda}(\lambda,\mu) &= \frac{1}{2\pi^2}\sum_{l,m\ge0}la_{l,m}\left[(\log\mu+b_{l,m})^2-c_{l,m}\right]\lambda^{l-1}\mu^m+\frac{1}{2} d_{l+1,m}\left(l+2m+1\right)\frac{\lambda^{l+2m}}{\mu^{l+m+1}}\label{eq:phi_2_lam}\\
\frac{\partial\varphi_2}{\partial\mu}(\lambda,\mu) &= \frac{1}{2\pi^2}\sum_{l,m\ge0}a_{l,m}\left\{2(\log\mu + b_{l,m})+m\left[(\log\mu+b_{l,m})^2-c_{l,m}\right]\right\}\lambda^l\mu^{m-1} \\
&\hphantom{=}\quad 
-\frac{1}{2} d_{l+1,m}\left(l+m+1\right)\frac{\lambda^{l+2m+1}}{\mu^{l+m+2}}\label{eq:phi_2_mu}\\
\frac{\partial^2\varphi_2}{\partial\lambda^2}(\lambda,\mu) &= \frac{1}{2\pi^2}\sum_{l,m\ge0}l(l-1)a_{l,m}\left[(\log\mu+b_{l,m})^2-c_{l,m}\right]\lambda^{l-2}\mu^m \\
&\hphantom{=}\quad +\frac{1}{2} d_{l+1,m}\left(l+2m+1\right) \left(l+2m\right)\frac{\lambda^{l+2m-1}}{\mu^{l+m+1}}\label{eq:phi_2_lamlam}\\
\frac{\partial^2\varphi_2}{\partial\mu\partial\lambda}(\lambda,\mu) &= \frac{1}{2\pi^2}\sum_{l,m\ge0}la_{l,m}\left\{2(\log\mu + b_{l,m})+m\left[(\log\mu+b_{l,m})^2-c_{l,m}\right]\right\}\lambda^{l-1}\mu^{m-1} \\
&\hphantom{=}\quad -\frac{1}{2} d_{l+1,m}\left(l+m+1\right) \left(l+2m+1\right)\frac{\lambda^{l+2m}}{\mu^{l+m+2}}\label{eq:phi_2_mulam}\\
\frac{\partial^2\varphi_2}{\partial\mu^2}(\lambda,\mu) &= \frac{1}{2\pi^2}\sum_{l,m\ge0}a_{l,m}\left\{2+(4m-2)(\log\mu+b_{l,m})+m(m-1)\left[(\log\mu+b_{l,m})^2-c_{l,m}\right]\right\}\lambda^l\mu^{m-2} \\
&\hphantom{=}\quad +\frac{1}{2} d_{l+1,m} \left(l+m+1\right) \left(l+m+2\right)\frac{\lambda^{l+2m+1}}{\mu^{l+m+3}}.\label{eq:phi_2_mumu}
\end{align}

Third, the derivatives of the fundamental solution $\varphi_3$ defined in \eqref{eq:phi3} are given by
\begin{align}
\frac{\partial\varphi_3}{\partial\lambda}(\lambda,\mu) &= \frac{1}{4\pi^2}\sum_{l,m\ge0}d_{l+\frac{1}{2},m}\left(l+2m+\frac{1}{2}\right)\frac{\lambda^{l+2m-\frac{1}{2}}}{\mu^{l+m+\frac{1}{2}}}\label{eq:phi_3_lam}\\
\frac{\partial\varphi_3}{\partial\mu}(\lambda,\mu) &= -\frac{1}{4\pi^2}\sum_{l,m\ge0}d_{l+\frac{1}{2},m}\left(l+m+\frac{1}{2}\right)\frac{\lambda^{l+2m+\frac{1}{2}}}{\mu^{l+m+\frac{3}{2}}}\label{eq:phi_3_mu}\\
\frac{\partial^2\varphi_3}{\partial\lambda^2}(\lambda,\mu) &= \frac{1}{4\pi^2}\sum_{l,m\ge0}d_{l+\frac{1}{2},m}\left(l+2m+\frac{1}{2}\right)\left(l+2m-\frac{1}{2}\right)\frac{\lambda^{l+2m-\frac{3}{2}}}{\mu^{l+m+\frac{1}{2}}}\label{eq:phi_3_lamlam}\\
\frac{\partial^2\varphi_3}{\partial\mu\partial\lambda}(\lambda,\mu) &= -\frac{1}{4\pi^2}\sum_{l,m\ge0}d_{l+\frac{1}{2},m}\left(l+2m+\frac{1}{2}\right)\left(l+m+\frac{1}{2}\right)\frac{\lambda^{l+2m-\frac{1}{2}}}{\mu^{l+m+\frac{3}{2}}}\label{eq:phi_3_mulam}\\
\frac{\partial^2\varphi_3}{\partial\mu^2}(\lambda,\mu) &= \frac{1}{4\pi^2}\sum_{l,m\ge0}d_{l+\frac{1}{2},m}\left(l+m+\frac{1}{2}\right)\left(l+m+\frac{3}{2}\right)\frac{\lambda^{l+2m+\frac{1}{2}}}{\mu^{l+m+\frac{5}{2}}}.\label{eq:phi_3_mumu}
\end{align}

Lastly, the derivatives of the fundamental solution $\varphi_4$ defined in \eqref{eq:phi4} are given by
\begin{align}
\frac{\partial\varphi_4}{\partial\lambda}(\lambda,\mu) &= \frac{1}{2\pi\im}\sum_{l,m\ge0}la_{l,m}(\log\mu+b_{l,m})\lambda^{l-1}\mu^m\label{eq:phi_4_lam}\\
\frac{\partial\varphi_4}{\partial\mu}(\lambda,\mu) &= \frac{1}{2\pi\im}\sum_{l,m\ge0}a_{l,m} \left[1+m(\log\mu+b_{l,m})\right]\lambda^l\mu^{m-1}\label{eq:phi_4_mu}\\
\frac{\partial^2\varphi_4}{\partial\lambda^2}(\lambda,\mu) &= \frac{1}{2\pi\im}\sum_{l,m\ge0}l(l-1)a_{l,m}(\log\mu+b_{l,m})\lambda^{l-2}\mu^m\label{eq:phi_4_lamlam}\\
\frac{\partial^2\varphi_4}{\partial\mu\partial\lambda}(\lambda,\mu) &= \frac{1}{2\pi\im}\sum_{l,m\ge0}la_{l,m} \left[1+m(\log\mu+b_{l,m})\right]\lambda^{l-1}\mu^{m-1}\label{eq:phi_4_mulam}\\
\frac{\partial^2\varphi_4}{\partial\mu^2}(\lambda,\mu) &= \frac{1}{2\pi\im}\sum_{l,m\ge0}a_{l,m}\left\{(m-1)+m\left[1+(m-1)(\log\mu+b_{l,m})\right]\right\}\lambda^l\mu^{m-2}.\label{eq:phi_4_mumu}
\end{align}

Consequently, using the formulas \eqref{eq:phi_x}, \eqref{eq:phi_y}, and \eqref{eq:phi_xy}, we can explicitly obtain the derivatives of the particular solutions, that is $(\varphi_k)_x$, $(\varphi_k)_y$, and $(\varphi_k)_{xy}$ for $k=1,2,3,4$. Using these formulas, we construct the fundamental system of solution to \eqref{eq:diff_eq} in the neighborhood of the base point.

\subsection{Finding the monodromy matrix via analytic continuation}

Let $(\Sigma_i)_{\ast}$ denote the operation of analytic continuation along the path of contour $\Sigma_{i}$ ($i=1,\dots,6$) defined in Section \ref{sec:path_contour}. As explained in Section \ref{Sct:K3}, there exists
the monodromy matrix $M_{\Sigma_i}\in \mathrm{GL}_4(\Z)$ such that
\[
(\Sigma_i)_{\ast}\left(\bm{\varphi}^1,\bm{\varphi}^2,\bm{\varphi}^3,\bm{\varphi}^4\right) = \left(\bm{\varphi}^1,\bm{\varphi}^2,\bm{\varphi}^3,\bm{\varphi}^4\right) M_{\Sigma_i},
\]
where $\bm{\varphi}^{k}\bydef \left(\varphi_k,\partial_x \varphi_k,\partial_y \varphi_k,\partial_{xy} \varphi_k\right)^T$ denote the fundamental system of solutions using the particular solutions defined in \eqref{eq:phi1}--\eqref{eq:phi4}. 
Then the monodromy matrix gives the following group homomorphism:
\[
	\rho:\pi_1\left(p_0,\C^2\setminus \mathcal{S}\right)\to \mathrm{GL}_4(\Z),\quad \Sigma_i\mapsto M_{\Sigma_i}.
\]

To compute this monodromy matrix $M_{\Sigma_i}$, we use the value of fundamental system of solutions at the base point $p_0$, say $\Phi(p_0)\in \mathrm{GL}_4(\C)$ defined by
\begin{align}\label{eq:fundamental_sol_at_base_pt}
	\Phi(p_0)\bydef\begin{bmatrix}
	\varphi^1 & \varphi^2 & \varphi^3 &\varphi^4\\[2pt]
	\varphi_{x}^1 & \varphi_{x}^2 & \varphi_{x}^3 &\varphi_{x}^4\\[2pt]
	\varphi_{y}^1 & \varphi_{y}^2 & \varphi_{y}^3 &\varphi_{y}^4\\[2pt]
	\varphi_{xy}^1 & \varphi_{xy}^2 & \varphi_{xy}^3 &\varphi_{xy}^4
	\end{bmatrix},
\end{align}
where $\varphi^k\bydef\varphi_k(\lambda_0,\mu_0)$ and the subscript denotes partial derivatives, that is $\varphi_x^{k}\bydef\partial_x \varphi_k(\lambda_0,\mu_0)$, $\varphi_y^{k}\bydef\partial_y \varphi_k(\lambda_0,\mu_0)$, and  $\varphi_{xy}^{k}\bydef\partial_{xy} \varphi_k(\lambda_0,\mu_0)$ ($k=1,2,3,4$).

From the existence and uniqueness of the Cauchy problem of the linear ODEs, the solution space of \eqref{eq:diff_eq} and the space of its initial values are isomorphism. Therefore, the monodromy matrix satisfies
\[
	(\Sigma_i)_{\ast}\Phi(p_0) = \Phi(p_0) M_{\Sigma_i}\quad (i=1,\dots 6).
\]
As introduced in Section \ref{sec:general_approach}, this fact leads the following formula of the monodromy matrix:
\begin{align}\label{eq:monodromy_mat}
	M_{\Sigma_i} = \Phi(p_0)^{-1}\left((\Sigma_i)_\ast \mathrm{Id}\right)\Phi(p_0).
\end{align}
Here, we represents the results of each analytic continuation as described in \eqref{eq:sigI_1}, \eqref{eq:sigI_2}, \eqref{eq:sigI_3}, \eqref{eq:sigI_4}, \eqref{eq:sigI_5}, and \eqref{eq:sigI_6}
\[
(\Sigma_1)_\ast\, \mathrm{Id}=\left(\bphi_{1,3}^{1}(0),\bphi_{1,3}^{2}(0),\bphi_{1,3}^{3}(0),\bphi_{1,3}^{4}(0)\right)
\]
\[
(\Sigma_2)_\ast\, \mathrm{Id}=\left(\bphi_{2}^{1}(2),\bphi_{2}^{2}(2),\bphi_{2}^{3}(2),\bphi_{2}^{4}(2)\right)
\]
\[
(\Sigma_3)_\ast\, \mathrm{Id}=\left(\bphi_{3}^{1}(1),\bphi_{3}^{2}(1),\bphi_{3}^{3}(1),\bphi_{3}^{4}(1)\right)
\]
\[
(\Sigma_4)_\ast\, \mathrm{Id}=\left(\bphi_{4,3}^{1}(1),\bphi_{4,3}^{2}(1),\bphi_{4,3}^{3}(1),\bphi_{4,3}^{4}(1)\right)
\]
\[
(\Sigma_5)_\ast\, \mathrm{Id}=\left(\bphi_{5,3}^{1}(1),\bphi_{5,3}^{2}(1),\bphi_{5,3}^{3}(1),\bphi_{5,3}^{4}(1)\right)
\]
\[
(\Sigma_6)_\ast\, \mathrm{Id}=\left(\bphi_{6,3}^{1}(1),\bphi_{6,3}^{2}(1),\bphi_{6,3}^{3}(1),\bphi_{6,3}^{4}(1)\right).
\]

\section{Rigorous computation of the analytic continuation}\label{sec:rigor}
In this section, we present how we rigorously construct the monodromy matrix \eqref{eq:monodromy_mat} introduced in Section~\ref{sec:setup}. Our approach is based on validated numerics using interval arithmetic, which provides mathematically certified results through numerical computations. For further details on validated numerics, we refer the reader to relevant works such as \cite{MR2652784,MR2807595,MR3971222,MR1420838,notices_jb_jp}, etc. This approach enables explicit a \emph{posteriori} error controls for the computed results.
All computations in this study were executed on Ubuntu 22.04 LTS using an AMD(R) EPYC(TM) 9754 @ 2.25 GHz processor and the \emph{kv library} --- a C++ library for verified numerical computation (version 0.4.57) \cite{kv}. This library supports interval arithmetic and provides various rigorous numerical functions, including methods for verifying solutions to nonlinear systems and rigorous integration of ODEs using Taylor series representations with affine arithmetic \cite{Rump2015}. The code used to produce the results presented in the following sections is publicly available at \cite{code}.

\subsection{Explicit position of singular points}\label{sec:singular_points}
In this section, we obtain the explicit position of the six singular points, which is the intersection of the singular locus $\mathcal{S}$ in \eqref{eq:singula-locus} and the generic line $\bm{\ell}$ in \eqref{eq:C2plane}.
As presented in \eqref{eq:sing_pts_approx}, there are four real roots ($p_1,p_2,p_3,p_4$) and two complex roots ($p_5$, $p_6$).
Solving the system of complex-valued nonlinear equation
\begin{align}
	\begin{cases}
		xy(4x+y)\left[\left(36x+\frac{27}{2}y+1\right)^2-(1-12x)^3\right]=0\\
		2(x-x_0) + (y-y_0) = 0
	\end{cases}
\end{align}
using the Krawczyk method \cite{MR2652784,MR0255046} of interval arithmetic, we proved that these six singular points are included in the following interval vectors:
%
%
\begin{align*}
	p_1& \in \left(-0.0158713013_{9}^{8},~0\right)\\
	p_2& \in \left(0,~-0.0317426027_{7}^{6}\right)\\
	p_3& \in \left(0.0158713013_{8}^{9},~-0.0634852055_{4}^{3}\right)\\
	p_4& \in \left(0.0164304190_{3}^{4},~-0.0646034408_{5}^{4}\right)\\
	p_5& \in \left(0.0933472904_{8}^{9}+0.122494937_{2}^{3}\,\im,~-0.218437183_{8}^{7}-0.244989874_{5}^{4}\,\im\right)\\
	p_6& \in \left(0.0933472904_{8}^{9}-0.122494937_{3}^{2}\,\im,~-0.218437183_{8}^{7}+0.244989874_{4}^{5}\,\im\right),
\end{align*}
where the superscript and subscript represent the upper and lower bounds of the interval, respectively.

\subsection{Rigorous inclusion of the fundamental system of solutions}\label{sec:Rig_inclution_fund_sol}
\revise{Then next task} is to rigorously compute the values of $\varphi^{k}$, $\varphi^{k}_{x}$, $\varphi^{k}_{y}$, $\varphi^{k}_{xy}$ ($k=1,2,3,4$) in \eqref{eq:fundamental_sol_at_base_pt}. We recall that $\varphi^{k}= \varphi_k(\lambda_0,\mu_0)$ is the value of the particular solutions defined in \eqref{eq:phi1}--\eqref{eq:phi4} at $(\lambda_0,\mu_0)$.
Furthermore, $\varphi^{k}_{x}$, $\varphi^{k}_{y}$, and $\varphi^{k}_{xy}$ denote $x$-, $y$-, and $xy$-derivative at $(\lambda_0,\mu_0)$, respectively.

Our approach consists of two steps: First, we truncate the particular solutions \eqref{eq:phi1}--\eqref{eq:phi4} with sufficiently large indexes and compute the values of the particular solutions at $(\lambda_0,\mu_0)$ using interval arithmetic. Second, we estimate the truncation error bound based on the series representation. By combining the computed function value and the truncated error bound for each function, we obtain a rigorous inclusion of the fundamental system of solutions.

More precisely, letting $N$ be a number of truncation, we define the value of the truncated functions as
\begin{align}\label{eq:phi1_N}
	\varphi_1^{(N)}(\lambda_0,\mu_0) &\bydef \sum_{0\le l+m\le N} a_{l,m}\lambda_0^l\mu_0^m\\\label{eq:phi2_N}
	\varphi_2^{(N)}(\lambda_0,\mu_0) &\bydef \frac{1}{2\pi^2}\sum_{0\le l+m\le N}a_{l,m}\left[(\log\mu_0+b_{l,m})^2-c_{l,m}\right]\lambda_0^l\mu_0^m+\frac{1}{2} d_{l+1,m}\frac{\lambda_0^{l+2m+1}}{\mu_0^{l+m+1}}\\\label{eq:phi3_N}
	\varphi_3^{(N)}(\lambda_0,\mu_0) &\bydef \frac{1}{4\pi^2}\sum_{0\le l+m\le N}d_{l+\frac{1}{2},m}\frac{\lambda_0^{l+2m+\frac{1}{2}}}{\mu_0^{l+m+\frac{1}{2}}}\\\label{eq:phi4_N}
	\varphi_4^{(N)}(\lambda_0,\mu_0) &\bydef \frac{1}{2\pi\im}\sum_{0\le l+m\le N}a_{l,m}(\log\mu_0+b_{l,m})\lambda_0^l\mu_0^m,
\end{align}
where $a_{l,m}$, $b_{l,m}$, $c_{l,m}$, and $d_{l,m}$ are defined in \eqref{eq:a_lm}, \eqref{eq:b_lm}, \eqref{eq:c_lm}, and \eqref{eq:d_lm}, respectively.
The first step is to compute the value of truncated functions \eqref{eq:phi1_N}, \eqref{eq:phi2_N}, \eqref{eq:phi3_N}, and \eqref{eq:phi4_N} using interval arithmetic.

The second step is to get the truncated error bounds $\varepsilon^k$, $\varepsilon^k_x$, $\varepsilon^k_y$, $\varepsilon^k_{xy}>0$ ($k=1,2,3,4$) such that
\begin{align}
	\left|\varphi^k-\varphi_k^{(N)}(\lambda_0,\mu_0)\right|&\le \varepsilon^k\label{eq:epsilon}\\
	\left|\varphi^k_x-\partial_x\varphi_k^{(N)}(\lambda_0,\mu_0)\right|&\le \varepsilon^k_x\\
	\left|\varphi^k_y-\partial_y\varphi_k^{(N)}(\lambda_0,\mu_0)\right|&\le \varepsilon^k_y\\
	\left|\varphi^k_{xy}-\partial_{xy}\varphi_k^{(N)}(\lambda_0,\mu_0)\right|&\le \varepsilon^k_{xy},
\end{align}
where $\partial_x\varphi_k^{(N)}(\lambda_0,\mu_0)$, $\partial_y\varphi_k^{(N)}(\lambda_0,\mu_0)$, and $\partial_{xy}\varphi_k^{(N)}(\lambda_0,\mu_0)$ denote each partial derivative of the truncated function $\varphi_k^{(N)}$ at $(\lambda_0,\mu_0)$ via the form \eqref{eq:phi_x}, \eqref{eq:phi_y}, and \eqref{eq:phi_xy}, respectively.

From \eqref{eq:lam_diff}, \eqref{eq:mu_diff}, \eqref{eq:phi_x}, \eqref{eq:phi_y}, and \eqref{eq:phi_xy}, the above $\varepsilon^k_x$, $\varepsilon^k_y$, $\varepsilon^k_{xy}$ bounds are expressed by
\begin{align}
	\varepsilon^k_x&\bydef\varepsilon^k_\lambda|(\lambda_0)_x|+\varepsilon^k_\mu|(\mu_0)_x|=\varepsilon^k_\lambda|y_0^{-1}|+\varepsilon^k_\mu|3x_0^2y_0^{-2}|\label{eq:epsilon_x}\\
	\varepsilon^k_y&\bydef\varepsilon^k_\lambda|(\lambda_0)_y|+\varepsilon^k_\mu|(\mu_0)_y|=\varepsilon^k_\lambda|x_0y_0^{-2}|+\varepsilon^k_\mu|2x_0^3y_0^{-3}|\label{eq:epsilon_y}\\
	\varepsilon^k_{xy}&\bydef\left(\varepsilon^k_{\lambda\lambda}|(\lambda_0)_x|+\varepsilon^k_{\lambda\mu}|(\mu_0)_x|\right)|(\lambda_0)_y| + \varepsilon^k_\lambda\left|(\lambda_0)_{xy}\right|+\left(\varepsilon^k_{\lambda\mu}|(\lambda_0)_x|+\varepsilon^k_{\mu\mu}|(\mu_0)_x|\right)|(\mu_0)_y| + \varepsilon^k_\mu\left|(\mu_0)_{xy}\right|\\
	&=\left(\varepsilon^k_{\lambda\lambda}|y_0^{-1}|+\varepsilon^k_{\lambda\mu}|3x_0^2y_0^{-2}|\right)|x_0y_0^{-2}|+ \varepsilon^k_\lambda\left|y_0^{-2}\right|+\left(\varepsilon^k_{\lambda\mu}|y_0^{-1}|+\varepsilon^k_{\mu\mu}|3x_0^2y_0^{-2}|\right)|2x_0^3y_0^{-3}| + \varepsilon^k_\mu\left|6x_0^2y_0^{-3}\right|,\qquad\label{eq:epsilon_xy}
\end{align}
where $\varepsilon^k_\lambda$, $\varepsilon^k_\mu$, $\varepsilon^k_{\lambda\lambda}$, $\varepsilon^k_{\mu\lambda}$, $\varepsilon^k_{\mu\mu}$ are the truncated error bounds such that
\begin{align}
	\left|\partial_\lambda\varphi_k(\lambda_0,\mu_0)-\partial_\lambda\varphi_k^{(N)}(\lambda_0,\mu_0)\right|&\le \varepsilon^k_\lambda\label{eq:epsilon_lam}\\
	\left|\partial_\mu\varphi_k(\lambda_0,\mu_0)-\partial_\mu\varphi_k^{(N)}(\lambda_0,\mu_0)\right|&\le \varepsilon^k_\mu\label{eq:epsilon_mu}\\
	\left|\partial_{\lambda\lambda}\varphi_k(\lambda_0,\mu_0)-\partial_{\lambda\lambda}\varphi_k^{(N)}(\lambda_0,\mu_0)\right|&\le \varepsilon^k_{\lambda\lambda}\label{eq:epsilon_lamlam}\\
	\left|\partial_{\mu\lambda}\varphi_k(\lambda_0,\mu_0)-\partial_{\mu\lambda}\varphi_k^{(N)}(\lambda_0,\mu_0)\right|&\le \varepsilon^k_{\mu\lambda}\label{eq:epsilon_mulam}\\
	\left|\partial_{\mu\mu}\varphi_k(\lambda_0,\mu_0)-\partial_{\mu\mu}\varphi_k^{(N)}(\lambda_0,\mu_0)\right|&\le \varepsilon^k_{\mu\mu},\label{eq:epsilon_mumu}
\end{align}
respectively.

Therefore, using the truncated error bounds $\varepsilon^k$, $\varepsilon^k_x$, $\varepsilon^k_y$, $\varepsilon^k_{xy}$, we rigorously include the value of fundamental system of solutions for $k=1,2,3,4$
\begin{align}
	\varphi^{k}&\in\varphi_k^{(N)}(\lambda_0,\mu_0) + \varepsilon^k([-1,1]+[-1,1]\im)\\
	\varphi^{k}_x&\in\partial_x\varphi_k^{(N)}(\lambda_0,\mu_0) + \varepsilon^k_x([-1,1]+[-1,1]\im)\\
	\varphi^{k}_y&\in\partial_y\varphi_k^{(N)}(\lambda_0,\mu_0) + \varepsilon^k_y([-1,1]+[-1,1]\im)\\
	\varphi^{k}_{xy}&\in\partial_{xy}\varphi_k^{(N)}(\lambda_0,\mu_0) + \varepsilon^k_{xy}([-1,1]+[-1,1]\im).
\end{align}

The rest of this section is dedicated to the explicit construction of the bounds $\varepsilon^k$, $\varepsilon^k_\lambda$, $\varepsilon^k_\mu$, $\varepsilon^k_{\lambda\lambda}$, $\varepsilon^k_{\mu\lambda}$, and $\varepsilon^k_{\mu\mu}$ for each $k$ (=1,2,3,4).
Once we obtain these bounds, then the bounds $\varepsilon^k_x$, $\varepsilon^k_y$, and $\varepsilon^k_{xy}$ are obtained by the form \eqref{eq:epsilon_x}, \eqref{eq:epsilon_y}, and \eqref{eq:epsilon_xy}, respectively.


\subsubsection{Fundamental lemmas for truncation error bounds}\label{sec:fundamental_lemmas}
Before proceeding, we present three fundamental lemmas required for deriving the truncation error bounds. For improved readability, all proofs of these lemmas are provided in Appendix \ref{sec:appendix_proof}.

\begin{lem}\label{lem:TruncationError1}
	Let $a_{l,m}$, $b_{l,m}$, $c_{l,m}$, and $d_{l,m}$ be the coefficients for the particular solutions of \eqref{eq:PicardFuchs} defined in \eqref{eq:a_lm}, \eqref{eq:b_lm}, \eqref{eq:c_lm}, and \eqref{eq:d_lm}, respectively.
	Let $N\in\N$ be the truncated number and let
	\begin{align}\label{eq:betan}
		\beta_n\bydef 1 + \frac{3\log \left(1+\frac{1}{n}\right)}{4\log 4 + 3 + 3\log n},\quad n\ge 1.
	\end{align}
	If the variables $(\lambda,\mu)$ satisfy 
	\begin{align}\label{eq:convergencecondition2}
		\begin{cases}
			{256\beta_{\revise{N+1}}^2}\left(\lvert\lambda\rvert+\lvert\mu\rvert\right)<1\\[1mm]
			\frac{25}{64}\sqrt{1+\frac{1}{{\revise{N+1}}}}(\lvert\lambda\rvert+{\lvert\lambda\rvert}^2)/\lvert\mu\rvert<1
		\end{cases},
	\end{align}
	then the following inequalities hold:
	\begin{align}\label{eq:delta1}
		\left|\sum_{l+m\ge N+1}a_{l,m}\lambda^l \mu^m\right| &\le \frac{(4(N+1))!}{\left((N+1)!\right)^4}\frac{(|\lambda|+|\mu|)^{N+1}}{1-256(|\lambda|+|\mu|)}\bydef \delta_1(N,\lambda,\mu)\\\label{eq:delta2}
		\left|\sum_{l+m\ge N+1}a_{l,m}b_{l,m}\lambda^l \mu^m\right| &\le \frac{(4(N+1))!}{\left((N+1)!\right)^4}\left(4\log 4 + 3 + 3\log (N+1)\right)\frac{(|\lambda|+|\mu|)^{N+1}}{1-256\beta_{\revise{N+1}}(|\lambda|+|\mu|)}\bydef \delta_2(N,\lambda,\mu)\\\label{eq:delta3}
		\left|\sum_{l+m\ge N+1}a_{l,m}b_{l,m}^2\lambda^l \mu^m\right| &\le\frac{(4(N+1))!}{\left((N+1)!\right)^4}\left(4\log 4 + 3 + 3\log (N+1)\right)^2\frac{(|\lambda|+|\mu|)^{N+1}}{1-256\beta_{\revise{N+1}}^2(|\lambda|+|\mu|)}\bydef \delta_3(N,\lambda,\mu)\\\label{eq:delta4}
		\left|\sum_{l+m\ge N+1}a_{l,m}c_{l,m}\lambda^l \mu^m\right| &\le\frac{8}{3}\pi^2\frac{(4(N+1))!}{\left((N+1)!\right)^4}\frac{(|\lambda|+|\mu|)^{N+1}}{1-256(|\lambda|+|\mu|)}\bydef \delta_4(N,\lambda,\mu)\\\label{eq:delta5}
		\left|\sum_{l+m\ge N+1}d_{l+\epsilon,m}\frac{\lambda^{l+2m+\epsilon}}{\mu^{l+m+\epsilon}}\right| &\le\frac{e^3}{2\sqrt{2}\,\pi}\left|\frac{\lambda}{\mu}\right|^\epsilon\frac{\sqrt{N+1}\left[\frac{25}{64}\left(\frac{|\lambda|+|\lambda|^2}{|\mu|}\right)\right]^{N+1}}{1-\frac{25}{64}\sqrt{1+\frac{1}{\revise{N+1}}}\left(\frac{|\lambda|+|\lambda|^2}{|\mu|}\right)}\bydef \delta_5^\epsilon(N,\lambda,\mu),
	\end{align}
	where $\epsilon\in\{1/2,1\}$.
\end{lem}

\begin{lem}\label{lem:TruncationError2}
	Let $a_{l,m}$, $b_{l,m}$, $c_{l,m}$, and $d_{l,m}$ be the coefficients for the particular solutions of \eqref{eq:PicardFuchs} defined in \eqref{eq:a_lm}, \eqref{eq:b_lm}, \eqref{eq:c_lm}, and \eqref{eq:d_lm}, respectively.
	Let $N\in\N$ be the truncated number and let $\beta_{\revise{n}}$ be the same as that defined in \eqref{eq:betan}.
	Additionally, for $\epsilon\in\{1/2,1\}$, let $\gamma_{\revise{n}}$, $\eta_{\revise{n}}^\epsilon$, and $\theta_{\revise{n}}^\epsilon$ be defined by
	\begin{align}\label{eq:gam_eta_theta}
		\gamma_{\revise{n}}\bydef64\left(4+\frac1{\revise{n}}\right),\quad\eta_{\revise{n}}^\epsilon\bydef\frac{25}{64}\left(1+\frac2{2{\revise{n}}+\epsilon}\right)\sqrt{1+\frac1{\revise{n}}},~\mbox{and}\quad\theta_{\revise{n}}^\epsilon\bydef\frac{25}{64}\left(1+\frac1{{\revise{n}}+\epsilon}\right)\sqrt{1+\frac1{\revise{n}}},
	\end{align}
	respectively.
	If the variables $(\lambda,\mu)$ satisfy 
	\begin{align}\label{eq:convergencecondition3}
		\begin{cases}
			\lvert\lambda\rvert+\lvert\mu\rvert<\min\left\{1/{\gamma_{\revise{N+1}}},~1/({\beta_{\revise{N+1}}\gamma_{\revise{N+1}}}),~1/({\beta_{\revise{N+1}}^2\gamma_{\revise{N+1}}})\right\}\\
			(\lvert\lambda\rvert+{\lvert\lambda\rvert}^2)/\lvert\mu\rvert<\min\left\{1/{\eta_{\revise{N+1}}^\epsilon},~1/{\theta_{\revise{N+1}}^\epsilon}\right\},
		\end{cases}
	\end{align}
	then the following inequalities related to first order partial derivatives hold:
	\begin{align}\label{eq:delta6}
		&\left|\sum_{l+m\ge N+1}la_{l,m}\lambda^{l-1} \mu^m\right|,\qquad\left|\sum_{l+m\ge N+1}ma_{l,m}\lambda^{l} \mu^{m-1}\right|  \le \frac{(4(N+1))!}{N!\left((N+1)!\right)^3}\frac{(|\lambda|+|\mu|)^{N}}{1-\gamma_{\revise{N+1}}(|\lambda|+|\mu|)}\bydef \delta_6(N,\lambda,\mu)\\
		&\left|\sum_{l+m\ge N+1}la_{l,m}b_{l,m}\lambda^{l-1} \mu^m\right|,~\left|\sum_{l+m\ge N+1}ma_{l,m}b_{l,m}\lambda^l \mu^{m-1}\right|\\  &\qquad\qquad\qquad\qquad\qquad\qquad\le\frac{(4(N+1))!}{N!\left((N+1)!\right)^3}\frac{\left(4\log 4 + 3 + 3\log (N+1)\right)(|\lambda|+|\mu|)^{N}}{1-\beta_{\revise{N+1}}\gamma_{\revise{N+1}}(|\lambda|+|\mu|)}\bydef \delta_7(N,\lambda,\mu)\label{eq:delta7}\\
		&\left|\sum_{l+m\ge N+1}la_{l,m}b_{l,m}^2\lambda^{l-1} \mu^m\right|,~\left|\sum_{l+m\ge N+1}ma_{l,m}b_{l,m}^2\lambda^l \mu^{m-1}\right|\\ 
		&\qquad\qquad\qquad\qquad\qquad\qquad\le \frac{(4(N+1))!}{N!\left((N+1)!\right)^3}\frac{\left(4\log 4 + 3 + 3\log (N+1)\right)^2(|\lambda|+|\mu|)^{N}}{1-\beta_{\revise{N+1}}^2\gamma_{\revise{N+1}}(|\lambda|+|\mu|)}\bydef \delta_8(N,\lambda,\mu)\label{eq:delta8}\\
		&\left|\sum_{l+m\ge N+1}la_{l,m}c_{l,m}\lambda^{l-1} \mu^m\right|,~\left|\sum_{l+m\ge N+1}ma_{l,m}c_{l,m}\lambda^l \mu^{m-1}\right| \le\frac83\pi^2\frac{(4(N+1))!}{N!\left((N+1)!\right)^3}\frac{(|\lambda|+|\mu|)^{N}}{1-\gamma_{\revise{N+1}}(|\lambda|+|\mu|)}\bydef \delta_9(N,\lambda,\mu)\label{eq:delta9}\\
		&\left|\sum_{l+m\ge N+1}d_{l+\epsilon,m}(l+2m+\epsilon)\frac{\lambda^{l+2m+\epsilon-1}}{\mu^{l+m+\epsilon}}\right|\\ &\qquad\qquad\qquad\qquad\qquad\le\frac{e^3}{2\sqrt{2}\,\pi}\left|\frac{\lambda}{\mu}\right|^\epsilon|\lambda|^{-1}\frac{\sqrt{N+1}\,(2(N+1)+\epsilon)\left[\frac{25}{64}\left(\frac{|\lambda|+|\lambda|^2}{|\mu|}\right)\right]^{N+1}}{1-\eta_{\revise{N+1}}^\epsilon\left(\frac{|\lambda|+|\lambda|^2}{|\mu|}\right)}\bydef \delta_{10}^\epsilon(N,\lambda,\mu)\label{eq:delta10}\\
		&\left|\sum_{l+m\ge N+1}d_{l+\epsilon,m}(l+m+\epsilon)\frac{\lambda^{l+2m+\epsilon}}{\mu^{l+m+\epsilon+1}}\right|\\ 
		&\quad\qquad\qquad\qquad\qquad\qquad\le\frac{e^3}{2\sqrt{2}\,\pi}\left|\frac{\lambda}{\mu}\right|^\epsilon|\mu|^{-1}\frac{\sqrt{N+1}\,(N+1+\epsilon)\left[\frac{25}{64}\left(\frac{|\lambda|+|\lambda|^2}{|\mu|}\right)\right]^{N+1}}{1-\theta_{\revise{N+1}}^\epsilon\left(\frac{|\lambda|+|\lambda|^2}{|\mu|}\right)}\bydef \delta_{11}^\epsilon(N,\lambda,\mu).\label{eq:delta11}
	\end{align}
\end{lem}

\begin{lem}\label{lem:TruncationError3}
	Let $a_{l,m}$, $b_{l,m}$, $c_{l,m}$, and $d_{l,m}$ be the same coefficients as in \eqref{eq:a_lm}, \eqref{eq:b_lm}, \eqref{eq:c_lm}, and \eqref{eq:d_lm}, respectively.
	Let $N\in\N$ be the truncated number and let $\beta_{\revise{n}}$ be the same as that in \eqref{eq:betan}.
	Additionally, for $\epsilon\in\{1/2,1\}$, let $\iota_{\revise{n}}$, $\nu_{\revise{n}}^\epsilon$, $\xi_{\revise{n}}^\epsilon$, $\sigma_{\revise{n}}^\epsilon$ be defined by
	\begin{align}\label{eq:nu_N}
		\iota_{\revise{n}}\bydef64\left(4+\frac5{{\revise{n}}-1}\right),\quad\nu_{\revise{n}}^\epsilon\bydef\frac{25}{64}\left(1+\frac2{2{\revise{n}}+\epsilon}\right)\left(1+\frac2{2{\revise{n}}+\epsilon-1}\right)\sqrt{1+\frac1{\revise{n}}}
	\end{align}
	\begin{align}\label{eq:xi_N_sigma_N}
		\xi_{\revise{n}}^\epsilon\bydef\frac{25}{64}\left(1+\frac2{2{\revise{n}}+\epsilon}\right)\left(1+\frac1{{\revise{n}}+\epsilon}\right)\sqrt{1+\frac1{\revise{n}}},~\sigma_{\revise{n}}^\epsilon\bydef\frac{25}{64}\left(1+\frac1{{\revise{n}}+\epsilon}\right)\left(1+\frac1{{\revise{n}}+\epsilon+1}\right)\sqrt{1+\frac1{\revise{n}}},
	\end{align}
	respectively.
	If the variables $(\lambda,\mu)$ satisfy 
	\begin{align}\label{eq:convergencecondition4}
		\begin{cases}
			\lvert\lambda\rvert+\lvert\mu\rvert<\min\left\{1/{\iota_{\revise{N+1}}},~1/({\beta_{\revise{N+1}}\iota_{\revise{N+1}}}),~1/({\beta_{\revise{N+1}}^2\iota_{\revise{N+1}}})\right\}\\
			(\lvert\lambda\rvert+{\lvert\lambda\rvert}^2)/\lvert\mu\rvert<\min\left\{1/{\nu_{\revise{N+1}}^\epsilon},~1/{\xi_{\revise{N+1}}^\epsilon},~1/{\sigma_{\revise{N+1}}^\epsilon}\right\},
		\end{cases}
	\end{align}
	then inequalities related to second order partial derivatives are bounded by
	\begin{align}
		&\left|\sum_{l+m\ge N+1}l(l-1)a_{l,m}\lambda^{l-2} \mu^m\right|,~\left|\sum_{l+m\ge N+1}lma_{l,m}\lambda^{l-1} \mu^{m-1}\right|,~\left|\sum_{l+m\ge N+1}m(m-1)a_{l,m}\lambda^{l} \mu^{m-2}\right|\\
		&\quad\qquad\qquad\qquad\qquad\qquad\qquad\qquad\qquad\le \frac{(4(N+1))!}{(N-1)!\left((N+1)!\right)^3}\frac{(|\lambda|+|\mu|)^{N-1}}{1-\iota_{\revise{N+1}}(|\lambda|+|\mu|)}\bydef \delta_{12}(N,\lambda,\mu)\label{eq:delta12}\\
		&\left|\sum_{l+m\ge N+1}l(l-1)a_{l,m}b_{l,m}\lambda^{l-2} \mu^m\right|,~\left|\sum_{l+m\ge N+1}lma_{l,m}b_{l,m}\lambda^{l-1} \mu^{m-1}\right|,~\left|\sum_{l+m\ge N+1}m(m-1)a_{l,m}b_{l,m}\lambda^l \mu^{m-2}\right| \\
		&\qquad\qquad\qquad\qquad\le \frac{(4(N+1))!}{(N-1)!\left((N+1)!\right)^3}\frac{\left(4\log 4 + 3 + 3\log (N+1)\right)(|\lambda|+|\mu|)^{N-1}}{1-\beta_{\revise{N+1}}\iota_{\revise{N+1}}(|\lambda|+|\mu|)}\bydef \delta_{13}(N,\lambda,\mu)\label{eq:delta13}\\
		&\left|\sum_{l+m\ge N+1}l(l-1)a_{l,m}b_{l,m}^2\lambda^{l-2} \mu^m\right|,~\left|\sum_{l+m\ge N+1}lma_{l,m}b_{l,m}^2\lambda^{l-1} \mu^{m-1}\right|,~\left|\sum_{l+m\ge N+1}m(m-1)a_{l,m}b_{l,m}^2\lambda^l \mu^{m-2}\right| \\
		&\qquad\qquad\qquad\qquad\le \frac{(4(N+1))!}{(N-1)!\left((N+1)!\right)^3}\frac{\left(4\log 4 + 3 + 3\log (N+1)\right)^2(|\lambda|+|\mu|)^{N-1}}{1-\beta_{\revise{N+1}}^2\iota_{\revise{N+1}}(|\lambda|+|\mu|)}\bydef \delta_{14}(N,\lambda,\mu)\label{eq:delta14}\\
		&\left|\sum_{l+m\ge N+1}l(l-1)a_{l,m}c_{l,m}\lambda^{l-2} \mu^m\right|,~\left|\sum_{l+m\ge N+1}lma_{l,m}c_{l,m}\lambda^{l-1} \mu^{m-1}\right|,~\left|\sum_{l+m\ge N+1}m(m-1)a_{l,m}c_{l,m}\lambda^l \mu^{m-2}\right| \\
		&\qquad\qquad\qquad\qquad\qquad\qquad\qquad\qquad\le \frac{8}{3}\pi^2\frac{(4(N+1))!}{(N-1)!\left((N+1)!\right)^3}\frac{(|\lambda|+|\mu|)^{N-1}}{1-\iota_{\revise{N+1}}(|\lambda|+|\mu|)}\bydef \delta_{15}(N,\lambda,\mu)\label{eq:delta15}\\
		&\left|\sum_{l+m\ge N+1}d_{l+\epsilon,m}(l+2m+\epsilon)(l+2m+\epsilon-1)\frac{\lambda^{l+2m+\epsilon-2}}{\mu^{l+m+\epsilon}}\right| \\
		&\qquad\qquad\le \frac{e^3}{2\sqrt{2}\,\pi}\left|\frac{\lambda}{\mu}\right|^\epsilon|\lambda|^{-2}\frac{\sqrt{N+1}\,\left(2(N+1)+\epsilon\right)\left(2(N+1)+\epsilon-1\right)\left[\frac{25}{64}\left(\frac{|\lambda|+|\lambda|^2}{|\mu|}\right)\right]^{N+1}}{1-\nu_{\revise{N+1}}^\epsilon\left(\frac{|\lambda|+|\lambda|^2}{|\mu|}\right)}\bydef \delta_{16}^\epsilon(N,\lambda,\mu)\label{eq:delta16}\\
		&\left|\sum_{l+m\ge N+1}d_{l+\epsilon,m}(l+2m+\epsilon)(l+m+\epsilon)\frac{\lambda^{l+2m+\epsilon-1}}{\mu^{l+m+\epsilon+1}}\right| \\
		&\quad\qquad\qquad\qquad\le \frac{e^3}{2\sqrt{2}\,\pi}\left|\frac{\lambda}{\mu}\right|^\epsilon|\lambda\mu|^{-1}\frac{\sqrt{N+1}\,(2(N+1)+\epsilon)(N+1+\epsilon)\left[\frac{25}{64}\left(\frac{|\lambda|+|\lambda|^2}{|\mu|}\right)\right]^{N+1}}{1-\xi_{\revise{N+1}}^\epsilon\left(\frac{|\lambda|+|\lambda|^2}{|\mu|}\right)}\bydef \delta_{17}^\epsilon(N,\lambda,\mu)\label{eq:delta17}\\
		&\left|\sum_{l+m\ge N+1}d_{l+\epsilon,m}(l+m+\epsilon)(l+m+\epsilon+1)\frac{\lambda^{l+2m+\epsilon}}{\mu^{l+m+\epsilon+2}}\right| \\
		&\qquad\qquad\qquad\qquad\le \frac{e^3}{2\sqrt{2}\,\pi}\left|\frac{\lambda}{\mu}\right|^\epsilon|\mu|^{-2}\frac{\sqrt{N+1}\,(N+1+\epsilon)(N+2+\epsilon)\left[\frac{25}{64}\left(\frac{|\lambda|+|\lambda|^2}{|\mu|}\right)\right]^{N+1}}{1-\sigma_{\revise{N+1}}^\epsilon\left(\frac{|\lambda|+|\lambda|^2}{|\mu|}\right)}\bydef \delta_{18}^\epsilon(N,\lambda,\mu).\label{eq:delta18}
	\end{align}
\end{lem}

\subsubsection{Error bounds for \boldmath$\varphi_1$\boldmath}
Using Lemmas \ref{lem:TruncationError1}, \ref{lem:TruncationError2}, and \ref{lem:TruncationError3}, we derive the bounds $\varepsilon^1$, $\varepsilon^1_\lambda$, $\varepsilon^1_\mu$, $\varepsilon^1_{\lambda\lambda}$, $\varepsilon^1_{\mu\lambda}$, and $\varepsilon^1_{\mu\mu}$, which satisfy \eqref{eq:epsilon}, \eqref{eq:epsilon_lam}, \eqref{eq:epsilon_mu}, \eqref{eq:epsilon_lamlam}, \eqref{eq:epsilon_mulam}, and \eqref{eq:epsilon_mumu}, respectively, with $k=1$.


\begin{thm}\label{thm:TruncationError_phi1}
	Let the truncation number  be $N\in\N$. Denote by $\varphi_1$ the particular solution of \eqref{eq:PicardFuchs}, as defined in \eqref{eq:phi1}, and let $\varphi_1^{(N)}(\lambda_0,\mu_0)$ represent the value of the truncated function, as defined in \eqref{eq:phi1_N}, evaluated at  $(\lambda_0,\mu_0)$.
	Let $\gamma_{\revise{n}}$ and $\iota_{\revise{n}}$ be as defined in \eqref{eq:gam_eta_theta} and \eqref{eq:nu_N}, respectively.
	If $(\lambda_0,\mu_0)$ satisfy 
	\begin{align}\label{eq:convergencecondition_phi1}
		|\lambda_0|+|\mu_0|<\min\left\{\frac1{256},~\frac1{\gamma_{\revise{N+1}}},~\frac1{\iota_{\revise{N+1}}}\right\},
	\end{align}
	then the bounds $\varepsilon^1$, $\varepsilon^1_\lambda$, $\varepsilon^1_\mu$, $\varepsilon^1_{\lambda\lambda}$, $\varepsilon^1_{\lambda\mu}$, $\varepsilon^1_{\mu\mu}$  are given by
	\begin{align}
		\varepsilon^1 &\bydef \delta_1(N,\lambda_0,\mu_0)\\
		\varepsilon^1_\lambda=\varepsilon^1_\mu&\bydef \delta_6(N,\lambda_0,\mu_0)\\
		\varepsilon^1_{\lambda\lambda}=\varepsilon^1_{\mu\lambda}=\varepsilon^1_{\mu\mu}&\bydef \delta_{12}(N,\lambda_0,\mu_0).
	\end{align}
\end{thm}
\begin{proof}
	From \eqref{eq:phi1}, \eqref{eq:phi1_N}, and \eqref{eq:delta1} in Lemma~\ref{lem:TruncationError1}, the truncated error bound $\varepsilon^1$ is obtained by
	\begin{align}
		\left|\varphi^1-\varphi_1^{(N)}(\lambda_0,\mu_0)\right|=\left|\sum_{l+m\ge N+1}a_{l,m}\lambda_0^l\mu_0^m\right|\le\delta_1(N,\lambda_0,\mu_0)=\varepsilon^1.
	\end{align}

	Starting from \eqref{eq:phi_1_lam}, we derive the error bound at $(\lambda_0,\mu_0)$ using \eqref{eq:delta6} in Lemma~\ref{lem:TruncationError2}
	\begin{align}
		\left|\partial_\lambda\varphi_1(\lambda_0,\mu_0)-\partial_\lambda\varphi_1^{(N)}(\lambda_0,\mu_0)\right|=\left|\sum_{l+m\ge N+1}la_{l,m}\lambda_0^{l-1} \mu_0^m\right|\le \delta_6(N,\lambda_0,\mu_0)=\varepsilon^1_\lambda.
	\end{align}
	Similarly, the $\varepsilon^1_\mu$ bound is obtained from \eqref{eq:phi_1_mu} and \eqref{eq:delta6}
	\begin{align}
		\left|\partial_\mu\varphi_1(\lambda_0,\mu_0)-\partial_\mu\varphi_1^{(N)}(\lambda_0,\mu_0)\right|=\left|\sum_{l+m\ge N+1}ma_{l,m}\lambda_0^{l} \mu_0^{m-1}\right| \le \delta_6(N,\lambda_0,\mu_0)=\varepsilon^1_\mu.
	\end{align}
	
	Furthermore, $\varepsilon^1_{\lambda\lambda}$, $\varepsilon^1_{\mu\lambda}$, and $\varepsilon^1_{\mu\mu}$ bounds are given from \eqref{eq:phi_1_lamlam}, \eqref{eq:phi_1_mulam}, and \eqref{eq:phi_1_mumu} using \eqref{eq:delta12} in Lemma~\ref{lem:TruncationError3}
	\begin{align}
		\left|\partial_{\lambda\lambda}\varphi_1(\lambda_0,\mu_0)-\partial_{\lambda\lambda}\varphi_1^{(N)}(\lambda_0,\mu_0)\right|&=	\left|\sum_{l+m\ge N+1}l(l-1)a_{l,m}\lambda_0^{l-2} \mu_0^m\right|\le\delta_{12}(N,\lambda_0,\mu_0)=\varepsilon^1_{\lambda\lambda}\\
		\left|\partial_{\mu\lambda}\varphi_1(\lambda_0,\mu_0)-\partial_{\mu\lambda}\varphi_1^{(N)}(\lambda_0,\mu_0)\right|&=\left|\sum_{l+m\ge N+1}lma_{l,m}\lambda_0^{l-1} \mu_0^{m-1}\right|\le\delta_{12}(N,\lambda_0,\mu_0)=\varepsilon^1_{\mu\lambda}\\
		\left|\partial_{\mu\mu}\varphi_1(\lambda_0,\mu_0)-\partial_{\mu\mu}\varphi_1^{(N)}(\lambda_0,\mu_0)\right|&=\left|\sum_{l+m\ge N+1}m(m-1)a_{l,m}\lambda_0^{l} \mu_0^{m-2}\right|\le\delta_{12}(N,\lambda_0,\mu_0)=\varepsilon^1_{\mu\mu}.\qedhere
	\end{align}
\end{proof}

\subsubsection{Error bounds for \boldmath$\varphi_2$\boldmath}

The bounds $\varepsilon^2$, $\varepsilon^2_\lambda$, $\varepsilon^2_\mu$, $\varepsilon^2_{\lambda\lambda}$, $\varepsilon^2_{\mu\lambda}$, and $\varepsilon^2_{\mu\mu}$ are derived using Lemmas \ref{lem:TruncationError1}, \ref{lem:TruncationError2}, and \ref{lem:TruncationError3}, as presented in Section~\ref{sec:fundamental_lemmas}. These bounds provide rigorous estimates for the truncation errors associated with the fundamental solution $\varphi_2$ and its derivatives.

%
\begin{thm}\label{thm:TruncationError_phi2}
	Let the truncation number  be $N\in\N$. Denote by $\varphi_2$ the particular solution of \eqref{eq:PicardFuchs}, as defined in \eqref{eq:phi2}, and let $\varphi_2^{(N)}(\lambda_0,\mu_0)$ represent the value of truncated function, as defined in \eqref{eq:phi2_N}, evaluated at $(\lambda_0,\mu_0)$.
	If $(\lambda_0,\mu_0)$ satisfy the conditions \eqref{eq:convergencecondition2}, \eqref{eq:convergencecondition3}, and \eqref{eq:convergencecondition4}, then the truncation error bounds $\varepsilon^2$, $\varepsilon^2_\lambda$, $\varepsilon^2_\mu$, $\varepsilon^2_{\lambda\lambda}$, $\varepsilon^2_{\lambda\mu}$, and $\varepsilon^2_{\mu\mu}$ are given as follows:
		\begin{align}
			\varepsilon^2 &\bydef \frac1{2\pi^2}\left(\left(\log\mu_0\right)^2\delta_1+2\left|\log\mu_0\right|\delta_2+\delta_3+\delta_4+\frac12\delta_5^1\right)\\
			\varepsilon^2_\lambda&\bydef\frac1{2\pi^2}\left(\left(\log\mu_0\right)^2\delta_6+2\left|\log\mu_0\right|\delta_7+\delta_8+\delta_9+\frac12\delta_{10}^1\right)\\
			\varepsilon^2_\mu&\bydef\frac1{2\pi^2}\left(2\left|\frac{\log\mu_0}{\mu_0}\right|\delta_1+\frac2{|\mu_0|}\delta_2+\left(\log\mu_0\right)^2\delta_6+2\left|\log\mu_0\right|\delta_7+\delta_8+\delta_9+\frac12\delta_{11}^1\right)\\
			\varepsilon^2_{\lambda\lambda}&\bydef\frac1{2\pi^2}\left(\left(\log\mu_0\right)^2\delta_{12}+2\left|\log\mu_0\right|\delta_{13}+\delta_{14}+\delta_{15}+\frac12\delta_{16}^1\right)\\
			\varepsilon^2_{\mu\lambda}&\bydef\frac1{2\pi^2}\left(2\left|\frac{\log\mu_0}{\mu_0}\right|\delta_6+\frac2{|\mu_0|}\delta_7+\left(\log\mu_0\right)^2\delta_{12}+2\left|\log\mu_0\right|\delta_{13}+\delta_{14}+\delta_{15}+\frac12\delta_{17}^1\right)\\
			\varepsilon^2_{\mu\mu}&\bydef\frac1{2\pi^2}\left(\frac2{\mu_0^2}\delta_1+4\left|\frac{\log\mu_0}{\mu_0}\right|\delta_6+\frac4{|\mu_0|}\delta_7+\left(\log\mu_0\right)^2\delta_{12}+2\left|\log\mu_0\right|\delta_{13}+\delta_{14}+\delta_{15}+\frac12\delta_{18}^1\right),
		\end{align}
		where $\delta_i\equiv\delta_i(N,\lambda_0,\mu_0)$ defined in Lemmas \ref{lem:TruncationError1}, \ref{lem:TruncationError2}, and \ref{lem:TruncationError3}.
	\end{thm}
	\begin{proof}
		We begin by considering the $\varepsilon^2$ bound. From \eqref{eq:phi2}, \eqref{eq:phi2_N}, and \eqref{eq:delta1}--\eqref{eq:delta5} in Lemma~\ref{lem:TruncationError1}, the truncated error bound $\varepsilon^2$ is obtained by
		\begin{align}
			&\left|\varphi^2-\varphi_2^{(N)}(\lambda_0,\mu_0)\right|=\frac{1}{2\pi^2}\left|\sum_{l+m\ge N+1}a_{l,m}\left[(\log\mu_0+b_{l,m})^2-c_{l,m}\right]\lambda_0^l\mu_0^m+\frac{1}{2} d_{l+1,m}\frac{\lambda_0^{l+2m+1}}{\mu_0^{l+m+1}}\right|\\
			&\le\frac{1}{2\pi^2}\left(\left|\sum_{l+m\ge N+1}a_{l,m}\left(\log\mu_0\right)^2\lambda_0^l \mu_0^m\right|
			+ \left|\sum_{l+m\ge N+1}a_{l,m}(2\log\mu_0)b_{l,m}\lambda_0^l \mu_0^m\right|
			+ \left|\sum_{l+m\ge N+1}a_{l,m}b_{l,m}^2\lambda_0^l \mu_0^m\right|\right.\\
			&\hphantom{\le\frac{1}{2\pi^2}\Bigg(}\left.\quad+ \left|\sum_{l+m\ge N+1}a_{l,m}c_{l,m}\lambda_0^l \mu_0^m\right| 
			+ \left|\sum_{l+m\ge N+1}\frac12d_{l+1,m}\frac{\lambda_0^{l+2m+1}}{\mu_0^{l+m+1}}\right| 
			\right)\\
			&\le \frac1{2\pi^2}\left(\left(\log\mu_0\right)^2\delta_1(N,\lambda_0,\mu_0)+2\left|\log\mu_0\right|\delta_2(N,\lambda_0,\mu_0)+\delta_3(N,\lambda_0,\mu_0)+\delta_4(N,\lambda_0,\mu_0)+\frac12\delta_5^1(N,\lambda_0,\mu_0)\right)=\varepsilon^2.
		\end{align}
		
		Next, $\varepsilon^2_\lambda$ bound is obtained by the form \eqref{eq:phi_2_lam} and \eqref{eq:delta6}--\eqref{eq:delta10} in Lemma~\ref{lem:TruncationError2}
		\begin{align}
			&\left|\partial_\lambda\varphi_2(\lambda_0,\mu_0)-\partial_\lambda\varphi_2^{(N)}(\lambda_0,\mu_0)\right|\\
			&=
			\frac{1}{2\pi^2}\left|\sum_{l+m\ge N+1}la_{l,m}\left[(\log\mu_0+b_{l,m})^2-c_{l,m}\right]\lambda_0^{l-1}\mu_0^m+\frac{1}{2} d_{l+1,m}\left(l+2m+1\right)\frac{\lambda_0^{l+2m}}{\mu_0^{l+m+1}}\right|\\
			&\le\frac{1}{2\pi^2}\left(\left|\sum_{l+m\ge N+1}la_{l,m}\left(\log\mu_0\right)^2\lambda_0^{l-1} \mu_0^m\right|
			+ \left|\sum_{l+m\ge N+1}la_{l,m}(2\log\mu_0)b_{l,m}\lambda_0^{l-1} \mu_0^m\right|
			+ \left|\sum_{l+m\ge N+1}la_{l,m}b_{l,m}^2\lambda_0^{l-1} \mu_0^m\right|\right.\\
			&\hphantom{\le\frac{1}{2\pi^2}\Bigg(}\left.\quad+ \left|\sum_{l+m\ge N+1}la_{l,m}c_{l,m}\lambda_0^{l-1} \mu_0^m\right| 
			+ \left|\sum_{l+m\ge N+1}\frac12d_{l+1,m}(l+2m+1)\frac{\lambda_0^{l+2m}}{\mu_0^{l+m+1}}\right| 
			\right)\\
			&\le\frac1{2\pi^2}\left(\left(\log\mu_0\right)^2\delta_6(N,\lambda_0,\mu_0)+2\left|\log\mu_0\right|\delta_7(N,\lambda_0,\mu_0)+\delta_8(N,\lambda_0,\mu_0)+\delta_9(N,\lambda_0,\mu_0)+\frac12\delta_{10}^1(N,\lambda_0,\mu_0)\right)=\varepsilon^2_\lambda.
		\end{align}
		Similarly, using Lemmas~\ref{lem:TruncationError1} and \ref{lem:TruncationError2}, we have the $\varepsilon^2_\mu$ bound from \eqref{eq:phi_2_mu}, \eqref{eq:delta1}, \eqref{eq:delta2}, \eqref{eq:delta6}--\eqref{eq:delta9}, and  \eqref{eq:delta11}
		\begin{align}
			&\left|\partial_\mu\varphi_2(\lambda_0,\mu_0)-\partial_\mu\varphi_2^{(N)}(\lambda_0,\mu_0)\right|\\
			&= \frac{1}{2\pi^2}\left|\sum_{l+m\ge N+1}a_{l,m}\left\{2(\log\mu_0 + b_{l,m})+m\left[(\log\mu_0+b_{l,m})^2-c_{l,m}\right]\right\}\lambda_0^l\mu^{m-1} -\frac{1}{2} d_{l+1,m}\left(l+m+1\right)\frac{\lambda_0^{l+2m+1}}{\mu_0^{l+m+2}}\right|\\
			&\le\frac{1}{2\pi^2}\left(
			\left|\sum_{l+m\ge N+1}a_{l,m}\left(2\log\mu_0\right)\lambda_0^l \mu_0^{m-1}\right|
			+\left|\sum_{l+m\ge N+1}2a_{l,m}b_{l,m}\lambda_0^l \mu_0^{m-1}\right|
			+\left|\sum_{l+m\ge N+1}ma_{l,m}\left(\log\mu_0\right)^2\lambda_0^{l} \mu_0^{m-1}\right|\right.\\
			&\hphantom{\le\frac{1}{2\pi^2}\Bigg(}\quad+ \left|\sum_{l+m\ge N+1}ma_{l,m}(2\log\mu_0)b_{l,m}\lambda_0^{l} \mu_0^{m-1}\right|
			+ \left|\sum_{l+m\ge N+1}ma_{l,m}b_{l,m}^2\lambda_0^{l} \mu_0^{m-1}\right|\\
			&\hphantom{\le\frac{1}{2\pi^2}\Bigg(}\left.\quad+ \left|\sum_{l+m\ge N+1}ma_{l,m}c_{l,m}\lambda_0^{l} \mu_0^{m-1}\right|
			+ \left|\sum_{l+m\ge N+1}\frac12d_{l+1,m}(l+m+1)\frac{\lambda_0^{l+2m+1}}{\mu_0^{l+m+2}}\right| 
			\right)\\
			&\le \frac1{2\pi^2}\left(2\left|\frac{\log\mu_0}{\mu_0}\right|\delta_1(N,\lambda_0,\mu_0)+\frac2{|\mu_0|}\delta_2(N,\lambda_0,\mu_0)+\left(\log\mu_0\right)^2\delta_6(N,\lambda_0,\mu_0)+2\left|\log\mu_0\right|\delta_7(N,\lambda_0,\mu_0)\right.\\
			&\hphantom{\le \frac1{2\pi^2}\big(}\quad\left.+\delta_8(N,\lambda_0,\mu_0)+\delta_9(N,\lambda_0,\mu_0)+\frac12\delta_{11}^1(N,\lambda_0,\mu_0)\right)=\varepsilon^2_\mu.
		\end{align}

		Furthermore, the $\varepsilon^2_{\lambda\lambda}$ bound is derived from the form \eqref{eq:phi_2_lamlam} and \eqref{eq:delta12}--\eqref{eq:delta16} in Lemma~\ref{lem:TruncationError3}
		\begin{align}
			&\left|\partial_{\lambda\lambda}\varphi_2(\lambda_0,\mu_0)-\partial_{\lambda\lambda}\varphi_2^{(N)}(\lambda_0,\mu_0)\right|\\
			&=\frac{1}{2\pi^2}\left|\sum_{l+m\ge N+1}l(l-1)a_{l,m}\left[(\log\mu_0+b_{l,m})^2-c_{l,m}\right]\lambda_0^{l-2}\mu_0^m +\frac{1}{2} d_{l+1,m}\left(l+2m+1\right) \left(l+2m\right)\frac{\lambda_0^{l+2m-1}}{\mu_0^{l+m+1}}\right|\\
			&\le\frac{1}{2\pi^2}\left(\left|\sum_{l+m\ge N+1}l(l-1)a_{l,m}\left(\log\mu_0\right)^2\lambda_0^{l-2} \mu_0^m\right|
			+ \left|\sum_{l+m\ge N+1}l(l-1)a_{l,m}(2\log\mu_0)b_{l,m}\lambda_0^{l-2} \mu_0^m\right|\right.\\
			&\hphantom{\le\frac{1}{2\pi^2}\Bigg(}\quad+ \left|\sum_{l+m\ge N+1}l(l-1)a_{l,m}b_{l,m}^2\lambda_0^{l-2} \mu_0^m\right|
			+ \left|\sum_{l+m\ge N+1}l(l-1)a_{l,m}c_{l,m}\lambda_0^{l-2} \mu_0^m\right| \\
			&\hphantom{\le\frac{1}{2\pi^2}\Bigg(}\left.\quad+ \left|\sum_{l+m\ge N+1}\frac12d_{l+1,m}(l+2m+1)(l+2m)\frac{\lambda_0^{l+2m-1}}{\mu_0^{l+m+1}}\right| \right)\\
			&\le \frac1{2\pi^2}\left(\left(\log\mu_0\right)^2\delta_{12}(N,\lambda_0,\mu_0)+2\left|\log\mu_0\right|\delta_{13}(N,\lambda_0,\mu_0)+\delta_{14}(N,\lambda_0,\mu_0)+\delta_{15}(N,\lambda_0,\mu_0)+\frac12\delta_{16}^1(N,\lambda_0,\mu_0)\right)\\
			&=\varepsilon^2_{\lambda\lambda}.
		\end{align}
		A similar argument yields the $\varepsilon^2_{\mu\lambda}$ bound, derived from the form \eqref{eq:phi_2_mulam}, \eqref{eq:delta6} and \eqref{eq:delta7} in Lemma~\ref{lem:TruncationError2}, and \eqref{eq:delta12}--\eqref{eq:delta15} and \eqref{eq:delta17} ($\epsilon=1$) in Lemma~\ref{lem:TruncationError3}.
		\begin{align}
			&\left|\partial_{\mu\lambda}\varphi_2(\lambda_0,\mu_0)-\partial_{\mu\lambda}\varphi_2^{(N)}(\lambda_0,\mu_0)\right|\\
			&=\frac{1}{2\pi^2}\left|\sum_{l+m\ge N+1}la_{l,m}\left\{2(\log\mu_0 + b_{l,m})+m\left[(\log\mu_0+b_{l,m})^2-c_{l,m}\right]\right\}\lambda_0^{l-1}\mu_0^{m-1} \right.\\
			&\hphantom{=\frac{1}{2\pi^2}\Bigg|\sum_{l,m\ge0}}\left.\quad -\frac{1}{2} d_{l+1,m}\left(l+m+1\right) \left(l+2m+1\right)\frac{\lambda_0^{l+2m}}{\mu_0^{l+m+2}}\right|\\
			&\le\frac{1}{2\pi^2}\left(
			\left|\sum_{l+m\ge N+1}la_{l,m}\left(2\log\mu_0\right)\lambda_0^{l-1} \mu_0^{m-1}\right|
			+\left|\sum_{l+m\ge N+1}2la_{l,m}b_{l,m}\lambda_0^{l-1} \mu_0^{m-1}\right|\right.\\
			&\hphantom{\le\frac{1}{2\pi^2}\Bigg(}\quad+\left|\sum_{l+m\ge N+1}lma_{l,m}\left(\log\mu_0\right)^2\lambda_0^{l-1} \mu_0^{m-1}\right|
			+ \left|\sum_{l+m\ge N+1}lma_{l,m}(2\log\mu_0)b_{l,m}\lambda_0^{l-1} \mu_0^{m-1}\right|\\
			&\hphantom{\le\frac{1}{2\pi^2}\Bigg(}\quad+ \left|\sum_{l+m\ge N+1}lma_{l,m}b_{l,m}^2\lambda_0^{l-1} \mu_0^{m-1}\right|
			+ \left|\sum_{l+m\ge N+1}lma_{l,m}c_{l,m}\lambda_0^{l-1} \mu_0^{m-1}\right|\\
			&\hphantom{\le\frac{1}{2\pi^2}\Bigg(}\left.\quad
			+ \left|\sum_{l+m\ge N+1}\frac12d_{l+1,m}(l+m+1)(l+2m+1)\frac{\lambda_0^{l+2m}}{\mu_0^{l+m+2}}\right| 
			\right)\\
			&\le \frac1{2\pi^2}\left(2\left|\frac{\log\mu_0}{\mu_0}\right|\delta_6(N,\lambda_0,\mu_0)+\frac2{|\mu_0|}\delta_7(N,\lambda_0,\mu_0)+\left(\log\mu_0\right)^2\delta_{12}(N,\lambda_0,\mu_0)+2\left|\log\mu_0\right|\delta_{13}(N,\lambda_0,\mu_0)\right.\\
			&\hphantom{\le \frac1{2\pi^2}\Bigg(}\left.+\delta_{14}(N,\lambda_0,\mu_0)+\delta_{15}(N,\lambda_0,\mu_0)+\frac12\delta_{17}^1(N,\lambda_0,\mu_0)\right)=\varepsilon^2_{\mu\lambda}.
		\end{align}
		Finally, the $\varepsilon^2_{\mu\mu}$ bound is derived using \eqref{eq:phi_2_mumu}, \eqref{eq:delta1} in Lemma~\ref{lem:TruncationError1}, \eqref{eq:delta6} and \eqref{eq:delta7} in Lemma~\ref{lem:TruncationError2}, 
		and \eqref{eq:delta12}--\eqref{eq:delta15} and \eqref{eq:delta18} ($\epsilon=1$) in Lemma~\ref{lem:TruncationError3}.
		\begin{align}
			&\left|\partial_{\mu\mu}\varphi_2(\lambda_0,\mu_0)-\partial_{\mu\mu}\varphi_2^{(N)}(\lambda_0,\mu_0)\right|\\
			&=\frac{1}{2\pi^2}\left|\sum_{l+m\ge N+1} a_{l,m}\left\{2+(4m-2)(\log\mu_0+b_{l,m})+m(m-1)\left[(\log\mu_0+b_{l,m})^2-c_{l,m}\right]\right\}\lambda_0^l\mu_0^{m-2}\right.\\
			&\hphantom{=\frac{1}{2\pi^2}\Bigg|\sum_{l,m\ge0}}\left. +\frac{1}{2} d_{l+1,m} \left(l+m+1\right) \left(l+m+2\right)\frac{\lambda_0^{l+2m+1}}{\mu_0^{l+m+3}}\right|\\
			&\le\frac{1}{2\pi^2}\left(
			\left|\sum_{l+m\ge N+1}2a_{l,m}\lambda_0^{l} \mu_0^{m-2}\right|
			+\left|\sum_{l+m\ge N+1}ma_{l,m}\left(4\log\mu_0\right)\lambda_0^{l} \mu_0^{m-2}\right|
			+\left|\sum_{l+m\ge N+1}4ma_{l,m}b_{l,m}\lambda_0^{l} \mu_0^{m-2}\right|\right.\\
			&\hphantom{\le\frac{1}{2\pi^2}\Bigg(}\quad+\left|\sum_{l+m\ge N+1}m(m-1)a_{l,m}\left(\log\mu_0\right)^2\lambda_0^{l} \mu_0^{m-2}\right|
			+ \left|\sum_{l+m\ge N+1}m(m-1)a_{l,m}(2\log\mu_0)b_{l,m}\lambda_0^{l} \mu_0^{m-2}\right|\\
			&\hphantom{\le\frac{1}{2\pi^2}\Bigg(}\quad+ \left|\sum_{l+m\ge N+1}m(m-1)a_{l,m}b_{l,m}^2\lambda_0^{l} \mu_0^{m-2}\right|
			+ \left|\sum_{l+m\ge N+1}m(m-1)a_{l,m}c_{l,m}\lambda_0^{l} \mu_0^{m-2}\right|\\
			&\hphantom{\le\frac{1}{2\pi^2}\Bigg(}\left.\quad
			+ \left|\sum_{l+m\ge N+1}\frac12d_{l+1,m}(l+m+1)(l+m+2)\frac{\lambda_0^{l+2m+1}}{\mu_0^{l+m+3}}\right| \right)\\
			&\le\frac1{2\pi^2}\left(\frac2{\mu_0^2}\delta_1(N,\lambda_0,\mu_0)+4\left|\frac{\log\mu_0}{\mu_0}\right|\delta_6(N,\lambda_0,\mu_0)+\frac4{|\mu_0|}\delta_7(N,\lambda_0,\mu_0)+\left(\log\mu_0\right)^2\delta_{12}(N,\lambda_0,\mu_0)\right.\\
			&\hphantom{\le \frac1{2\pi^2}\Bigg(}\left.+2\left|\log\mu_0\right|\delta_{13}(N,\lambda_0,\mu_0)+\delta_{14}(N,\lambda_0,\mu_0)+\delta_{15}(N,\lambda_0,\mu_0)+\frac12\delta_{18}^1(N,\lambda_0,\mu_0)\right)=\varepsilon^2_{\mu\mu}.\qedhere
		\end{align}
	\end{proof}

	\subsubsection{Error bounds for \boldmath$\varphi_3$\boldmath}
	The truncation errors $\varepsilon^3$, $\varepsilon^3_\lambda$, $\varepsilon^3_\mu$, $\varepsilon^3_{\lambda\lambda}$, $\varepsilon^3_{\mu\lambda}$, and $\varepsilon^3_{\mu\mu}$ are explicitly given from Lemmas \ref{lem:TruncationError1}, \ref{lem:TruncationError2}, and \ref{lem:TruncationError3}, detailed in Section~\ref{sec:fundamental_lemmas}. These errors quantify the bounds for the truncation of the fundamental solution $\varphi_3$ and its derivatives.
	
	\begin{thm}\label{thm:TruncationError_phi3}
		Let $N \in \mathbb{N}$ denote the truncation number. Let $\varphi_3$ be the fundamental solution of \eqref{eq:PicardFuchs} as defined in \eqref{eq:phi3} and its truncated approximation at $(\lambda_0,\mu_0)$ denotes $\varphi_3^{(N)}(\lambda_0,\mu_0)$, which is given in \eqref{eq:phi3_N}.
		Let $\eta_{\revise{n}}^{1/2}$, $\theta_{\revise{n}}^{1/2}$, $\nu_{\revise{n}}^{1/2}$, $\xi_{\revise{n}}^{1/2}$, and $\sigma_{\revise{n}}^{1/2}$ be as defined in \eqref{eq:gam_eta_theta}, \eqref{eq:nu_N}, and \eqref{eq:xi_N_sigma_N}.
		Suppose $(\lambda_0,\mu_0)$ satisfy
		\begin{align}
			\frac{\lvert\lambda_0\rvert+{\lvert\lambda_0\rvert}^2}{\lvert\mu_0\rvert}<\min\left\{
			\frac1{\frac{25}{64}\sqrt{1+\frac{1}{{\revise{N+1}}}}},~\frac1{\eta_{\revise{N+1}}^{1/2}},~\frac1{\theta_{\revise{N+1}}^{1/2}},~\frac1{\nu_{\revise{N+1}}^{1/2}},~\frac1{\xi_{\revise{N+1}}^{1/2}},~\frac1{\sigma_{\revise{N+1}}^{1/2}}
			\right\},
		\end{align}	
		then the truncation error bounds $\varepsilon^3$, $\varepsilon^3_\lambda$, $\varepsilon^3_\mu$, $\varepsilon^3_{\lambda\lambda}$, $\varepsilon^3_{\lambda\mu}$, and $\varepsilon^3_{\mu\mu}$ are explicitly given by
		\begin{align}
			\varepsilon^3 &\bydef \frac1{4\pi^2}\delta_5^{1/2}(N,\lambda_0,\mu_0),
			&\varepsilon^3_\lambda &\bydef \frac1{4\pi^2}\delta_{10}^{1/2}(N,\lambda_0,\mu_0),
			&\varepsilon^3_\mu &\bydef \frac1{4\pi^2}\delta_{11}^{1/2}(N,\lambda_0,\mu_0)\\
			\varepsilon^3_{\lambda\lambda} &\bydef \frac1{4\pi^2}\delta_{16}^{1/2}(N,\lambda_0,\mu_0),
			&\varepsilon^3_{\mu\lambda} &\bydef \frac1{4\pi^2}\delta_{17}^{1/2}(N,\lambda_0,\mu_0),
			&\varepsilon^3_{\mu\mu} &\bydef \frac1{4\pi^2}\delta_{18}^{1/2}(N,\lambda_0,\mu_0).
		\end{align}
	\end{thm}
	\begin{proof}
		From \eqref{eq:phi3}, \eqref{eq:phi3_N}, and \eqref{eq:delta5} in Lemma~\ref{lem:TruncationError1}, the truncated error bound $\varepsilon^3$ is obtained by
		\begin{align}
			\left|\varphi^3-\varphi_3^{(N)}(\lambda_0,\mu_0)\right|=\frac1{4\pi^2}\left|\sum_{l+m\ge N+1}
			d_{l+\frac{1}{2},m}\frac{\lambda^{l+2m+\frac{1}{2}}}{\mu^{l+m+\frac{1}{2}}}
			\right|\le\frac1{4\pi^2}\delta_5^{1/2}(N,\lambda_0,\mu_0)=\varepsilon^3.
		\end{align}

		From \eqref{eq:phi_3_lam}, we derive the error bound at $(\lambda_0,\mu_0)$ using \eqref{eq:delta10} in Lemma~\ref{lem:TruncationError2}
		\begin{align}
			\left|\partial_\lambda\varphi_3(\lambda_0,\mu_0)-\partial_\lambda\varphi_3^{(N)}(\lambda_0,\mu_0)\right|=\frac1{4\pi^2}\left|\sum_{l+m\ge N+1}ld_{l+\frac{1}{2},m}\left(l+2m+\frac{1}{2}\right)\frac{\lambda^{l+2m-\frac{1}{2}}}{\mu^{l+m+\frac{1}{2}}}\right|\le \frac1{4\pi^2}\delta_{10}^{1/2}(N,\lambda_0,\mu_0)=\varepsilon^3_\lambda.
		\end{align}
		Similarly, the $\varepsilon^3_\mu$ bound is obtained from \eqref{eq:phi_3_mu} and \eqref{eq:delta11}
		\begin{align}
			\left|\partial_\mu\varphi_3(\lambda_0,\mu_0)-\partial_\mu\varphi_3^{(N)}(\lambda_0,\mu_0)\right|=\frac1{4\pi^2}\left|\sum_{l+m\ge N+1}d_{l+\frac{1}{2},m}\left(l+m+\frac{1}{2}\right)\frac{\lambda^{l+2m+\frac{1}{2}}}{\mu^{l+m+\frac{3}{2}}}\right| \le \frac1{4\pi^2}\delta_{11}^{1/2}(N,\lambda_0,\mu_0)=\varepsilon^3_\mu.
		\end{align}
		
		Furthermore, the bounds $\varepsilon^3_{\lambda\lambda}$, $\varepsilon^3_{\mu\lambda}$, and $\varepsilon^3_{\mu\mu}$ are derived from \eqref{eq:phi_3_lamlam}, \eqref{eq:phi_3_mulam}, and \eqref{eq:phi_3_mumu}, respectively, using \eqref{eq:delta16}, \eqref{eq:delta17}, and \eqref{eq:delta17} in Lemma~\ref{lem:TruncationError3}.
		\begin{align}
			\left|\partial_{\lambda\lambda}\varphi_3(\lambda_0,\mu_0)-\partial_{\lambda\lambda}\varphi_3^{(N)}(\lambda_0,\mu_0)\right|&=	\frac1{4\pi^2}\left|\sum_{l+m\ge N+1}d_{l+\frac{1}{2},m}\left(l+2m+\frac{1}{2}\right)\left(l+2m-\frac{1}{2}\right)\frac{\lambda^{l+2m-\frac{3}{2}}}{\mu^{l+m+\frac{1}{2}}}\right|\\
			&\le\frac1{4\pi^2}\delta_{16}^{1/2}(N,\lambda_0,\mu_0)=\varepsilon^3_{\lambda\lambda}\\
			\left|\partial_{\mu\lambda}\varphi_3(\lambda_0,\mu_0)-\partial_{\mu\lambda}\varphi_3^{(N)}(\lambda_0,\mu_0)\right|&=\frac1{4\pi^2}\left|\sum_{l+m\ge N+1}d_{l+\frac{1}{2},m}\left(l+2m+\frac{1}{2}\right)\left(l+m+\frac{1}{2}\right)\frac{\lambda^{l+2m-\frac{1}{2}}}{\mu^{l+m+\frac{3}{2}}}\right|\\
			&\le\frac1{4\pi^2}\delta_{17}^{1/2}(N,\lambda_0,\mu_0)=\varepsilon^3_{\mu\lambda}\\
			\left|\partial_{\mu\mu}\varphi_3(\lambda_0,\mu_0)-\partial_{\mu\mu}\varphi_3^{(N)}(\lambda_0,\mu_0)\right|&=\frac1{4\pi^2}\left|\sum_{l+m\ge N+1}d_{l+\frac{1}{2},m}\left(l+m+\frac{1}{2}\right)\left(l+m+\frac{3}{2}\right)\frac{\lambda^{l+2m+\frac{1}{2}}}{\mu^{l+m+\frac{5}{2}}}\right|\\
			&\le\frac1{4\pi^2}\delta_{12}(N,\lambda_0,\mu_0)=\varepsilon^3_{\mu\mu}.\qedhere
		\end{align}
	\end{proof}

	\subsubsection{Error bounds for \boldmath$\varphi_4$\boldmath}
	Finally, the truncation errors for  the fundamental solution $\varphi_4$, namely $\varepsilon^4$, $\varepsilon^4_\lambda$, $\varepsilon^4_\mu$, $\varepsilon^4_{\lambda\lambda}$, $\varepsilon^4_{\mu\lambda}$, and $\varepsilon^4_{\mu\mu}$, are derived from Lemmas \ref{lem:TruncationError1}, \ref{lem:TruncationError2}, and \ref{lem:TruncationError3} in Section~\ref{sec:fundamental_lemmas}. These bounds explicitly provide precise estimates for the truncation of $\varphi_4$ and its derivatives.
	
	\begin{thm}\label{thm:TruncationError_phi4}
		Let $N \in \mathbb{N}$ be the truncation number. The fundamental solution of \eqref{eq:PicardFuchs}, denoted by $\varphi_4$, is defined in \eqref{eq:phi4}, and its truncated approximation at $(\lambda_0,\mu_0)$ is represented by $\varphi_4^{(N)}(\lambda_0,\mu_0)$ given in \eqref{eq:phi4_N}.
		If $(\lambda_0,\mu_0)$ satisfy the conditions \eqref{eq:convergencecondition2}, \eqref{eq:convergencecondition3}, and \eqref{eq:convergencecondition4}, then the truncation error bounds $\varepsilon^4$, $\varepsilon^4_\lambda$, $\varepsilon^4_\mu$, $\varepsilon^4_{\lambda\lambda}$, $\varepsilon^4_{\lambda\mu}$, and $\varepsilon^4_{\mu\mu}$ are explicitly given as follows:
		\begin{align}
			\varepsilon^4 &\bydef \frac1{2\pi}\left(|\log\mu_0|\delta_1(N,\lambda_0,\mu_0) + \delta_2(N,\lambda_0,\mu_0)\right)\\
			\varepsilon^4_\lambda &\bydef \frac1{2\pi}\left(|\log\mu_0|\delta_6(N,\lambda_0,\mu_0) + \delta_7(N,\lambda_0,\mu_0)\right)\\
			\varepsilon^4_\mu &\bydef \frac1{2\pi}\left(\frac1{|\mu_0|}\delta_1(N,\lambda_0,\mu_0) + |\log\mu_0|\delta_6(N,\lambda_0,\mu_0) + \delta_7(N,\lambda_0,\mu_0)\right)\\
			\varepsilon^4_{\lambda\lambda} &\bydef \frac1{2\pi}\left(|\log\mu_0|\delta_{12}(N,\lambda_0,\mu_0) + \delta_{13}(N,\lambda_0,\mu_0)\right)\\
			\varepsilon^4_{\mu\lambda} &\bydef \frac1{2\pi}\left(\frac1{|\mu_0|}\delta_6(N,\lambda_0,\mu_0) + |\log\mu_0|\delta_{12}(N,\lambda_0,\mu_0) + \delta_{13}(N,\lambda_0,\mu_0)\right)\\
			\varepsilon^4_{\mu\mu} &\bydef \frac1{2\pi}\left(\frac2{|\mu_0|}\delta_6(N,\lambda_0,\mu_0) + |\log\mu_0|\delta_{12}(N,\lambda_0,\mu_0) + \delta_{13}(N,\lambda_0,\mu_0)\right).
		\end{align}
	\end{thm}
	\begin{proof}
		Let us begin by considering the $\varepsilon^4$ bound. From \eqref{eq:phi4}, \eqref{eq:phi4_N}, the truncated error bound $\varepsilon^4$ is obtained, using \eqref{eq:delta1} and \eqref{eq:delta2} in Lemma~\ref{lem:TruncationError1}, by
		\begin{align}
			\left|\varphi^4-\varphi_4^{(N)}(\lambda_0,\mu_0)\right|&=\frac1{2\pi}\left|\sum_{l+m\ge N+1}a_{l,m}(\log\mu_0+b_{l,m})\lambda_0^l\mu_0^m\right|\\
			&\le\frac1{2\pi}\left(\left|\sum_{l+m\ge N+1}a_{l,m}\left(\log\mu_0\right)\lambda_0^l \mu_0^m\right|
			+ \left|\sum_{l+m\ge N+1}a_{l,m}b_{l,m}\lambda_0^l \mu_0^m\right|\right)\\
			&\le \frac1{2\pi}\left(|\log\mu_0|\delta_1(N,\lambda_0,\mu_0) + \delta_2(N,\lambda_0,\mu_0)\right)=\varepsilon^4.
		\end{align}
		
		Next, $\varepsilon^4_\lambda$ bound is obtained by the form \eqref{eq:phi_4_lam}, \eqref{eq:delta6} and \eqref{eq:delta7} in Lemma~\ref{lem:TruncationError2}
		\begin{align}
			\left|\partial_\lambda\varphi_4(\lambda_0,\mu_0)-\partial_\lambda\varphi_4^{(N)}(\lambda_0,\mu_0)\right|
			&=
			\frac1{2\pi}\left|\sum_{l+m\ge N+1}la_{l,m}(\log\mu_0+b_{l,m})\lambda_0^{l-1}\mu_0^m\right|\\
			&\le\frac1{2\pi}\left(\left|\sum_{l+m\ge N+1}la_{l,m}\left(\log\mu_0\right)\lambda_0^{l-1} \mu_0^m\right|
			+ \left|\sum_{l+m\ge N+1}la_{l,m}b_{l,m}\lambda_0^{l-1} \mu_0^m\right|
			\right)\\
			&\le\frac1{2\pi}\left(|\log\mu_0|\delta_6(N,\lambda_0,\mu_0) + \delta_7(N,\lambda_0,\mu_0)\right)=\varepsilon^4_\lambda.
		\end{align}
		Similarly, using Lemmas~\ref{lem:TruncationError1} and \ref{lem:TruncationError2}, we have the $\varepsilon^4_\mu$ bound from \eqref{eq:phi_4_mu}, \eqref{eq:delta1}, \eqref{eq:delta6} and \eqref{eq:delta7}
		\begin{align}
			&\left|\partial_\mu\varphi_4(\lambda_0,\mu_0)-\partial_\mu\varphi_4^{(N)}(\lambda_0,\mu_0)\right|\\
			&= \frac1{2\pi}\left|\sum_{l+m\ge N+1}a_{l,m} \left[1+m(\log\mu_0+b_{l,m})\right]\lambda_0^l\mu_0^{m-1}\right|\\
			&\le\frac1{2\pi}\left(
			\left|\sum_{l+m\ge N+1}a_{l,m}\lambda_0^l \mu_0^{m-1}\right|
			+\left|\sum_{l+m\ge N+1}ma_{l,m}(\log\mu_0)\lambda_0^l \mu_0^{m-1}\right|
			+\left|\sum_{l+m\ge N+1}ma_{l,m}b_{l,m}\lambda_0^l \mu_0^{m-1}\right|
			\right)\\
			&\le \frac1{2\pi}\left(\frac1{|\mu_0|}\delta_1(N,\lambda_0,\mu_0) + |\log\mu_0|\delta_6(N,\lambda_0,\mu_0) + \delta_7(N,\lambda_0,\mu_0)\right)=\varepsilon^4_\mu.
		\end{align}

		Furthermore, the bounds $\varepsilon^4_{\lambda\lambda}$, $\varepsilon^4_{\mu\lambda}$, and $\varepsilon^4_{\mu\mu}$ are derived from \eqref{eq:phi_4_lamlam}, \eqref{eq:phi_4_mulam}, and \eqref{eq:phi_4_mumu}, respectively, using \eqref{eq:delta6} in Lemma~\ref{lem:TruncationError2}, \eqref{eq:delta12} and \eqref{eq:delta13} in Lemma~\ref{lem:TruncationError3}. It follows that
		\begin{align}
			&\left|\partial_{\lambda\lambda}\varphi_4(\lambda_0,\mu_0)-\partial_{\lambda\lambda}\varphi_4^{(N)}(\lambda_0,\mu_0)\right|\\
			&=\frac1{2\pi}\left|\sum_{l+m\ge N+1}l(l-1)a_{l,m}(\log\mu_0+b_{l,m})\lambda_0^{l-2}\mu_0^m\right|\\
			&\le\frac1{2\pi}\left(\left|\sum_{l+m\ge N+1}l(l-1)a_{l,m}\left(\log\mu_0\right)\lambda_0^{l-2} \mu_0^m\right|
			+ \left|\sum_{l+m\ge N+1}l(l-1)a_{l,m}b_{l,m}\lambda_0^{l-2} \mu_0^m\right| \right)\\
			&\le \frac1{2\pi^2}\left(|\log\mu_0|\delta_{12}(N,\lambda_0,\mu_0) + \delta_{13}(N,\lambda_0,\mu_0)\right)=\varepsilon^4_{\lambda\lambda},
		\end{align}
		\begin{align}
			&\left|\partial_{\mu\lambda}\varphi_4(\lambda_0,\mu_0)-\partial_{\mu\lambda}\varphi_4^{(N)}(\lambda_0,\mu_0)\right|\\
			&=\frac1{2\pi}\left|\sum_{l+m\ge N+1}la_{l,m} \left[1+m(\log\mu_0+b_{l,m})\right]\lambda_0^{l-1}\mu_0^{m-1}\right|\\
			&\le\frac1{2\pi}\left(
			\left|\sum_{l+m\ge N+1}la_{l,m}\lambda_0^{l-1} \mu_0^{m-1}\right|
			+ \left|\sum_{l+m\ge N+1}lma_{l,m}(\log\mu_0)b_{l,m}\lambda_0^{l-1} \mu_0^{m-1}\right|
			+ \left|\sum_{l+m\ge N+1}lma_{l,m}b_{l,m}\lambda_0^{l-1} \mu_0^{m-1}\right|
			\right)\\
			&\le \frac1{2\pi}\left(\frac1{|\mu_0|}\delta_6(N,\lambda_0,\mu_0) + |\log\mu_0|\delta_{12}(N,\lambda_0,\mu_0) + \delta_{13}(N,\lambda_0,\mu_0)\right)=\varepsilon^4_{\mu\lambda},
		\end{align}
		and
		\begin{align}
			&\left|\partial_{\mu\mu}\varphi_4(\lambda_0,\mu_0)-\partial_{\mu\mu}\varphi_4^{(N)}(\lambda_0,\mu_0)\right|\\
			&=\frac1{2\pi}\left|\sum_{l+m\ge N+1}a_{l,m}\left\{(m-1)+m\left[1+(m-1)(\log\mu_0+b_{l,m})\right]\right\}\lambda_0^l\mu_0^{m-2}\right|\\
			&\le\frac1{2\pi}\left(
			\left|\sum_{l+m\ge N+1}2ma_{l,m}\lambda_0^{l} \mu_0^{m-2}\right|
			+\left|\sum_{l+m\ge N+1}m(m-1)a_{l,m}\left(\log\mu_0\right)\lambda_0^{l} \mu_0^{m-2}\right|\right.\\
			&\hphantom{\le\frac1{2\pi}\Bigg(}\quad\left.+ \left|\sum_{l+m\ge N+1}m(m-1)a_{l,m}b_{l,m}\lambda_0^{l} \mu_0^{m-2}\right|\right)\\
			&\le\frac1{2\pi}\left(\frac2{|\mu_0|}\delta_6(N,\lambda_0,\mu_0) + |\log\mu_0|\delta_{12}(N,\lambda_0,\mu_0) + \delta_{13}(N,\lambda_0,\mu_0)\right)=\varepsilon^4_{\mu\mu}.\qedhere
		\end{align}
	\end{proof}
	
	\subsection{Rigorously finding monodromy using analytic continuation}\label{sec:CAPmethod}
	%

	To sum up the above, this section outlines our method of computer-assisted proofs for computing the monodromy matrices of the Picard--Fuchs differential equation \eqref{eq:PicardFuchs} by rigorously integrating the Pfaffian equation \eqref{eq:diff_eq}, corresponding to rigorous analytic continuation along the prescribed contour. 
	The first step in this procedure begins with computing the fundamental system of solutions \( \Phi(p_0) \) at the base point \( p_0 \), as defined in \eqref{eq:base_pt}. 
	This computation is performed rigorously using interval arithmetic to account for rounding errors when evaluating finite series of the truncated functions $\varphi_k^{(N)}$ ($k=1,2,3,4$). The truncation error is then estimated using Theorems~\ref{thm:TruncationError_phi1}, \ref{thm:TruncationError_phi2}, \ref{thm:TruncationError_phi3}, and \ref{thm:TruncationError_phi4}. Combining these ensures that the value of the fundamental system of solutions is both rigorous and reliable.
	Next, we choose a closed loop \( \Sigma_i \) as presented in Section~\ref{sec:path_contour}, which encircles a singular point \( p_i \) ($i=1,\dots,6$).  As mentioned in Section~\ref{sec:path_contour}, \( \Sigma_i \) is typically represented parametrically, such as a circular arc or another simple path in the complex plane.
	In the third step, the fundamental system of solutions is analytically continued along the loop \( \Sigma_i \) using a rigorous numerical integrator. The initial condition is set as the identity matrix, say \( \mathrm{Id} \), indicating that no transformation has occurred at the start of the loop.
	The result of this step is the analytically continued fundamental system of solutions at the end of the loop, denoted by \( (\Sigma_i)_\ast \mathrm{Id} \).
	
	Once the analytic continuation along the loop \( \Sigma_i \) is complete, the monodromy matrix \( M_{\Sigma_i} \) is computed using the conjugacy formula:
	\[
	M_{\Sigma_i} = \Phi(p_0)^{-1} \left( (\Sigma_i)_\ast \mathrm{Id} \right) \Phi(p_0).
	\]
	This formula, implemented with interval arithmetic, rigorously enclose the transformation of the fundamental system of solutions resulting from the analytic continuation around the singular point.
	For each singular point $p_i$,  the loop \( \Sigma_i \) generates a corresponding monodromy matrix \( M_{\Sigma_i} \).
	Together, these matrices constitute the monodromy group, capturing the behavior of solutions under analytic continuation.

	This procedure enables the rigorous computation of monodromy matrices, which is a generator of Monodromy group for differential equations.
	The most important thing to emphasize here is that, by combining analytical techniques with validated numerics, our method of computer-assisted proofs is possible to systematically analyze the behavior of solutions and their transformations, providing a robust framework for studying the monodromy.

\section{Computational results}\label{sec:results}

In this section, we present the proof of our main result, namely Theorem~\ref{thm-main-result}. Detailed results are provided for the path $\Sigma_1$, as the results for the other paths are nearly identical to those presented here. For the remaining results on each path, we refer to the publicly available implementation code \cite{code}. The implementation is based on the kv library \cite{kv}, written in the C++ language.

First of all, the values of the fundamental system of solutions defined in Section~\ref{sec:fundamental_sol} at the base point $p_0$ are rigorously computed by the two-step procedure presented in Section~\ref{sec:Rig_inclution_fund_sol}. The truncation number is set to $N=41$. Using interval arithmetic and the truncation error bounds provided in Theorems~\ref{thm:TruncationError_phi1}, \ref{thm:TruncationError_phi2}, \ref{thm:TruncationError_phi3}, and \ref{thm:TruncationError_phi4}, the rigorous inclusion of each fundamental solution value is obtained as follows:
\begin{align}
	&\Phi(p_0)=\begin{bmatrix}
		\varphi^1 & \varphi^2 & \varphi^3 &\varphi^4\\[2pt]
		\varphi_{x}^1 & \varphi_{x}^2 & \varphi_{x}^3 &\varphi_{x}^4\\[2pt]
		\varphi_{y}^1 & \varphi_{y}^2 & \varphi_{y}^3 &\varphi_{y}^4\\[2pt]
		\varphi_{xy}^1 & \varphi_{xy}^2 & \varphi_{xy}^3 &\varphi_{xy}^4
	\end{bmatrix}\\
	&\in\begin{bmatrix}
		1.02865241561_{7963}^{8159} & 2.39217617242_{4841}^{5186} & [-4.3,4.4]\cdot 10^{-17} + 0.153171225738_{1849}^{2252}\,\im & 1.114490740819_{042}^{171}\,\im\\
		-33.248652909_{20167}^{04933} & 249.082552_{2945907}^{3001465} & [-5.3,5.1]\cdot 10^{-13}  -1165.9899089_{20037}^{19687}\,\im & -40.49146992226_{83}^{318}\,\im\\
		-8.7561501435_{48835}^{11071} & 73.21158171_{465447}^{603159} & [-1.3,1.3]\cdot 10^{-13} -289.2333425644_{927}^{053}\,\im& -7.68095985658_{7176}^{6149}\,\im\\
		288.265_{6899693559}^{7024336106} & 175537.369_{3964952}^{8700323} & [-1.5,1.4]\cdot 10^{-9} -2770453.21111_{7561}^{6786}\,\im& 263.835019245_{0897}^{3826}\,\im
	\end{bmatrix}.
\end{align}
\begin{rem}
	We can directly verify that the values of $\varphi_1$, $\varphi_2$, and $\varphi_4$ are either real or purely imaginary,as determined from their series representations in \eqref{eq:phi1}, \eqref{eq:phi2}, and \eqref{eq:phi4}, respectively. Similarly, $\varphi_3$  can be identified as purely imaginary; however, due to rounding errors in evaluating finite series, a tiny error appears in its real part.
\end{rem}

Next, as described in Section~\ref{sec:singular_points}, we identify six singular points $p_i$ ($i=1,\dots,6$) via validated numerics. In our implementation code \cite{code}, the file \texttt{verify\_singular\_points.cc} \revise{provides the rigorous enclosure of these singular points}, guaranteed by the Krawczyk method \cite{MR2652784,MR0255046}. For the first singular point $p_1$, we select the closed loop as $\Sigma_1$, as introduced in Section~\ref{sec:path_contour}.

The fundamental system of solutions is analytically continued along the loop $\Sigma_1$. This procedure can be executed using the file \texttt{find\_monodromy\_path1.cc} in our implementation code \cite{code}. The resulting matrix \( (\Sigma_i)_\ast \mathrm{Id} \) is rigorously enclosed within the following interval matrix:
\begin{align}
	&(\Sigma_i)_\ast \mathrm{Id}\in\\
	&{\footnotesize\begin{bmatrix}
			-0.5928_{4103}^{1368}  -0.8206_{3778}^{1043}\,\im & 0.177381_{11}^{67}  -0.106186_{93}^{37}\,\im & -0.62597_{548}^{386} + 0.30719_{318}^{48}\,\im & -7.89_{22669}^{18086}\cdot10^{-6} + 1.1831_{343}^{801}\cdot10^{-5}\,\im \\
			778.3_{3316}^{9694} -537._{32321}^{25944}\,\im & 212.32_{17}^{3} -7.12_{29292}^{16315}\,\im & -719.0_{1365}^{0986} -35.85_{5359}^{1577}\,\im & -0.01304_{8333}^{7265} + 0.00673_{43657}^{54342}\,\im \\ 
			194.8_{1319}^{2909} -134.2_{6552}^{4962}\,\im & 53.366_{122}^{445} -1.7_{403084}^{399848}\,\im & -180.86_{206}^{112} -9.10_{97147}^{87717}\,\im & -0.003265_{4937}^{2274} + 0.001682_{1324}^{3988}\,\im \\
			1896_{051.7}^{216.3} -1350_{876.1}^{711.5}\,\im & 52477_{3.86}^{7.21} -2531_{3.671}^{0.318}\,\im & -17709_{66.2}^{56.4} -620_{17.539}^{07.764}\,\im & -32.86_{931}^{6553} + 17.05_{0769}^{3526}\,\im 
		\end{bmatrix}.}
\end{align}
Although the maximum relative error of the resulting matrix is as small as  $7.5\cdot 10^{-5}$, the maximum radius of interval reaches approximately $82$. As a result, the monodromy matrix is enclosed within a significantly large interval.
Consequently, the resulting interval inclusion of the monodromy matrix ${M}_{\Sigma_1}=\Phi\left(p_0\right)^{-1} \left((\Sigma_1)_{\ast}\mathrm{Id}\right)\Phi\left(p_0\right)$ is given by

\begin{align*}
	&{M}_{\Sigma_1}\in\\
	&\left\langle{\small\begin{bmatrix}
			-1.00 + 9.52\cdot 10^{-12}\,\im & -2.00 + 1.11\cdot 10^{-9}\,\im & -2.00 -1.17\cdot 10^{-9}\,\im & -1.00-1.15\cdot 10^{-10}\,\im \\
			-1.67\cdot 10^{-11} -2.29\cdot 10^{-11}\,\im & -1.00  -1.38\cdot 10^{-11}\,\im & 7.02\cdot 10^{-10} -1.01\cdot 10^{-9}\,\im & 2.44\cdot 10^{-11} -1.01\cdot 10^{-10}\,\im \\
			-5.68\cdot 10^{-11} -2.49\cdot 10^{-11}\,\im & 4.00  + -3.01\cdot 10^{-9}\,\im & 3.00  -4.06\cdot 10^{-10}\,\im & 2.00  -1.81\cdot 10^{-10}\,\im \\
			1.13\cdot 10^{-10} + 1.18\cdot 10^{-11}\,\im & -4.00  + 4.73\cdot 10^{-9}\,\im & -4.00  + 2.11\cdot 10^{-9}\,\im & -3.00  + 2.42\cdot 10^{-10}\,\im 
	\end{bmatrix}}\right.,\\
	&\left.{\small\begin{bmatrix}
			281.2 & 1860.4 & 1.1037.5 & 302.5 \\
			281.2 & 1860.4 & 1.1037.5 & 302.5 \\
			281.2 & 1860.4 & 1.1037.5 & 302.5 \\
			281.2 & 1860.4 & 1.1037.5 & 302.5
	\end{bmatrix}}\right\rangle,
\end{align*}
where $\langle\cdot,\cdot\rangle$ denote the middle-radius form of the interval matrix.
%
Our target monodromy matrix \eqref{eq:theMonodromyMatrix} is unimodular, with all entries being integers. However, the above interval inclusion fails to confirm the uniqueness of an integer within each interval entry. This difficulty arises from the \emph{wrapping effect} in the rigorous integration for ODEs. In particular, the scales of each column in \( (\Sigma_i)_\ast \mathrm{Id} \) vary significantly, which can cause the absolute values of errors to grow disproportionately. While the kv library’s rigorous integrator is generally effective for obtaining rigorous inclusions of the solution to ODEs, achieving the desired precision in this case is challenging.

\begin{table}[htbp]
	\caption{Maximum radius of rigorous inclusion of analytic continuation using \texttt{double} precision.}
	\label{table1}
	\centering
	\begin{tabular}{c|ccc}
		\hline
		& $\Sigma_{1,1}$ ($\|\bphi_{1,1}^{k}(1)\|_\infty$) & $\Sigma_{1,2}$ ($\|\bphi_{1,2}^{k}(1)\|_\infty$) & $\Sigma_{1,3}$ ($\|\bphi_{1,3}^{k}(0)\|_\infty$)\\
		\hline 
		$\bphi^1$ & $1.340863453602736\cdot 10^{-7}$ & $2.406468399840378\cdot 10^{-5}$ & $82.297714100219306$\\
		$\bphi^2$ & $3.559879058201431\cdot 10^{-9}$ & $4.908559887439878\cdot 10^{-7}$ & $1.6765052650080179$\\
		$\bphi^3$ & $1.227905244149951\cdot 10^{-8}$ & $1.432497398923260\cdot 10^{-6}$ & $4.8877803169816616$\\
		$\bphi^4$ & $1.333333535811012\cdot 10^{-12}$ & $4.024366085601110\cdot 10^{-10}$ & $1.378798759620991\cdot 10^{-3}$\\
		\hline
	\end{tabular}
\end{table}

Table~\ref{table1} shows the maximum radius of rigorous inclusion obtained during the analytic continuation along  the loop $\Sigma_1$, which consists of three segments. Starting from the initial values $\bm{e}_k$ (the canonical basis) with the zero radius, the results illustrate how the rigorous integration of the ODEs \eqref{eq:ODEs_on_sigma1} loses precision for the solution along each segment. 
In particular, along the final path $\Sigma_{1,3}$, a significant error accumulates during the rigorous integration. This prevents achieving a rigorous inclusion of the monodromy matrix with the target margin of error less than $0.5$.

To overcome this difficulty, we employ the \texttt{DD} precision (Double-Double precision) of numbers to represent each interval.  The \texttt{DD} precision combines two \texttt{double}-precision floating-point numbers to represent a single value, enabling highly accurate computations with approximately $31$ digits of precision. The \texttt{DD} arithmetic \cite{Bailey1995,QDLibrary} performs operations such as addition, subtraction, multiplication, and division in a layered manner to correct rounding errors and maintain precision. Specifically, it decomposes each value into high-precision (high component) and low-precision (low component) parts, ensuring that results are adjusted to minimize errors. This approach is particularly valuable even in interval arithmetic for obtaining tight interval enclosure, where high precision arithmetic is crucial.
The kv library supports computations with the \texttt{DD} precision, providing four basic arithmetic operations, elementary math functions, and more. In particular, it facilitates the implementation of rigorous ODE integration by using the \texttt{DD} precision to represent the numbers in interval, ensuring both accuracy and tight enclosures in interval arithmetic.

\begin{table}[htbp]
	\caption{Maximum radius of rigorous inclusion of analytic continuation using \texttt{DD} precision.}
	\label{table2}
	\centering
	\begin{tabular}{c|ccc}
		\hline
		& $\Sigma_{1,1}$ ($\|\bphi_{1,1}^{k}(1)\|_\infty$) & $\Sigma_{1,2}$ ($\|\bphi_{1,2}^{k}(1)\|_\infty$) & $\Sigma_{1,3}$ ($\|\bphi_{1,3}^{k}(0)\|_\infty$)\\
		\hline 
		$\bphi^1$ & $5.6843418860808015\cdot 10^{-14}$ & $5.5511151231257836\cdot 10^{-17}$ & $2.3283064365386963\cdot 10^{-10}$\\
		$\bphi^2$ & $2.7755575615628918\cdot 10^{-17}$ & $1.7763568394002502\cdot 10^{-15}$ & $1.1641532182693481\cdot 10^{-10}$\\
		$\bphi^3$ & $1.7763568394002502\cdot 10^{-15}$& $2.8421709430404007\cdot 10^{-14}$ & $2.3283064365386963\cdot 10^{-10}$\\
		$\bphi^4$ & $2.7105054312137606\cdot 10^{-20}$ & $2.1684043449710089\cdot 10^{-19}$ & $2.1684043449710089\cdot 10^{-19}$\\
		\hline
	\end{tabular}
\end{table}

We implemented the rigorous analytic continuation using the \texttt{DD} precision in the file \texttt{find\_monodromy\_path1\_dd.cc}. Table~\ref{table2} displays the result of rigorous integration usin the \texttt{DD} precision, which represents the maximum radius of each inclusion during the analytic continuation. While some loss of precision in the solution occurs along the loop, the resulting error bound remains approximately $2\cdot 10^{-10}$, which is sufficiently small to get our target margin of error for the monodromy matrix.
Consequently, the resulting monodromy matrix is rigorously included in
\begin{align*}
	&{M}_{\Sigma_1}\in\\
	&\left\langle{\small\begin{bmatrix}
			-1.00  + 2.12\cdot 10^{-14}\,\im & -2.00  -3.49\cdot 10^{-10}\,\im & -2.00 -3.38\cdot 10^{-10}\,\im &-1.00 -9.19\cdot 10^{-11}\,\im \\
		4.09\cdot 10^{-15}  & -1.00 -1.48\cdot 10^{-10}\,\im & -1.38\cdot 10^{-17}  -1.48\cdot 10^{-10}\,\im & -7.12\cdot 10^{-18} -7.43\cdot 10^{-11}\,\im \\ 
		-4.12\cdot 10^{-15}  -3.85\cdot 10^{-14}\,\im & 4.00  -1.76\cdot 10^{-13}\,\im & 3.00  -1.27\cdot 10^{-11}\,\im & 2.00 -1.61\cdot 10^{-10}\,\im \\ 
		4.12\cdot 10^{-15} + 3.85\cdot 10^{-14}\,\im &  -4.00  + 1.77\cdot 10^{-13}\,\im &  -4.00 + 1.27\cdot 10^{-11}\,\im & -3.00  + 1.61\cdot 10^{-10}\,\im
		%
			\end{bmatrix}}\right.,\\
	&\left.{\small\begin{bmatrix}
		3.31\cdot 10^{-4} & 2.178\cdot 10^{-3} & 1.08\cdot 10^{-2}  & 3.55\cdot 10^{-4}  \\
		3.31\cdot 10^{-4} & 2.178\cdot 10^{-3} & 1.08\cdot 10^{-2}  & 3.55\cdot 10^{-4}  \\
		3.31\cdot 10^{-4} & 2.178\cdot 10^{-3} & 1.08\cdot 10^{-2}  & 3.55\cdot 10^{-4}  \\
		3.31\cdot 10^{-4} & 2.178\cdot 10^{-3} & 1.08\cdot 10^{-2}  & 3.55\cdot 10^{-4} 
			\end{bmatrix}}\right\rangle.
\end{align*}

From the fact that our target monodromy matrix \eqref{eq:theMonodromyMatrix} is unimodular, as introduced in Section~\ref{Sct:K3}, our computer-assisted approach proves there uniquely exists an integer entries in the above interval inclusions.
Finally, the monodromy matrix is determined as
\[
M_{\Sigma_1} = \begin{bmatrix}
	-1 & -2 & -2 & -1\\ 0 & -1 & 0 & 0\\ 0 & 4 & 3 & 2\\ 0 & -4 & -4 & -3
\end{bmatrix}.
\]
This completes the proofs for the path $\Sigma_1$. For the other paths, the same approach is employed to obtain the monodromy matrices $M_{\Sigma_i}$ ($i=2,\dots,6$) as defined in \eqref{eq:ResultingMonodromy}. These matrices can be obtained by executing the files \texttt{find\_monodromy\_pathi\_dd.cc} in \cite{code}, where \texttt{i} corresponds to the index $i$ labeling each path.


					\section*{Conclusion and Future works}

					In this paper, we provided a rigorous numerical framework for computing the monodromy matrices of the Picard--Fuchs differential equation, with a focus on applications to families of K3 toric hypersurfaces. By employing the Pfaffian equation and rigorous forward integration of ODEs using interval arithmetic, we achieved rigorous analytic continuation of the fundamental system of solutions along predefined contours. This enabled the computation of monodromy matrices with guaranteed precision and rigor, providing new insights into the monodromy properties of linear differential equations. Furthermore, combining the unimodularity of the monodromy matrices with its rigorous inclusion, we obtain a computer-assisted proof for the monodromy problem of the Picard--Fuchs differential equation.

					Our computational framework highlights the potential of computer-assisted proofs in addressing complex problems in algebraic geometry and mathematical physics. Central to our approach is the seamless integration of validated numerics with theoretical constructs.  The kv library, employing Taylor series expansions and affine arithmetic, played a pivotal role in ensuring rigorous inclusion of solution trajectories during forward integration of ODEs. In particular, implementation of the \texttt{DD} arithmetic for rigorous integration was a critical component of our success in achieving computer-assisted proofs.

					This work opens several avenues for future works. One promising direction is the extension of this framework to other differential equations, such as those associated with Calabi--Yau varieties, mirror symmetry, and Hodge theory. By providing a rigorous computational foundation, this approach should bridge the gap between numerical computation and abstract theoretical concepts, facilitating precise and systematic exploration of algebraic and geometric structures. Additionally, our method can be applied to computer-assisted proofs of the nonintegrability of dynamical systems. By combining the computation of monodromy matrices with Morales--Ramis theory \cite{MoralesRuiz1999}, rooted in differential Galois theory, one could rigorously validate properties of monodromy groups, providing a proof of nonintegrability of dynamical systems.

\appendix

\section{Proofs of lemmas in Section \ref{sec:fundamental_lemmas}}\label{sec:appendix_proof}
Here we provide the proofs of the fundamental lemmas \ref{lem:TruncationError1}, \ref{lem:TruncationError2}, and \ref{lem:TruncationError3} in Section \ref{sec:fundamental_lemmas}. First of all, we prepare Stirling's approximation for factorials as follows:
\begin{lem}[Stirling's approximation for factorials, e.g., \cite{Dutka1984}]
	Let $n$ be a positive integer. The following inequalities for factorials holds:
	\begin{equation}\label{eq:Stirling}
		\sqrt{2\pi}\,n^{n+\frac{1}{2}}e^{-n}\le n!\le n^{n+\frac{1}{2}}e^{-n+1}.
	\end{equation}
\end{lem}
Let us begin by considering the proof of Lemma \ref{lem:TruncationError1}.
\begin{proof}[Proof of Lemma \ref{lem:TruncationError1}]
	First, we obtain the estimate \eqref{eq:delta1} using \eqref{eq:a_lm}, the binomial theorem, and elementary calculations for the geometric sequence as follows:
	\begin{align}
		\left|\sum_{l+m\ge N+1}a_{l,m}\lambda^l \mu^m\right| &= \left|\sum_{n=N+1}^{\infty}\left(\sum_{l+m=n}a_{l,m}\lambda^l \mu^m\right)\right|\\
		&\le \sum_{n=N+1}^{\infty}\left(\sum_{l+m=n}\frac{(2l+4m)!}{(l+m)!\,l!\left(m!\right)^3}|\lambda|^l |\mu|^m\right)\\
		&= \sum_{n=N+1}^{\infty}\left(\sum_{l+m=n}\frac{(2l+4m)!}{\left((l+m)!\right)^2\left(m!\right)^2}\frac{(l+m)!}{l!\,m!}|\lambda|^l |\mu|^m\right)\\
		&\le \sum_{n=N+1}^{\infty}\frac{1}{\left(n!\right)^2}\max_{0\le m\le n}\left(\frac{(2n+2m)!}{\left(m!\right)^2}\right)\left(|\lambda| + |\mu|\right)^n\\
		&= \sum_{n=N+1}^{\infty}\frac{(4n)!}{\left(n!\right)^4}\left(|\lambda| + |\mu|\right)^n\\
		&\le\frac{(4(N+1))!}{\left((N+1)!\right)^4}(|\lambda|+|\mu|)^{N+1}\sum_{n'=0}^{\infty}\left[256\left(|\lambda| + |\mu|\right)\right]^{n'}\\
		&\le\frac{(4(N+1))!}{\left((N+1)!\right)^4}\frac{(|\lambda|+|\mu|)^{N+1}}{1-256(|\lambda|+|\mu|)}=\delta_1(N,\lambda,\mu),\label{eq:estimate_delta_1}
	\end{align}
	where we used the fact that $\beta_n>1$, and for $n\ge 0$, it forms a geometric series with a common ratio of 
	\begin{align}
		\frac{(4n+4)!}{\left((n+1)!\right)^4}\cdot\frac{\left(n!\right)^4}{(4n)!}&=\frac{(4n+4)(4n+3)(4n+2)(4n+1)}{(n+1)^4}\\
		&=4\left(4-\frac{1}{n+1}\right)\left(4-\frac{2}{n+1}\right)\left(4-\frac{3}{n+1}\right)\le 256.\label{eq:alpha_n}
	\end{align}
	
	Second, for $l+m\ge n$, we have from \eqref{eq:b_lm}
	\begin{align*}
		b_{l,m}&= \sum_{j=1}^{2l+4m}\frac{4}{j}-\sum_{j=1}^{l+m}\frac{1}{j}-\sum_{j=1}^{m}\frac{3}{j}\\
		&=\sum_{j=l+m+1}^{2l+4m}\frac{4}{j}+\sum_{j=m+1}^{l+m}\frac{3}{j}\\
		&\le \sum_{j=n+1}^{4n}\frac{4}{j}+\sum_{j=1}^{n}\frac{3}{j}\\
		&<\int_{n}^{4n}\frac{4}{x}dx + 3 + \int_{1}^{n}\frac{3}{x}dx=4\log 4 + 3 + 3\log n.
	\end{align*}
	We observe that
	\begin{align}
		\frac{4\log 4 + 3 + 3\log (n+1)}{4\log 4 + 3 + 3\log n}
		&=\frac{4\log 4 + 3 + 3\log n + 3\log \left(\frac{n+1}{n}\right)}{4\log 4 + 3 + 3\log n}\\
		&=1 + \frac{3\log \left(1+\frac{1}{n}\right)}{4\log 4 + 3+ 3\log n}=\beta_n.\label{eq:beta_n}
	\end{align}
	Since $\beta_n$ is monotonically decreasing with respect to $n$, it follows that $\beta_n\le \beta_N$ holds for any $n\ge N$.
	The estimate \eqref{eq:delta2} is obtained by using \eqref{eq:alpha_n}, \eqref{eq:beta_n} and the similar calculations used in \eqref{eq:estimate_delta_1}
	\begin{align}
		\left|\sum_{l+m\ge N+1}a_{l,m}b_{l,m}\lambda^l \mu^m\right| &= \left|\sum_{n=N+1}^{\infty}\left(\sum_{l+m=n}a_{l,m}b_{l,m}\lambda^l \mu^m\right)\right|\\
		&\le 
		\sum_{n=N+1}^{\infty}\frac{(4n)!}{\left(n!\right)^4}\left(4\log 4 + 3 + 3\log n\right)\left(|\lambda| + |\mu|\right)^n\\
		&\le \frac{(4(N+1))!}{\left((N+1)!\right)^4}\left(4\log 4 + 3 + 3\log (N+1)\right)(|\lambda|+|\mu|)^{N+1}\sum_{n'=0}^{\infty}\left[256\beta_{\revise{N+1}}(|\lambda|+|\mu|)\right]^{n'}\\
		&\le \frac{(4(N+1))!}{\left((N+1)!\right)^4}\left(4\log 4 + 3 + 3\log (N+1)\right)\frac{(|\lambda|+|\mu|)^{N+1}}{1-256\beta_{\revise{N+1}}(|\lambda|+|\mu|)}\\
		&=\delta_2(N,\lambda,\mu).\label{eq:estimate_delta2}
	\end{align}
	
	Third, the estimate \eqref{eq:delta3} is given by an analogue of the above, that is
	\begin{align}
		\left|\sum_{l+m\ge N+1}a_{l,m}b_{l,m}^2\lambda^l \mu^m\right| &= \left|\sum_{n=N+1}^{\infty}\left(\sum_{l+m=n}a_{l,m}b_{l,m}^2\lambda^l \mu^m\right)\right|\\
		&\le 
		\sum_{n=N+1}^{\infty}\frac{(4n)!}{\left(n!\right)^4}\left(4\log 4 + 3 + 3\log n\right)^2\left(|\lambda| + |\mu|\right)^n\\
		&\le \frac{(4(N+1))!}{\left((N+1)!\right)^4}\left(4\log 4 + 3 + 3\log (N+1)\right)^2(|\lambda|+|\mu|)^{N+1}\sum_{n'=0}^{\infty}\left[256\beta_{\revise{N+1}}^2(|\lambda|+|\mu|)\right]^{n'}\\
		&\le\frac{(4(N+1))!}{\left((N+1)!\right)^4}\left(4\log 4 + 3 + 3\log (N+1)\right)^2\frac{(|\lambda|+|\mu|)^{N+1}}{1-256\beta_{\revise{N+1}}^2(|\lambda|+|\mu|)}=\delta_3(N,\lambda,\mu).\label{eq:estimate_delta3}
	\end{align}
	
	Fourth, for $l+m\ge n$, we have from \eqref{eq:c_lm}
	\begin{align}\label{eq:c_estimate}
		c_{l,m}&= \sum_{j=1}^{2l+4m}\frac{16}{j^2}-\sum_{j=1}^{l+m}\frac{1}{j^2}-\sum_{j=1}^{m}\frac{3}{j^2}\le \sum_{j=1}^{2n+2m}\frac{16}{j^2}\le 16\zeta(2)=\frac{8}{3}\pi^2,
	\end{align}
	where $\zeta(s)$ is the Riemann zeta function defined by $\zeta(s)\bydef\sum_{n=1}^{\infty} \frac{1}{n^{s}}$.
	Then we have the estimate \eqref{eq:delta4} using \eqref{eq:estimate_delta_1} and \eqref{eq:c_estimate}
	\begin{align}
		\left|\sum_{l+m\ge N+1}a_{l,m}c_{l,m}\lambda^l \mu^m\right| &= \left|\sum_{n=N+1}^{\infty}\left(\sum_{l+m=n}a_{l,m}c_{l,m}\lambda^l \mu^m\right)\right|\\
		&\le \frac{8}{3}\pi^2\sum_{n=N+1}^{\infty}\frac{(4n)!}{\left(n!\right)^4}\left(|\lambda| + |\mu|\right)^n\\
		&\le \frac{8}{3}\pi^2\frac{(4(N+1))!}{\left((N+1)!\right)^4}\frac{(|\lambda|+|\mu|)^{N+1}}{1-256(|\lambda|+|\mu|)}=\delta_4(N,\lambda,\mu).\label{eq:estimate_delta4}
	\end{align}
	
	Finally, we consider the estimate \eqref{eq:delta5}. For $\epsilon\in\{1/2,1\}$ it follows from \eqref{eq:d_lm} and the fact that the gamma function $\Gamma(x)$ is monotonically increasing for positive integers $x$
	\begin{align}
		\left|\sum_{l+m\ge N+1}d_{l+\epsilon,m}\frac{\lambda^{l+2m+\epsilon}}{\mu^{l+m+\epsilon}}\right|
		&= \left|\sum_{n=N+1}^\infty\left(\sum_{l+m=n}d_{l+\epsilon,m}\frac{\lambda^{l+2m+\epsilon}}{\mu^{l+m+\epsilon}}\right)\right|\nonumber\\
		&\le\left|\frac{\lambda}{\mu}\right|^\epsilon\left|\sum_{n=N+1}^\infty\left(\sum_{l+m=n}d_{l+\epsilon,m}\left(\frac{\lambda}{\mu}\right)^{l+m}\lambda^m\right)\right|\nonumber\\
		&\le\left|\frac{\lambda}{\mu}\right|^\epsilon\sum_{n=N+1}^\infty\left|\frac{\lambda}{\mu}\right|^{n}\sum_{l+m=n}\frac{\Gamma(l+m+\epsilon)^3\,l!}{\Gamma(2(l+\epsilon))\Gamma(l+2m+\epsilon+1)(l+m)!}\frac{(l+m)!}{m!\,l!}|\lambda|^m\nonumber\\
		&\le\left|\frac{\lambda}{\mu}\right|^\epsilon\sum_{n=N+1}^\infty\left|\frac{\lambda}{\mu}\right|^{n}\left(1+|\lambda|\right)^n\sum_{l+m=n}\frac{(n!)^2\,l!}{(2l)!\,(n+m)!}.\label{eq:(v)_1}
	\end{align}
	Here, from \eqref{eq:Stirling}, we have the following estimates:
	\[
	\begin{aligned}
		\left(n!\right)^2&\le n^{2n+1}e^{-2n+2}\\
		l!&\le l^{l+\frac{1}{2}}e^{-l+1}\\
		\left(2l\right)! &\ge \sqrt{2\pi}\,(2l)^{2l+\frac{1}{2}}e^{-2l}\\
		(n+m)!&\ge \sqrt{2\pi}\,(n+m)^{n+m+\frac{1}{2}}e^{-(n+m)}.
	\end{aligned}
	\]
	Using these estimates we obtain an upper bound for \eqref{eq:(v)_1}
	\begin{align}
		\left|\sum_{l+m\ge N+1}d_{l+\epsilon,m}\frac{\lambda^{l+2m+\epsilon}}{\mu^{l+m+\epsilon}}\right|\nonumber
		&\le\left|\frac{\lambda}{\mu}\right|^\epsilon\sum_{n=N+1}^\infty\left|\frac{\lambda}{\mu}\right|^{n}\left(1+|\lambda|\right)^n\sum_{l+m=n}\frac{(n!)^2\,l!}{(2l)!\,(n+m)!}\nonumber\\
		&\le\left|\frac{\lambda}{\mu}\right|^\epsilon\sum_{n=N+1}^\infty\left|\frac{\lambda}{\mu}\right|^{n}\left(1+|\lambda|\right)^n\sum_{l+m=n}\frac{n^{2n+1}e^{-2n+2}l^{l+\frac{1}{2}}e^{-l+1}}{2\pi (2l)^{2l+\frac{1}{2}}e^{-2l}(n+m)^{n+m+\frac{1}{2}}e^{-(n+m)}}\nonumber\\
		&=\left|\frac{\lambda}{\mu}\right|^\epsilon\sum_{n=N+1}^\infty\left|\frac{\lambda}{\mu}\right|^{n}\left(1+|\lambda|\right)^n\sum_{l+m=n}\frac{e^3}{2\sqrt{2}\,\pi}\frac{1}{2^{2l}}\frac{n^{2n+1}}{l^l(n+m)^{n+m+\frac{1}{2}}}\nonumber\\
		&\le\left|\frac{\lambda}{\mu}\right|^\epsilon\sum_{n=N+1}^\infty\left|\frac{\lambda}{\mu}\right|^{n}\left(1+|\lambda|\right)^n\sum_{l+m=n}\frac{e^3}{2\sqrt{2}\,\pi}\frac{1}{2^{2l}}\frac{n^{2n+\frac{1}{2}}}{l^l(n+m)^{n+m}}.\label{eq:(v)_2}
	\end{align}
	
	Furthermore, we consider the minimization problem for the denominator of \eqref{eq:(v)_2}, which involves finding $x\bydef l$ and $y\bydef n+m$ such that
	\begin{align*}
		\mbox{minimize}~\log\left(2^{2x}x^xy^y\right)\quad \mbox{subject to}~ x+y=2n.
	\end{align*}
	Using the method of Lagrange multipliers, the solution $(x,y)$ is determined by finding a stationary point of the Lagrangian function
	\[
	f(x,y)\bydef 2x\log 2 + x\log x + y\log y - \lambda(x+y),
	\]
	where $\lambda$ is the Lagrange multiplier. The stationary point satisfies the following system of equations for the derivatives:
	\[
	\begin{cases}
		\partial_x f = 2\log 2 + \log x + 1 - \lambda = 0,\\
		\partial_y f = \log y + 1 - \lambda = 0.
	\end{cases}
	\]
	Solving this system yields the relation $\log {y} = 2\log 2 + \log {x}$, implying ${y}=4{x}$. 
	From the definition of $x$ and $y$, we have $4l=n+m=2n-l$ at the stationary point, leading to \revise{the optimal value} $\revise{l^\ast}=2n/5$.
	
	Using this stationary point, we have an upper bound for \eqref{eq:(v)_2}
	\begin{align}
		\left|\sum_{l+m\ge N+1}d_{l+\epsilon,m}\frac{\lambda^{l+2m+\epsilon}}{\mu^{l+m+\epsilon}}\right|\nonumber
		&\le\left|\frac{\lambda}{\mu}\right|^\epsilon\sum_{n=N+1}^\infty\left|\frac{\lambda}{\mu}\right|^{n}\left(1+|\lambda|\right)^n\sum_{l+m=n}\frac{e^3}{2\sqrt{2}\,\pi}\frac{1}{2^{2l}}\frac{n^{2n+\frac{1}{2}}}{l^l(n+m)^{n+m}}\\
		&\le\left|\frac{\lambda}{\mu}\right|^\epsilon\sum_{n=N+1}^\infty\left|\frac{\lambda}{\mu}\right|^{n}\left(1+|\lambda|\right)^n\frac{e^3}{2\sqrt{2}\,\pi}\frac{n^{2n+\frac{1}{2}}}{4^{\revise{l^\ast}}\revise{l^\ast}^\revise{l^\ast}(4\revise{l^\ast})^{4\revise{l^\ast}}}\\
		&=\left|\frac{\lambda}{\mu}\right|^\epsilon\sum_{n=N+1}^\infty\left|\frac{\lambda}{\mu}\right|^{n}\left(1+|\lambda|\right)^n\frac{e^3}{2\sqrt{2}\,\pi}\frac{n^{2n+\frac{1}{2}}}{(8n/5)^{2n}}\\
		&=\frac{e^3}{2\sqrt{2}\,\pi}\left|\frac{\lambda}{\mu}\right|^\epsilon\sum_{n=N+1}^\infty\left|\frac{\lambda}{\mu}\right|^{n}\left(1+|\lambda|\right)^n\sqrt{n}\,\left(\frac{25}{64}\right)^n\\
		&\le \frac{e^3}{2\sqrt{2}\,\pi}\left|\frac{\lambda}{\mu}\right|^\epsilon\sqrt{N+1}\left[\frac{25}{64}\left(\frac{|\lambda|+|\lambda|^2}{|\mu|}\right)\right]^{N+1}\sum_{n'=0}^\infty\left(\frac{|\lambda|+|\lambda|^2}{|\mu|}\right)^{n'}\left(\frac{25}{64}\sqrt{1+\frac{1}{\revise{N+1}}}\right)^{n'}\\
		&\le  \frac{e^3}{2\sqrt{2}\,\pi}\left|\frac{\lambda}{\mu}\right|^\epsilon\frac{\sqrt{N+1}\left[\frac{25}{64}\left(\frac{|\lambda|+|\lambda|^2}{|\mu|}\right)\right]^{N+1}}{1-\frac{25}{64}\sqrt{1+\frac{1}{\revise{N+1}}}\left(\frac{|\lambda|+|\lambda|^2}{|\mu|}\right)}=\delta_5^\epsilon(N,\lambda,\mu).\label{eq:estimate_delta5}
	\end{align}
	This completes the proof.
\end{proof}


Next, we present the proof of Lemma \ref{lem:TruncationError2}, which concerns the first partial derivatives of each fundamental solution as defined in Theorem \ref{thm:Fundamantal_sol}.

\begin{proof}[Proof of Lemma \ref{lem:TruncationError2}]
	First, we derive the estimate \eqref{eq:delta6}, which follows a similar argument to \revise{the one used for \eqref{eq:estimate_delta_1}} in the proof of Lemma \ref{lem:TruncationError1}. Using \eqref{eq:a_lm}, the binomial theorem, and elementary calculations for the geometric sequence, it follows that
	\begin{align}
		\left|\sum_{l+m\ge N+1}la_{l,m}\lambda^{l-1} \mu^m\right|
		&= \left|\sum_{n=N+1}^{\infty}\left(\sum_{l+m=n}la_{l,m}\lambda^{l-1} \mu^m\right)\right|\\
		&\le \sum_{n=N+1}^{\infty}\left(\sum_{l+m=n}l\frac{(2l+4m)!}{(l+m)!\,l!\left(m!\right)^3}|\lambda|^{l-1} |\mu|^m\right)\\
		&= \sum_{n=N+1}^{\infty}\left(\sum_{l+m=n\atop l\ge1}\frac{(2l+4m)!}{(l+m-1)!(l+m)!\left(m!\right)^2}\frac{(l+m-1)!}{(l-1)!\,m!}|\lambda|^{l-1} |\mu|^m\right)\\
		&\le \sum_{n=N+1}^{\infty}\frac{1}{\left(n-1\right)!n!}\max_{0\le m\le n}\left(\frac{(2n+2m)!}{\left(m!\right)^2}\right)\left(|\lambda| + |\mu|\right)^{n-1}\\
		&= \sum_{n=N+1}^{\infty}\frac{(4n)!}{(n-1)!\left(n!\right)^3}\left(|\lambda| + |\mu|\right)^{n-1}\\
		&\le \frac{(4(N+1))!}{N!\left((N+1)!\right)^3}\left(|\lambda| + |\mu|\right)^{N} \sum_{n'=0}^{\infty}\left[\gamma_{\revise{N+1}}(|\lambda|+|\mu|)\right]^{n'}\\
		&\le\frac{(4(N+1))!}{N!\left((N+1)!\right)^3}\frac{(|\lambda|+|\mu|)^{N}}{1-\gamma_{\revise{N+1}}(|\lambda|+|\mu|)}=\delta_6(N,\lambda,\mu),
	\end{align}
	where we used the fact that\revise{, for $n=N+1, N+2, \dots$,} the infinite sum is bounded by a geometric series with a common ratio of
	\begin{align*}
		\frac{(4n+4)!}{n!\left((n+1)!\right)^3}\cdot\frac{(n-1)!\left(n!\right)^3}{(4n)!}&=\frac{(4n+4)(4n+3)(4n+2)(4n+1)}{n(n+1)^3}\\
		&=4\left(4-\frac{1}{n+1}\right)\left(4-\frac{2}{n+1}\right)\left(4+\frac{1}{n}\right)\\
		&\le 64\left(4+\frac1{\revise{N+1}}\right)=\gamma_{\revise{N+1}}.
	\end{align*}
	Similarly, we have
	\begin{align}
		\left|\sum_{l+m\ge N+1}ma_{l,m}\lambda^{l} \mu^{m-1}\right| 
		&= \left|\sum_{n=N+1}^{\infty}\left(\sum_{l+m=n}ma_{l,m}\lambda^{l} \mu^{m-1}\right)\right|\\
		&\le \sum_{n=N+1}^{\infty}\left(\sum_{l+m=n}m\frac{(2l+4m)!}{(l+m)!\,l!\left(m!\right)^3}|\lambda|^{l} |\mu|^{m-1}\right)\\
		&= \sum_{n=N+1}^{\infty}\left(\sum_{l+m=n\atop m\ge1}\frac{(2l+4m)!}{(l+m-1)!(l+m)!\left(m!\right)^2}\frac{(l+m-1)!}{l!\,(m-1)!}|\lambda|^{l} |\mu|^{m-1}\right)\\
		&\le \sum_{n=N+1}^{\infty}\frac{1}{\left(n-1\right)!n!}\max_{0\le m\le n}\left(\frac{(2n+2m)!}{\left(m!\right)^2}\right)\left(|\lambda| + |\mu|\right)^{n-1}\\
		&= \sum_{n=N+1}^{\infty}\frac{(4n)!}{(n-1)!\left(n!\right)^3}\left(|\lambda| + |\mu|\right)^{n-1}\\
		&\le\frac{(4(N+1))!}{N!\left((N+1)!\right)^3}\frac{(|\lambda|+|\mu|)^{N}}{1-\gamma_{\revise{N+1}}(|\lambda|+|\mu|)}=\delta_6(N,\lambda,\mu).
	\end{align}

	Second, we derive the $\delta_7$ bound in \eqref{eq:delta7}, which follows an \revise{analogous argument to the one used for \eqref{eq:estimate_delta2}} in the proof of Lemma \ref{lem:TruncationError1}. The binomial theorem and basic calculations for the geometric sequence also provide the following:
	\begin{align}
		\left|\sum_{l+m\ge N+1}la_{l,m}b_{l,m}\lambda^{l-1} \mu^m\right|
		&= \left|\sum_{n=N+1}^{\infty}\left(\sum_{l+m=n}la_{l,m}b_{l,m}\lambda^{l-1} \mu^m\right)\right|\\
		&\le \sum_{n=N+1}^{\infty}\left(\sum_{l+m=n}l\frac{(2l+4m)!}{(l+m)!\,l!\left(m!\right)^3}b_{l,m}|\lambda|^{l-1} |\mu|^m\right)\\
		&\le\sum_{n=N+1}^{\infty}\frac{(4n)!}{(n-1)!\left(n!\right)^3}\left(4\log 4+3+3\log n\right)\left(|\lambda| + |\mu|\right)^{n-1}\\
		&\le\frac{(4(N+1))!}{N!\left((N+1)!\right)^3}\frac{\left(4\log 4 + 3 + 3\log (N+1)\right)(|\lambda|+|\mu|)^{N}}{1-\beta_{\revise{N+1}}\gamma_{\revise{N+1}}(|\lambda|+|\mu|)}=\delta_7(N,\lambda,\mu).
	\end{align}
	Correspondingly, we have
	\begin{align}
		\left|\sum_{l+m\ge N+1}ma_{l,m}b_{l,m}\lambda^{l} \mu^{m-1}\right| 
		&= \left|\sum_{n=N+1}^{\infty}\left(\sum_{l+m=n}ma_{l,m}b_{l,m}\lambda^{l} \mu^{m-1}\right)\right|\\
		&\le \sum_{n=N+1}^{\infty}\frac{(4n)!}{(n-1)!\left(n!\right)^3}\left(4\log 4+3+3\log n\right)\left(|\lambda| + |\mu|\right)^{n-1}\\
		&\le\frac{(4(N+1))!}{N!\left((N+1)!\right)^3}\frac{\left(4\log 4 + 3 + 3\log (N+1)\right)(|\lambda|+|\mu|)^{N}}{1-\beta_{\revise{N+1}}\gamma_{\revise{N+1}}(|\lambda|+|\mu|)}=\delta_7(N,\lambda,\mu).
	\end{align}

	Third, $\delta_8$ bound is derived using an argument analogous to \revise{the one used for \eqref{eq:estimate_delta3}} in the proof of Lemma \ref{lem:TruncationError1}, that is
	\begin{align}
		\left|\sum_{l+m\ge N+1}la_{l,m}b_{l,m}^2\lambda^{l-1} \mu^m\right|
		&= \left|\sum_{n=N+1}^{\infty}\left(\sum_{l+m=n}la_{l,m}b_{l,m}^2\lambda^{l-1} \mu^m\right)\right|\\
		&\le\sum_{n=N+1}^{\infty}\frac{(4n)!}{(n-1)!\left(n!\right)^3}\left(4\log 4+3+3\log n\right)^2\left(|\lambda| + |\mu|\right)^{n-1}\\
		&\le\frac{(4(N+1))!}{N!\left((N+1)!\right)^3}\frac{\left(4\log 4 + 3 + 3\log (N+1)\right)^2(|\lambda|+|\mu|)^{N}}{1-\beta_{\revise{N+1}}^2\gamma_{\revise{N+1}}(|\lambda|+|\mu|)}=\delta_8(N,\lambda,\mu).
	\end{align}
	Additionally,
	\begin{align}
		\left|\sum_{l+m\ge N+1}ma_{l,m}b_{l,m}^2\lambda^{l} \mu^{m-1}\right| 
		&= \left|\sum_{n=N+1}^{\infty}\left(\sum_{l+m=n}ma_{l,m}b_{l,m}^2\lambda^{l} \mu^{m-1}\right)\right|\\
		&\le \sum_{n=N+1}^{\infty}\frac{(4n)!}{(n-1)!\left(n!\right)^3}\left(4\log 4+3+3\log n\right)^2\left(|\lambda| + |\mu|\right)^{n-1}\\
		&\le\frac{(4(N+1))!}{N!\left((N+1)!\right)^3}\frac{\left(4\log 4 + 3 + 3\log (N+1)\right)^2(|\lambda|+|\mu|)^{N}}{1-\beta_{\revise{N+1}}^2\gamma_{\revise{N+1}}(|\lambda|+|\mu|)}=\delta_8(N,\lambda,\mu).
	\end{align}

	Fourth, the argument similar to that used to derive \eqref{eq:estimate_delta4} yields that
	\begin{align}
		\left|\sum_{l+m\ge N+1}la_{l,m}c_{l,m}\lambda^{l-1} \mu^m\right|
		&= \left|\sum_{n=N+1}^{\infty}\left(\sum_{l+m=n}la_{l,m}c_{l,m}\lambda^{l-1} \mu^m\right)\right|\\
		&\le\frac83\pi^2\sum_{n=N+1}^{\infty}\frac{(4n)!}{(n-1)!\left(n!\right)^3}\left(|\lambda| + |\mu|\right)^{n-1}\\
		&\le\frac83\pi^2\frac{(4(N+1))!}{N!\left((N+1)!\right)^3}\frac{(|\lambda|+|\mu|)^{N}}{1-\gamma_{\revise{N+1}}(|\lambda|+|\mu|)}=\delta_9(N,\lambda,\mu)
	\end{align}
	and
	\begin{align}
		\left|\sum_{l+m\ge N+1}ma_{l,m}c_{l,m}\lambda^{l} \mu^{m-1}\right| 
		&= \left|\sum_{n=N+1}^{\infty}\left(\sum_{l+m=n}ma_{l,m}c_{l,m}\lambda^{l} \mu^{m-1}\right)\right|\\
		&\le\frac83\pi^2\sum_{n=N+1}^{\infty}\frac{(4n)!}{(n-1)!\left(n!\right)^3}\left(|\lambda| + |\mu|\right)^{n-1}\\
		&\le\frac83\pi^2\frac{(4(N+1))!}{N!\left((N+1)!\right)^3}\frac{(|\lambda|+|\mu|)^{N}}{1-\gamma_{\revise{N+1}}(|\lambda|+|\mu|)}=\delta_9(N,\lambda,\mu).
	\end{align}
	
	Fifth, we consider the estimate \eqref{eq:delta10}. For $\epsilon\in\{1/2,1\}$ we follow the similar estimate used in \eqref{eq:(v)_1}, \eqref{eq:(v)_2}, and \eqref{eq:estimate_delta5}
	\begin{align}
		&\left|\sum_{l+m\ge N+1}d_{l+\epsilon,m}(l+2m+\epsilon)\frac{\lambda^{l+2m+\epsilon-1}}{\mu^{l+m+\epsilon}}\right|\nonumber\\
		&= \left|\sum_{n=N+1}^\infty\left(\sum_{l+m=n}d_{l+\epsilon,m}(l+2m+\epsilon)\frac{\lambda^{l+2m+\epsilon-1}}{\mu^{l+m+\epsilon}}\right)\right|\nonumber\\
		&\le\left|\frac{\lambda}{\mu}\right|^\epsilon|\lambda|^{-1}\left|\sum_{n=N+1}^\infty\left(\sum_{l+m=n}d_{l+\epsilon,m}(l+2m+\epsilon)\left(\frac{\lambda}{\mu}\right)^{l+m}\lambda^m\right)\right|\nonumber\\
		&\le\left|\frac{\lambda}{\mu}\right|^\epsilon|\lambda|^{-1}\sum_{n=N+1}^\infty\left|\frac{\lambda}{\mu}\right|^{n}\left(1+|\lambda|\right)^n\sum_{l+m=n}\frac{(n!)^2\,l!(n+m+\epsilon)}{(2l)!\,(n+m)!}\\
		&\le\left|\frac{\lambda}{\mu}\right|^\epsilon|\lambda|^{-1}\sum_{n=N+1}^\infty\left|\frac{\lambda}{\mu}\right|^{n}\left(1+|\lambda|\right)^n\sum_{l+m=n}\frac{e^3}{2\sqrt{2}\,\pi}\frac{1}{2^{2l}}\frac{n^{2n+\frac{1}{2}}(n+m+\epsilon)}{l^l(n+m)^{n+m}}\\
		&\le\frac{e^3}{2\sqrt{2}\,\pi}\left|\frac{\lambda}{\mu}\right|^\epsilon|\lambda|^{-1}\sum_{n=N+1}^\infty\left|\frac{\lambda}{\mu}\right|^{n}\left(1+|\lambda|\right)^n\sqrt{n}\,(2n+\epsilon)\left(\frac{25}{64}\right)^n\\
		&\le \frac{e^3}{2\sqrt{2}\,\pi}\left|\frac{\lambda}{\mu}\right|^\epsilon|\lambda|^{-1}\sqrt{N+1}\,(2(N+1)+\epsilon)\left[\frac{25}{64}\left(\frac{|\lambda|+|\lambda|^2}{|\mu|}\right)\right]^{N+1}\\
		&\hphantom{\le}\quad\cdot\sum_{n'=0}^\infty\left(\frac{|\lambda|+|\lambda|^2}{|\mu|}\right)^{n'}\left[\frac{25}{64}\left(1+\frac{2}{2(\revise{N+1})+\epsilon}\right)\sqrt{1+\frac{1}{\revise{N+1}}}\,\right]^{n'}\\
		&\le  \frac{e^3}{2\sqrt{2}\,\pi}\left|\frac{\lambda}{\mu}\right|^\epsilon|\lambda|^{-1}\frac{\sqrt{N+1}\,(2(N+1)+\epsilon)\left[\frac{25}{64}\left(\frac{|\lambda|+|\lambda|^2}{|\mu|}\right)\right]^{N+1}}{1-\eta_{\revise{N+1}}^\epsilon\left(\frac{|\lambda|+|\lambda|^2}{|\mu|}\right)}=\delta_{10}^\epsilon(N,\lambda,\mu).
	\end{align}
	Finally, \revise{a similar argument yields the $\delta_{11}$ bound} 
	\begin{align}
		&\left|\sum_{l+m\ge N+1}d_{l+\epsilon,m}(l+m+\epsilon)\frac{\lambda^{l+2m+\epsilon}}{\mu^{l+m+\epsilon+1}}\right|\nonumber\\
		&= \left|\sum_{n=N+1}^\infty\left(\sum_{l+m=n}d_{l+\epsilon,m}(l+m+\epsilon)\frac{\lambda^{l+m+\epsilon}}{\mu^{l+m+\epsilon+1}}\right)\right|\nonumber\\
		&\le\left|\frac{\lambda}{\mu}\right|^\epsilon|\mu|^{-1}\left|\sum_{n=N+1}^\infty\left(\sum_{l+m=n}d_{l+\epsilon,m}(l+m+\epsilon)\left(\frac{\lambda}{\mu}\right)^{l+m}\lambda^m\right)\right|\nonumber\\
		&\le\left|\frac{\lambda}{\mu}\right|^\epsilon|\mu|^{-1}\sum_{n=N+1}^\infty\left|\frac{\lambda}{\mu}\right|^{n}\left(1+|\lambda|\right)^n\sum_{l+m=n}\frac{(n!)^2\,l!(n+\epsilon)}{(2l)!\,(n+m)!}\\
		&\le\left|\frac{\lambda}{\mu}\right|^\epsilon|\mu|^{-1}\sum_{n=N+1}^\infty\left|\frac{\lambda}{\mu}\right|^{n}\left(1+|\lambda|\right)^n\sum_{l+m=n}\frac{e^3}{2\sqrt{2}\,\pi}\frac{1}{2^{2l}}\frac{n^{2n+\frac{1}{2}}(n+\epsilon)}{l^l(n+m)^{n+m}}\\
		&\le\frac{e^3}{2\sqrt{2}\,\pi}\left|\frac{\lambda}{\mu}\right|^\epsilon|\mu|^{-1}\sum_{n=N+1}^\infty\left|\frac{\lambda}{\mu}\right|^{n}\left(1+|\lambda|\right)^n\sqrt{n}\,(n+\epsilon)\left(\frac{25}{64}\right)^n\\
		&\le \frac{e^3}{2\sqrt{2}\,\pi}\left|\frac{\lambda}{\mu}\right|^\epsilon|\mu|^{-1}\sqrt{N+1}\,(N+1+\epsilon)\left[\frac{25}{64}\left(\frac{|\lambda|+|\lambda|^2}{|\mu|}\right)\right]^{N+1}\\
		&\hphantom{\le}\quad\cdot\sum_{n'=0}^\infty\left(\frac{|\lambda|+|\lambda|^2}{|\mu|}\right)^{n'}\left[\frac{25}{64}\left(1+\frac{1}{\revise{N+1}+\epsilon}\right)\sqrt{1+\frac{1}{\revise{N+1}}}\,\right]^{n'}\\
		&\le  \frac{e^3}{2\sqrt{2}\,\pi}\left|\frac{\lambda}{\mu}\right|^\epsilon|\mu|^{-1}\frac{\sqrt{N+1}\,(N+1+\epsilon)\left[\frac{25}{64}\left(\frac{|\lambda|+|\lambda|^2}{|\mu|}\right)\right]^{N+1}}{1-\theta_{\revise{N+1}}^\epsilon\left(\frac{|\lambda|+|\lambda|^2}{|\mu|}\right)}=\delta_{11}^\epsilon(N,\lambda,\mu).
	\end{align}
	This completes the proof.
\end{proof}


The final proof concerns the second derivatives of the particular solutions. While slightly more \revise{complicated}, the main procedure for deriving the estimates closely follows the approach used in the preceding two proofs.

\begin{proof}[Proof of lemma \ref{lem:TruncationError3}]
	First, let us consider the estimate \eqref{eq:delta12}. It follows that
	\begin{align}
		\left|\sum_{l+m\ge N+1}l(l-1)a_{l,m}\lambda^{l-2} \mu^m\right|
		&\le \sum_{n=N+1}^{\infty}\left(\sum_{l+m=n}l(l-1)\frac{(2l+4m)!}{(l+m)!\,l!\left(m!\right)^3}|\lambda|^{l-2} |\mu|^m\right)\\
		&= \sum_{n=N+1}^{\infty}\left(\sum_{l+m=n\atop l\ge2}\frac{(2l+4m)!}{(l+m-2)!(l+m)!\left(m!\right)^2}\frac{(l+m-2)!}{(l-2)!\,m!}|\lambda|^{l-2} |\mu|^m\right)\\
		&\le \sum_{n=N+1}^{\infty}\frac{1}{\left(n-2\right)!n!}\max_{0\le m\le n}\left(\frac{(2n+2m)!}{\left(m!\right)^2}\right)\left(|\lambda| + |\mu|\right)^{n-2}\\
		&= \sum_{n=N+1}^{\infty}\frac{(4n)!}{(n-2)!\left(n!\right)^3}\left(|\lambda| + |\mu|\right)^{n-2}\\
		&\le\frac{(4(N+1))!}{(N-1)!\left((N+1)!\right)^3}\left(|\lambda| + |\mu|\right)^{N-1}\sum_{n'=0}^{\infty}\left[\iota_{\revise{N+1}}(|\lambda|+|\mu|)\right]^{n'}\\
		&\le\frac{(4(N+1))!}{(N-1)!\left((N+1)!\right)^3}\frac{(|\lambda|+|\mu|)^{N-1}}{1-\iota_{\revise{N+1}}(|\lambda|+|\mu|)}=\delta_{12}(N,\lambda,\mu),
	\end{align}
	where we also used the fact that, \revise{for $n=N+1, N+2, \dots$,} the infinite sum is bounded by a geometric series with a common ratio of
	\begin{align*}
		\frac{(4n+4)!}{(n-1)!\left((n+1)!\right)^3}\cdot\frac{(n-2)!\left(n!\right)^3}{(4n)!}&=\frac{(4n+4)(4n+3)(4n+2)(4n+1)}{(n-1)(n+1)^3}\\
		&=4\left(4-\frac{1}{n+1}\right)\left(4-\frac{2}{n+1}\right)\left(4+\frac{5}{n-1}\right)\\
		&\le 64\left(4+\frac5{\revise{N}}\right)=\iota_{\revise{N+1}}.
	\end{align*}
	It also follows that
	\begin{align}
		\left|\sum_{l+m\ge N+1}lma_{l,m}\lambda^{l-1} \mu^{m-1}\right|
		&\le \sum_{n=N+1}^{\infty}\left(\sum_{l+m=n}lm\frac{(2l+4m)!}{(l+m)!\,l!\left(m!\right)^3}|\lambda|^{l-1} |\mu|^{m-1}\right)\\
		&= \sum_{n=N+1}^{\infty}\left(\sum_{l+m=n\atop l\ge1,~m\ge1}\frac{(2l+4m)!}{(l+m-2)!(l+m)!\left(m!\right)^2}\frac{(l+m-2)!}{(l-1)!\,(m-1)!}|\lambda|^{l-1} |\mu|^{m-1}\right)\\
		&\le \sum_{n=N+1}^{\infty}\frac{1}{\left(n-2\right)!n!}\max_{0\le m\le n}\left(\frac{(2n+2m)!}{\left(m!\right)^2}\right)\left(|\lambda| + |\mu|\right)^{n-2}\\
		&\le\frac{(4(N+1))!}{(N-1)!\left((N+1)!\right)^3}\frac{(|\lambda|+|\mu|)^{N-1}}{1-\iota_{\revise{N+1}}(|\lambda|+|\mu|)}=\delta_{12}(N,\lambda,\mu).
	\end{align}
	Additionally, we have
	\begin{align}
		\left|\sum_{l+m\ge N+1}m(m-1)a_{l,m}\lambda^{l} \mu^{m-2}\right|
		&\le \sum_{n=N+1}^{\infty}\left(\sum_{l+m=n}m(m-1)\frac{(2l+4m)!}{(l+m)!\,l!\left(m!\right)^3}|\lambda|^{l} |\mu|^{m-2}\right)\\
		&= \sum_{n=N+1}^{\infty}\left(\sum_{l+m=n\atop m\ge 2}\frac{(2l+4m)!}{(l+m-2)!(l+m)!\left(m!\right)^2}\frac{(l+m-2)!}{l!\,(m-2)!}|\lambda|^{l} |\mu|^{m-2}\right)\\
		&\le \sum_{n=N+1}^{\infty}\frac{1}{\left(n-2\right)!n!}\max_{0\le m\le n}\left(\frac{(2n+2m)!}{\left(m!\right)^2}\right)\left(|\lambda| + |\mu|\right)^{n-2}\\
		&\le \frac{(4(N+1))!}{(N-1)!\left((N+1)!\right)^3}\frac{(|\lambda|+|\mu|)^{N-1}}{1-\iota_{\revise{N+1}}(|\lambda|+|\mu|)}=\delta_{12}(N,\lambda,\mu).
	\end{align}
	These \revise{provide} the estimate \eqref{eq:delta12}. 
	
	Second, we consider the $\delta_{13}$ bound in the estimate \eqref{eq:delta13}. \revise{A similar procedure} to obtain \eqref{eq:estimate_delta2} in the proof of Lemma \ref{lem:TruncationError1} yields
	\begin{align}
		\left|\sum_{l+m\ge N+1}l(l-1)a_{l,m}b_{l,m}\lambda^{l-2} \mu^m\right|
		&\le \sum_{n=N+1}^{\infty}\left(\sum_{l+m=n}l(l-1)\frac{(2l+4m)!}{(l+m)!\,l!\left(m!\right)^3}b_{l,m}|\lambda|^{l-2} |\mu|^m\right)\\
		&\le \sum_{n=N+1}^{\infty}\frac{(4n)!}{(n-2)!\left(n!\right)^3}\left(4\log 4+3+3\log n\right)\left(|\lambda| + |\mu|\right)^{n-2}\\
		&\le\frac{(4(N+1))!}{(N-1)!\left((N+1)!\right)^3}\frac{\left(4\log 4 + 3 + 3\log (N+1)\right)(|\lambda|+|\mu|)^{N-1}}{1-\beta_{\revise{N+1}}\iota_{\revise{N+1}}(|\lambda|+|\mu|)}=\delta_{13}(N,\lambda,\mu),
	\end{align}
	where $\beta_N$ is that defined in \eqref{eq:betan} with $n=N$.
	It also follows that
	\begin{align}
		\left|\sum_{l+m\ge N+1}lma_{l,m}b_{l,m}\lambda^{l-1} \mu^{m-1}\right|
		&\le \sum_{n=N+1}^{\infty}\left(\sum_{l+m=n}lm\frac{(2l+4m)!}{(l+m)!\,l!\left(m!\right)^3}b_{l,m}|\lambda|^{l-1} |\mu|^{m-1}\right)\\
		&\le \sum_{n=N+1}^{\infty}\frac{(4n)!}{(n-2)!\left(n!\right)^3}\left(4\log 4+3+3\log n\right)\left(|\lambda| + |\mu|\right)^{n-2}\\
		&\le \frac{(4(N+1))!}{(N-1)!\left((N+1)!\right)^3}\frac{\left(4\log 4 + 3 + 3\log (N+1)\right)(|\lambda|+|\mu|)^{N-1}}{1-\beta_{\revise{N+1}}\iota_{\revise{N+1}}(|\lambda|+|\mu|)}=\delta_{13}(N,\lambda,\mu).
	\end{align}
	\begin{align}
		\left|\sum_{l+m\ge N+1}m(m-1)a_{l,m}b_{l,m}\lambda^{l} \mu^{m-2}\right|
		&\le \sum_{n=N+1}^{\infty}\left(\sum_{l+m=n}m(m-1)\frac{(2l+4m)!}{(l+m)!\,l!\left(m!\right)^3}b_{l,m}|\lambda|^{l} |\mu|^{m-2}\right)\\
		&\le \sum_{n=N+1}^{\infty}\frac{(4n)!}{(n-2)!\left(n!\right)^3}\left(4\log 4+3+3\log n\right)\left(|\lambda| + |\mu|\right)^{n-2}\\
		&\le \frac{(4(N+1))!}{(N-1)!\left((N+1)!\right)^3}\frac{\left(4\log 4 + 3 + 3\log (N+1)\right)(|\lambda|+|\mu|)^{N-1}}{1-\beta_{\revise{N+1}}\iota_{\revise{N+1}}(|\lambda|+|\mu|)}=\delta_{13}(N,\lambda,\mu).
	\end{align}

	Third, the $\delta_{14}$ bound in the estimate \eqref{eq:delta14} is derived using an argument analogous to that employed for obtaining \eqref{eq:estimate_delta3}.
	\begin{align}
		\left|\sum_{l+m\ge N+1}l(l-1)a_{l,m}b_{l,m}^2\lambda^{l-2} \mu^m\right|
		&\le \sum_{n=N+1}^{\infty}\frac{(4n)!}{(n-2)!\left(n!\right)^3}\left(4\log 4+3+3\log n\right)^2\left(|\lambda| + |\mu|\right)^{n-2}\\
		&\le \frac{(4(N+1))!}{(N-1)!\left((N+1)!\right)^3}\frac{\left(4\log 4 + 3 + 3\log (N+1)\right)^2(|\lambda|+|\mu|)^{N-1}}{1-\beta_{\revise{N+1}}^2\iota_{\revise{N+1}}(|\lambda|+|\mu|)}=\delta_{14}(N,\lambda,\mu).
	\end{align}
	\begin{align}
		\left|\sum_{l+m\ge N+1}lma_{l,m}b_{l,m}^2\lambda^{l-1} \mu^{m-1}\right|
		&\le \sum_{n=N+1}^{\infty}\frac{(4n)!}{(n-2)!\left(n!\right)^3}\left(4\log 4+3+3\log n\right)^2\left(|\lambda| + |\mu|\right)^{n-2}\\
		&\le \frac{(4(N+1))!}{(N-1)!\left((N+1)!\right)^3}\frac{\left(4\log 4 + 3 + 3\log (N+1)\right)^2(|\lambda|+|\mu|)^{N-1}}{1-\beta_{\revise{N+1}}^2\iota_{\revise{N+1}}(|\lambda|+|\mu|)}=\delta_{14}(N,\lambda,\mu).
	\end{align}
	\begin{align}
		\left|\sum_{l+m\ge N+1}m(m-1)a_{l,m}b_{l,m}^2\lambda^{l} \mu^{m-2}\right|
		&\le \sum_{n=N+1}^{\infty}\frac{(4n)!}{(n-2)!\left(n!\right)^3}\left(4\log 4+3+3\log n\right)^2\left(|\lambda| + |\mu|\right)^{n-2}\\
		&\le \frac{(4(N+1))!}{(N-1)!\left((N+1)!\right)^3}\frac{\left(4\log 4 + 3 + 3\log (N+1)\right)^2(|\lambda|+|\mu|)^{N-1}}{1-\beta_{\revise{N+1}}^2\iota_{\revise{N+1}}(|\lambda|+|\mu|)}=\delta_{14}(N,\lambda,\mu).
	\end{align}

	Fourth, we consider the estimate \eqref{eq:delta15} related to the $\delta_{15}$ bound. We use \revise{the same argument that was used to derive} \eqref{eq:estimate_delta4} in the proof of Lemma \ref{lem:TruncationError1}. That is,
	\begin{align}
		\left|\sum_{l+m\ge N+1}l(l-1)a_{l,m}c_{l,m}\lambda^{l-2} \mu^m\right|
		&\le\frac83\pi^2\sum_{n=N+1}^{\infty}\frac{(4n)!}{(n-2)!\left(n!\right)^3}\left(|\lambda| + |\mu|\right)^{n-2}\\
		&\le\frac83\pi^2\frac{(4(N+1))!}{(N-1)!\left((N+1)!\right)^3}\frac{(|\lambda|+|\mu|)^{N-1}}{1-\iota_{\revise{N+1}}(|\lambda|+|\mu|)}=\delta_{15}(N,\lambda,\mu).
	\end{align}
	\begin{align}
		\left|\sum_{l+m\ge N+1}lma_{l,m}c_{l,m}\lambda^{l-1} \mu^{m-1}\right|
		&\le \frac83\pi^2\sum_{n=N+1}^{\infty}\frac{(4n)!}{(n-2)!\left(n!\right)^3}\left(|\lambda| + |\mu|\right)^{n-2}\\
		&\le \frac83\pi^2\frac{(4(N+1))!}{(N-1)!\left((N+1)!\right)^3}\frac{(|\lambda|+|\mu|)^{N-1}}{1-\iota_{\revise{N+1}}(|\lambda|+|\mu|)}=\delta_{15}(N,\lambda,\mu).
	\end{align}
	\begin{align}
		\left|\sum_{l+m\ge N+1}m(m-1)a_{l,m}c_{l,m}\lambda^{l} \mu^{m-2}\right|
		&\le \frac83\pi^2\sum_{n=N+1}^{\infty}\frac{(4n)!}{(n-2)!\left(n!\right)^3}\left(|\lambda| + |\mu|\right)^{n-2}\\
		&\le \frac83\pi^2\frac{(4(N+1))!}{(N-1)!\left((N+1)!\right)^3}\frac{(|\lambda|+|\mu|)^{N-1}}{1-\iota_{\revise{N+1}}(|\lambda|+|\mu|)}=\delta_{15}(N,\lambda,\mu).
	\end{align}

	Fifth, let us consider the estimate \eqref{eq:delta16}. For $\epsilon\in\{1/2,1\}$, we also follow the same argument used in \eqref{eq:(v)_1}, \eqref{eq:(v)_2}, and \eqref{eq:estimate_delta5}. We have
	\begin{align}
		&\left|\sum_{l+m\ge N+1}d_{l+\epsilon,m}(l+2m+\epsilon)(l+2m+\epsilon-1)\frac{\lambda^{l+2m+\epsilon-2}}{\mu^{l+m+\epsilon}}\right|\nonumber\\
		&= \left|\sum_{n=N+1}^\infty\left(\sum_{l+m=n}d_{l+\epsilon,m}(l+2m+\epsilon)(l+2m+\epsilon-1)\frac{\lambda^{l+2m+\epsilon-2}}{\mu^{l+m+\epsilon}}\right)\right|\nonumber\\
		&\le\left|\frac{\lambda}{\mu}\right|^\epsilon|\lambda|^{-2}\left|\sum_{n=N+1}^\infty\left(\sum_{l+m=n}d_{l+\epsilon,m}(l+2m+\epsilon)(l+2m+\epsilon-1)\left(\frac{\lambda}{\mu}\right)^{l+m}\lambda^m\right)\right|\nonumber\\
		&\le\left|\frac{\lambda}{\mu}\right|^\epsilon|\lambda|^{-2}\sum_{n=N+1}^\infty\left|\frac{\lambda}{\mu}\right|^{n}\left(1+|\lambda|\right)^n\sum_{l+m=n}\frac{(n!)^2\,l!(n+m+\epsilon)(n+m+\epsilon-1)}{(2l)!\,(n+m)!}\\
		&\le\left|\frac{\lambda}{\mu}\right|^\epsilon|\lambda|^{-2}\sum_{n=N+1}^\infty\left|\frac{\lambda}{\mu}\right|^{n}\left(1+|\lambda|\right)^n\sum_{l+m=n}\frac{e^3}{2\sqrt{2}\,\pi}\frac{1}{2^{2l}}\frac{n^{2n+\frac{1}{2}}(n+m+\epsilon)(n+m+\epsilon-1)}{l^l(n+m)^{n+m}}\\
		&\le\frac{e^3}{2\sqrt{2}\,\pi}\left|\frac{\lambda}{\mu}\right|^\epsilon|\lambda|^{-2}\sum_{n=N+1}^\infty\left|\frac{\lambda}{\mu}\right|^{n}\left(1+|\lambda|\right)^n\sqrt{n}\,(2n+\epsilon)(2n+\epsilon-1)\left(\frac{25}{64}\right)^n\\
		&\le \frac{e^3}{2\sqrt{2}\,\pi}\left|\frac{\lambda}{\mu}\right|^\epsilon|\lambda|^{-2}\sqrt{N+1}\,\left(2(N+1)+\epsilon\right)\left(2(N+1)+\epsilon-1\right)\left[\frac{25}{64}\left(\frac{|\lambda|+|\lambda|^2}{|\mu|}\right)\right]^{N+1}\\
		&\hphantom{\le}\quad\cdot\sum_{n'=0}^\infty\left(\frac{|\lambda|+|\lambda|^2}{|\mu|}\right)^{n'}\left[\frac{25}{64}\left(1+\frac{2}{2(\revise{N+1})+\epsilon}\right)\left(1+\frac{2}{2(\revise{N+1})+\epsilon-1}\right)\sqrt{1+\frac{1}{\revise{N+1}}}\,\right]^{n'}\\
		&\le  \frac{e^3}{2\sqrt{2}\,\pi}\left|\frac{\lambda}{\mu}\right|^\epsilon|\lambda|^{-2}\frac{\sqrt{N+1}\,\left(2(N+1)+\epsilon\right)\left(2(N+1)+\epsilon-1\right)\left[\frac{25}{64}\left(\frac{|\lambda|+|\lambda|^2}{|\mu|}\right)\right]^{N+1}}{1-\nu_{\revise{N+1}}^\epsilon\left(\frac{|\lambda|+|\lambda|^2}{|\mu|}\right)}=\delta_{16}^\epsilon(N,\lambda,\mu),
	\end{align}
	where $\nu_N^\epsilon$ is defined in \eqref{eq:nu_N}.
	
	Similarly, we have the estimate \eqref{eq:delta17} as follows.
	\begin{align}
		&\left|\sum_{l+m\ge N+1}d_{l+\epsilon,m}(l+2m+\epsilon)(l+m+\epsilon)\frac{\lambda^{l+2m+\epsilon-1}}{\mu^{l+m+\epsilon+1}}\right|\nonumber\\
		&\le\left|\frac{\lambda}{\mu}\right|^\epsilon|\lambda\mu|^{-1}\left|\sum_{n=N+1}^\infty\left(\sum_{l+m=n}d_{l+\epsilon,m}(l+2m+\epsilon)(l+m+\epsilon)\left(\frac{\lambda}{\mu}\right)^{l+m}\lambda^m\right)\right|\nonumber\\
		&\le\frac{e^3}{2\sqrt{2}\,\pi}\left|\frac{\lambda}{\mu}\right|^\epsilon|\lambda\mu|^{-1}\sum_{n=N+1}^\infty\left|\frac{\lambda}{\mu}\right|^{n}\left(1+|\lambda|\right)^n\sqrt{n}\,(2n+\epsilon)(n+\epsilon)\left(\frac{25}{64}\right)^n\\
		&\le \frac{e^3}{2\sqrt{2}\,\pi}\left|\frac{\lambda}{\mu}\right|^\epsilon|\lambda\mu|^{-1}\sqrt{N+1}\,(2(N+1)+\epsilon)(N+1+\epsilon)\left[\frac{25}{64}\left(\frac{|\lambda|+|\lambda|^2}{|\mu|}\right)\right]^{N+1}\\
		&\hphantom{\le}\quad\cdot\sum_{n'=0}^\infty\left(\frac{|\lambda|+|\lambda|^2}{|\mu|}\right)^{n'}\left[\frac{25}{64}\left(1+\frac{2}{2(\revise{N+1})+\epsilon}\right)\left(1+\frac{1}{\revise{N+1}+\epsilon}\right)\sqrt{1+\frac{1}{\revise{N+1}}}\,\right]^{n'}\\
		&\le  \frac{e^3}{2\sqrt{2}\,\pi}\left|\frac{\lambda}{\mu}\right|^\epsilon|\lambda\mu|^{-1}\frac{\sqrt{N+1}\,(2(N+1)+\epsilon)(N+1+\epsilon)\left[\frac{25}{64}\left(\frac{|\lambda|+|\lambda|^2}{|\mu|}\right)\right]^{N+1}}{1-\xi_{\revise{N+1}}^\epsilon\left(\frac{|\lambda|+|\lambda|^2}{|\mu|}\right)}=\delta_{17}^\epsilon(N,\lambda,\mu),
	\end{align}
	where $\xi_N^\epsilon$ is defined in \eqref{eq:xi_N_sigma_N}.

	Finally, the estimate \eqref{eq:delta18} is obtained by
	\begin{align}
		&\left|\sum_{l+m\ge N+1}d_{l+\epsilon,m}(l+m+\epsilon)(l+m+\epsilon+1)\frac{\lambda^{l+2m+\epsilon}}{\mu^{l+m+\epsilon+2}}\right|\nonumber\\
		&\le\left|\frac{\lambda}{\mu}\right|^\epsilon|\mu|^{-2}\left|\sum_{n=N+1}^\infty\left(\sum_{l+m=n}d_{l+\epsilon,m}(l+m+\epsilon)(l+m+\epsilon+1)\left(\frac{\lambda}{\mu}\right)^{l+m}\lambda^m\right)\right|\nonumber\\
		&\le\frac{e^3}{2\sqrt{2}\,\pi}\left|\frac{\lambda}{\mu}\right|^\epsilon|\mu|^{-2}\sum_{n=N+1}^\infty\left|\frac{\lambda}{\mu}\right|^{n}\left(1+|\lambda|\right)^n\sqrt{n}\,(n+\epsilon)(n+\epsilon+1)\left(\frac{25}{64}\right)^n\\
		&\le \frac{e^3}{2\sqrt{2}\,\pi}\left|\frac{\lambda}{\mu}\right|^\epsilon|\mu|^{-2}\sqrt{N+1}\,(N+1+\epsilon)(N+2+\epsilon)\left[\frac{25}{64}\left(\frac{|\lambda|+|\lambda|^2}{|\mu|}\right)\right]^{N+1}\\
		&\hphantom{\le}\quad\cdot\sum_{n'=0}^\infty\left(\frac{|\lambda|+|\lambda|^2}{|\mu|}\right)^{n'}\left[\frac{25}{64}\left(1+\frac{1}{\revise{N+1}+\epsilon}\right)\left(1+\frac{1}{\revise{N+1}+\epsilon+1}\right)\sqrt{1+\frac{1}{\revise{N+1}}}\,\right]^{n'}\\
		&\le  \frac{e^3}{2\sqrt{2}\,\pi}\left|\frac{\lambda}{\mu}\right|^\epsilon|\mu|^{-2}\frac{\sqrt{N+1}\,(N+1+\epsilon)(N+2+\epsilon)\left[\frac{25}{64}\left(\frac{|\lambda|+|\lambda|^2}{|\mu|}\right)\right]^{N+1}}{1-\sigma_{\revise{N+1}}^\epsilon\left(\frac{|\lambda|+|\lambda|^2}{|\mu|}\right)}=\delta_{18}^\epsilon(N,\lambda,\mu),
	\end{align}
	where $\sigma_N^\epsilon$ is defined in \eqref{eq:xi_N_sigma_N}.
\end{proof}

\paragraph{Acknowledgement}
The authors would like to express their gratitude to AT's former student, Naoya Inoue, for his contributions during the initial stages of this work. AT is partially supported by the Top Runners in Strategy of Transborder Advanced Researches (TRiSTAR) program conducted as the Strategic Professional Development Program for Young Researchers by the MEXT, and by JSPS KAKENHI Grant Numbers JP22K03411, JP23K20813, and JP24K00538.

\paragraph{Declaration of generative AI and AI-assisted technologies in the writing process}
The authors used ChatGPT during the preparation of this paper to refine the English phrasing and improve the clarity of the content. After its use, the authors thoroughly reviewed and revised the text to ensure accuracy and take full responsibility for the final content of the publication.

\bibliographystyle{unsrt}
\bibliography{papers}
\end{document}